\documentclass[final]{siamart1116}

\usepackage{amsfonts}
\usepackage{graphicx,epstopdf}
\usepackage{mathrsfs}
\usepackage{bm}
\usepackage{multirow}
\usepackage{color}
\usepackage{hyperref} 
\hypersetup{colorlinks=false} %

\newsiamthm{claim}{Claim}
\newsiamremark{remark}{Remark}
\newsiamremark{expl}{Example}
\numberwithin{theorem}{section}

\allowdisplaybreaks

\title{
  Provably Positive 
  Discontinuous Galerkin Methods     
  for Multidimensional Ideal  
  Magnetohydrodynamics\thanks{Submitted to the editors on January 31, 2018.}
}

\author{
  Kailiang Wu\thanks{Department of Mathematics, The Ohio State University, 
  	Columbus, OH 43210, USA  (\email{wu.3423@osu.edu}).}
  \and
  Chi-Wang Shu\thanks{Division of Applied Mathematics, Brown University, Providence, RI 02912, USA (\email{shu@dam.brown.edu}). Research is supported in part by ARO grant W911NF-15-1-0226 and NSF grant DMS-1719410. }  
}

\headers{Provably Positive DG methods for MHD}
{Kailiang Wu and Chi-Wang Shu}

\ifpdf
\hypersetup{ pdftitle={Provably Positive DG methods for MHD} }
\fi

\begin{document}
\maketitle

%% ------------------------------------------------------------------
%% ABSTRACT
%% ------------------------------------------------------------------

\begin{abstract}
The density and pressure are positive physical quantities in 	
magnetohydrodynamics (MHD). 
Design of {\em provably} positivity-preserving (PP) numerical schemes 
for ideal compressible MHD is highly desirable, but remains 
a challenge especially in the multidimensional cases. 	
In this paper, we first develop 
%arbitrarily
uniformly 
high-order discontinuous Galerkin (DG) schemes which provably 
preserve the positivity of density and pressure for multidimensional ideal MHD. 
The schemes are constructed by using the locally divergence-free DG schemes 
for the symmetrizable ideal MHD equations as the base schemes, 
a PP limiter to enforce the positivity of the
DG solutions, and the 
strong stability preserving methods for 
time discretization. 
The significant innovation is that we discover and 
 rigorously prove the 
PP property of the proposed DG schemes by using a novel equivalent form of the admissible state set and very technical estimates. 
Several two-dimensional numerical examples further confirm 
the PP property, and demonstrate the accuracy, effectiveness and robustness 
of the proposed PP methods.
\end{abstract}



\section{Introduction} 
\label{sec:intro}
%\section{Governing Equations}
In this paper, we would like to 
develop the high-order numerical methods which {\em provably} preserve the positivity of 
density, pressure and internal energy %under all circumstances 
for the ideal magnetohydrodynamics (MHD).  
In the laboratory frame, 
the equations governing 
the $d$-dimensional ideal compressible MHD flows can be
written as a set of nonlinear hyperbolic conservation laws
\begin{equation}\label{eq:MHD}
\frac{{\partial {\bf U}}}{{\partial t}} + \sum\limits_{i = 1}^d {\frac{{\partial { {\bf F}_i}({\bf U})}}{{\partial x_i}}}  = {\bf 0},
\end{equation}
where $d=1, 2$ or $3$. In Eq. \eqref{eq:MHD}, the conservative vector ${\bf U} = ( \rho,\rho {\bf v},{\bf B},E )^{\top}$,
and ${\bf F}_i({\bf U})$ is the flux in the $x_i$-direction, $i=1,\cdots,d$, defined by
\begin{align*}
	%&
	%{\bf U} = \big( \rho,\rho {\bf v},{\bf B},E \big)^{\top}, \\
	%&
	{\bf F}_i({\bf U}) = \Big( \rho v_i,~\rho v_i {\bf v}  -  B_i {\bf B}   + p_{tot}  {\bf e}_i,~v_i {\bf B} - B_i {\bf v},~v_i(E+p_{tot} ) -  B_i({\bf v}\cdot {\bf B})
	\Big)^{\top}.
\end{align*}
Here $\rho$ denotes the density,  ${\bf v}=(v_1,v_2,v_3)$ is 
the fluid velocity, ${\bf B}=(B_1,B_2,B_3)$ is the magnetic field,
$p_{\rm tot}$ denotes the total pressure consisting of 
the gas pressure $p$ and  magnetic pressure $p_m=\frac{|{\bf B}|^2}2$, 
the vector ${\bf e}_i$ is the $i$-th row of the unit matrix of size 3,
$E=\rho e + \frac12 \left( \rho |{\bf v}|^2 + |{\bf B}|^2 \right)$ denotes the total energy 
consisting of kinetic, thermal and magnetic energies, and $e$ is the specific internal energy.
An additional equation for the thermodynamical quantities---the so-called equation of state (EOS)---is required to
close the system \eqref{eq:MHD}. For ideal gases the EOS is given by 
\begin{equation}\label{eq:EOS}
p = (\gamma-1) \rho e,
\end{equation}
where $\gamma>1$ is the adiabatic index. Although the EOS  \eqref{eq:EOS} is widely used, 
there are scenarios where it is more suitable to use other EOSs. A general EOS may be
expressed as
\begin{equation}\label{eq:gEOS}
p = p(\rho,e).
\end{equation}
We assume \eqref{eq:gEOS} satisfy 
\begin{equation}\label{eq:assumpEOS}
\mbox{if}\quad \rho \ge 0,\quad \mbox{then}\quad e>0~\Leftrightarrow~p(\rho,e) > 0.
\end{equation}
Such a condition is reasonable and holds for the ideal EOS \eqref{eq:EOS}. It was also used 
in \cite{Zhang2011} to develop positive high-order schemes for the Euler equations with a general EOS.

The exact solution of the $d$-dimensional MHD equations \eqref{eq:MHD}  must also satisfy the following divergence-free condition on the magnetic field
\begin{equation}\label{eq:2D:BxBy0}
\nabla \cdot {\bf B} := \sum\limits_{i = 1}^d \frac{\partial B_i } {\partial x_i} =0,
\end{equation}
if the initial magnetic field is divergence-free. 
Most of numerical methods for
the multidimensional MHD equations, however, lead to a nonzero divergence of numerical magnetic field due to truncation errors, even if the initial condition satisfies \eqref{eq:2D:BxBy0}. 
Existing evidences indicate that negligence in dealing with the divergence-free condition \eqref{eq:2D:BxBy0} can cause  nonphysical features or numerical instabilities in computed solutions; see, for example,  \cite{Brackbill1980,Evans1988,BalsaraSpicer1999,Toth2000,Dedner2002,Li2005}.
Up to now, a number of numerical techniques have been developed to enforce the
divergence-free condition or reduce the divergence-error in discrete sense.  
They include but are not limited to: %Powell1999
the hyperbolic divergence cleaning methods \cite{Dedner2002},
the projection method \cite{Brackbill1980}, 
the locally divergence-free methods (cf.~\cite{Li2005,Yakovlev2013}), 
the constrained transport method \cite{Evans1988} and its  variants 
(e.g., \cite{Ryu1998,BalsaraSpicer1999,Londrillo2000,Balsara2004,Torrilhon2005,Rossmanith2006,Li2011,Li2012,Christlieb2014}),  
%Artebrant2008 Li2008 ,Balsara2009 Londrillo2004 ,Torrilhon2004
and the eight-wave methods (e.g., \cite{Powell1994,Powell1995,Chandrashekar2016,LiuShuZhang2017}). 
The eight-wave method 
was first proposed by Powell \cite{Powell1994,Powell1995}, 
 based on a
proper discretization of the Godunov form \cite{Godunov1972} of ideal MHD equations %which reads
\begin{equation}\label{eq:MHD:GP}
\frac{{\partial {\bf U}}}{{\partial t}} + \sum\limits_{i = 1}^d {\frac{{\partial { {\bf F}_i}(
			{\bf U})}}{{\partial x_i}}} 
= -\big(\nabla \cdot {\bf B} \big)~{\bf S} ( {\bf U} ),
\end{equation}
where ${\bf S} = \big( 0,~{\bf B},~{\bf v},~{\bf v} \cdot {\bf B} \big)^\top.$ 
In the literature, \eqref{eq:MHD:GP} is sometimes also called Powell's system. 
The right-hand side term of \eqref{eq:MHD:GP}, 
abbreviated as ``GP source term'' in the following, is proportional 
to $\nabla \cdot {\bf B}$ and thus identically zero if $\nabla \cdot {\bf B}=0$. 
This means \eqref{eq:MHD:GP} and \eqref{eq:MHD} are equivalent 
under the condition \eqref{eq:2D:BxBy0}. However, 
for the following reasons it is sometimes advantageous to add  
the GP source term in the equations.  
First, Godunov \cite{Godunov1972} pointed out that \eqref{eq:MHD:GP} is the unique form of MHD equations which is symmetrizable. 
The symmetrized form  
%the equations 
is useful for designing entropy stable schemes \cite{Barth1998,Barth2006,Chandrashekar2016,LiuShuZhang2017}. 
Powell \cite{Powell1994} noticed that the system \eqref{eq:MHD} 
is incompletely hyperbolic and should add the source term to recover the missing eigenvector. 
Besides, when $\nabla \cdot {\bf B}\neq0$, 
the system \eqref{eq:MHD} is not Galilean invariant, 
while 
the GP source term 
renders the system \eqref{eq:MHD:GP} Galiean invariant (cf. \cite{Dellar2001}). 
 In most of numerical schemes the condition \eqref{eq:2D:BxBy0} 
is only satisfied up to a discretization error. 
As demonstrated by Powell \cite{Powell1995}, the 
inclusion of GP source term assures that those small divergence-errors are consistently accounted %for 
in a numerically stable way and do not lead to accumulation of inaccuracies.
This makes the eight-wave method stable
to control the divergence-error, although some drawbacks \cite{Toth2000} may be caused
due to the loss of conservativeness.

Besides controlling the divergence-error, another 
numerical challenge for MHD is to preserve the positivity of density and pressure. 
In physics, these two quantities are always nonnegative. Numerically their
positivity is very important, but not always satisfied by the  numerical solutions.
In fact, as soon as negative density or pressure is obtained in the MHD simulations,
the discrete problem becomes ill-posed, %due to the loss of hyperbolicity, 
causing the breakdown of codes.
However, most existing MHD methods are generally not positivity-preserving (PP), and thus may suffer from a large risk of failure %when simulating 
in solving 
MHD problems with low density, low pressure, low plasma-beta or strong discontinuity.
A few efforts were made to reduce this risk. By switching the Riemann solvers for different wave situations, Balsara and Spicer \cite{BalsaraSpicer1999a} proposed a strategy to maintain the positive pressure. 
In \cite{Janhunen2000}, Janhunen 
noticed the difficulty of developing PP schemes based on the conservative MHD system \eqref{eq:MHD}, so he 
proposed a modified MHD system, which is similar to the Godunov form \eqref{eq:MHD:GP} but includes only the source term in 
the induction equation. Based on this modified system, 
Janhunen \cite{Janhunen2000} designed an approximate 1D Riemann solver, and numerically demonstrated its PP property. 
Bouchut et al.\ \cite{Bouchut2007} derived several approximate Riemann solvers 
for 1D ideal MHD,  
with sufficient conditions for those solvers to satisfy the discrete entropy inequalities and PP property. 
Those sufficient conditions are satisfied by explicit wave speed estimates in \cite{Bouchut2010}, where the Riemann solvers were implemented and multidimensional extension was discussed 
with the aid of Janhunen's modified system.
Waagan \cite{Waagan2009} developed a positive second-order scheme 
for the ideal MHD based on the approximate Riemann solver of 
\cite{Bouchut2007,Bouchut2010} and a new linear reconstruction. 
The robustness of that scheme 
was further demonstrated in \cite{waagan2011robust}  by extensive  benchmark tests and comparisons.  
Recent years have witnessed significant progresses  in developing high-order bound-preserving methods for hyperbolic systems (see, e.g.,    \cite{zhang2010,zhang2010b,Xing2010,zhang2011b,Hu2013,Xu2014,Liang2014,Xu2017,Wu2017,ZHANG2017301}) including the ideal MHD system \cite{Balsara2012,cheng,Christlieb,Christlieb2016} and the relativistic MHD system \cite{WuTangM3AS}. 
Two PP limiting techniques were 
developed in \cite{Balsara2012,cheng} for the finite volume or discontinuous Galerkin (DG) methods for \eqref{eq:MHD} to enforce the admissibility\footnote{In this paper, the admissibility of a solution or state $\bf U$ means  
	that the density $\rho$ and pressure $p$ corresponding to $\bf U$ are both positive; see 
	Definition \ref{def:G}.} of
the reconstructed or DG solutions at certain nodal points. Those techniques are built on a presumed proposition that the cell-averaged solutions of those schemes are always admissible.
Such a proposition has not been rigorously proved for those schemes \cite{Balsara2012,cheng},
although it could be deduced for the 1D schemes in \cite{cheng} 
under some assumptions (see a discussion in \cite[Remark 2.12]{Wu2017a}). 
In fact, unfortunately, a usual way of using PP limiter does not necessarily ensure the PP property of the standard conservative DG schemes for multidimensional MHD system \eqref{eq:MHD};  
see \cite{Wu2017a} for a rigorous analysis.  
Based on the presumed PP property of the Lax--Friedrichs (LF) scheme, Christlieb et al.\ \cite{Christlieb,Christlieb2016} developed PP high-order finite difference weighted essentially non-oscillatory schemes for \eqref{eq:MHD} by extending the flux limiters in  \cite{Xu2014,Xiong2016}. 

It was numerically demonstrated that  
all the above PP techniques could improve the robustness of the MHD codes. 
However, there were few theoretical evidences,  
especially in the multidimensional cases, 
to genuinely and
completely prove the PP property of those or any other schemes for \eqref{eq:MHD}.  
Very recently, 
rigorous PP analysis was first carried out in 
\cite{Wu2017a} for %high-order 
conservative finite volume and 
DG schemes with the LF flux. 
The analysis showed in theory that 
the divergence-free condition \eqref{eq:2D:BxBy0} is strongly connected with the PP property of numerical
schemes for \eqref{eq:MHD}, and found that a discrete divergence-free (DDF) condition is necessary and crucial for designing the PP conservative schemes for \eqref{eq:MHD}. 
It was also proved in \cite{Wu2017a} that 
	even the first-order multidimensional LF scheme for \eqref{eq:MHD} 
	is not PP, if the proposed DDF condition 
	is slightly violated.  
Moreover, the DDF condition 
cannot be ensured by using a locally divergence-free basis  \cite{Li2005}. (Note that the first-order LF scheme is locally divergence-free.)  
This implies, in the multidimensional cases,   
%unlike the Euler equations \cite{zhang2010b}, 
a usual PP limiter (e.g., \cite{cheng}) does not guarantee the PP property of the standard DG schemes for the conservative MHD system \eqref{eq:MHD}, 
even if the locally divergence-free DG element \cite{Li2005} is used. 
%In practice it is nontrivial to rigorously meet 
%the DDF condition for multidimensional high-order schemes.

Interestingly, on the other hand, in the PDE level the preservation of positivity and the divergence-free condition (1) are also inextricably linked for the
ideal MHD system. 
For the conservative MHD system \eqref{eq:MHD}, 
Janhunen \cite{Janhunen2000} pointed out that 
 the exact solutions to 1D Riemann
problems sometimes fail to be positive, if there is  
a jump in the normal magnetic field, i.e., a nonzero $\nabla \cdot {\bf B}$, in the initial data.
We also observe that, in the multidimensional cases, the 
non-negativity of pressure is not always preserved  
by even the exact smooth solution of the 
conservative system \eqref{eq:MHD} if the divergence-free condition \eqref{eq:2D:BxBy0} is (slightly) violated, see Appendix \ref{app:evidences} of this paper.  Fortunately, it seems that the Godunov form \eqref{eq:MHD:GP} does not suffer from 
this issue. 
Therefore, we have the strong motivation to 
construct multidimensional provably PP schemes via 
proper discretization of the modified system \eqref{eq:MHD:GP} 
 rather than the conservative system \eqref{eq:MHD}.

The aim of this paper is to design and analyze provably PP high-order DG methods for multidimensional ideal MHD %by considering 
with the aid of 
the Godunov form \eqref{eq:MHD:GP}. 
This is highly nontrivial. 
The difficulties mainly arise from the intrinsic complexity of the MHD equations as well as the underlying relation between the PP property and the divergence-free condition. 
%Those are overcome by a novel equivalent form of the admissible state set, and some technical inequalities and estimates. 
%The paper is organized as follows. 
%Section \ref{sec:states} introduces 
Our analysis techniques include 
a novel equivalent form of the admissible state set and technical inequalities, introduced in Section \ref{sec:states}. 
This paper would give an insight into further understanding the importance of divergence-free condition \eqref{eq:2D:BxBy0} for preserving positivity. Other main contributions of this paper are 
follows. 
We prove a first-order LF scheme for \eqref{eq:MHD:GP} 
is PP (see Section
\ref{sec:2D:FirstOrder}), and we develop provably PP high-order DG methods 
for \eqref{eq:MHD:GP} (see Section \ref{sec:High2D}). 
Our PP DG schemes 
have three crucial ingredients\footnote{Notice that the coupling of these three techniques have also been tested in \cite{cheng} for the simulations of conservative MHD equations \eqref{eq:MHD}. It is worth clarifying that such coupling for a  conservative DG scheme does not necessarily give a genuinely PP scheme, as shown by the analysis in \cite{Wu2017a}.} ---  
the locally divergence-free spatial 
discretization for the modified MHD system \eqref{eq:MHD:GP}, 
 the PP limiter in \cite{cheng} to enforce the admissibility of the
DG solutions, and 
the strong stability preserving methods \cite{Gottlieb2009} for time discretization.  
The framework also applies to achieving provably PP high-order finite volume schemes for \eqref{eq:MHD:GP}. 
We rigorously prove the PP property of the proposed PP high-order  
schemes in Section \ref{sec:222}, and further confirm the PP property by 
  numerical experiments in Section \ref{sec:examples}, before concluding the paper in Section \ref{sec:con}.

\section{Admissible States}\label{sec:states}

Under the condition \eqref{eq:assumpEOS}, %it is natural to
it is very natural and intuitive to
define the set of (physically) 
admissible states of the ideal MHD
as follows.

\begin{definition}\label{def:G}
	The set of admissible states of the ideal MHD %equations
	%\eqref{eq:MHD} 
	is defined by
	\begin{equation}\label{eq:DefG}
	{\mathcal G} = \left\{   {\bf U} = (\rho,{\bf m},{\bf B},E)^\top ~\Big|~ \rho > 0,~
	{\mathcal E}(  {\bf U}  ) := E- \frac12 \left( \frac{|{\bf m}|^2}{\rho} + |{\bf B}|^2 \right) > 0 \right\},
	\end{equation}
	where ${\mathcal E} ({\bf U}) =  \rho e$ denotes the internal energy.
\end{definition}

Given that the initial data are admissible, 
a scheme is defined to be PP if the numerical solutions are always preserved in the set $\mathcal G$. One can verify that 
$\mathcal G$ is a convex set \cite{cheng} since ${\mathcal E}(  {\bf U}  )$  is a concave function of 
${\bf U}$ when $\rho>0$. While  
the function ${\mathcal E} ({\bf U})$ in \eqref{eq:DefG} is nonlinear, it is difficult to analytically analyze 
 the PP property of a given scheme. The following equivalent form of ${\mathcal G}$ was proposed in \cite{Wu2017a}.

\begin{lemma}%[Equivalent definition]
	\label{theo:eqDefG}
	The admissible state set ${\mathcal G}$ is equivalent to
	\begin{equation}\label{eq:newDefG}
	{\mathcal G}_* = \left\{   {\bf U} = (\rho,{\bf m},{\bf B},E)^\top ~\Big|~ \rho > 0,~~~
	{\bf U} \cdot {\bf n}^* + \frac{|{\bf B}^*|^2}{2} > 0,~\forall~{\bf v}^*, {\bf B}^* \in {\mathbb {R}}^3 \right\},
	\end{equation}
	where
	$${\bf n}^* = \bigg( \frac{|{\bf v}^*|^2}2,~- {\bf v}^*,~-{\bf B}^*,~1 \bigg)^\top.$$
\end{lemma}

{\em The two constraints in the set ${\mathcal G}_*$ 
	are both linear with respect to $\bf U$, making it more effective to analytically
	verify the PP property of schemes for the ideal MHD.}  
This novel equivalent form will play an important role in our PP analysis.

In addition, we also establish the following lemmas for the PP analysis.

\begin{lemma}\label{lem:rho}
	The set 
	$$
	{\mathcal G}_\rho = \big \{ {\bf U} = (\rho,{\bf m},{\bf B},E)^\top ~\big|~ \rho >0 \big \},
	$$
	is a convex set. And for any ${\bf U} \in {\mathcal G}_\rho$ and $\alpha > |v_i|$, 
	we have 
	${\bf U}\pm \alpha^{-1} {\bf F}_i ({\bf U}) \in {\mathcal G}_\rho $. 
\end{lemma}

\begin{proof}
	The result can be easily verified.
\end{proof}

\begin{lemma}
	For any ${\bf U} \in {\mathcal G}$ and ${\bf v}^*,{\bf B}^* \in {\mathbb R}^3$, it holds 
	\begin{align} \label{eq:indentityS}
		&{\bf S}({\bf U}) \cdot {\bf n}^* = ( {\bf v} - {\bf v}^* ) \cdot ( {\bf B} - {\bf B}^* ) - {\bf v}^* \cdot  {\bf B}^*, 
		\\ \label{eq:widelyusedIEQ}
		& |\sqrt{\rho}( {\bf v} - {\bf v}^* ) \cdot ( {\bf B} - {\bf B}^* )| 
		< {\bf U} \cdot {\bf n}^* + \frac{|{\bf B}^*|^2}{2}.
	\end{align}
	Furthermore, for any $b \in {\mathbb R}$, we have 
	\begin{equation}\label{eq:widelyusedIEQ2}
	b ( {\bf S}({\bf U}) \cdot {\bf n}^* ) \le \frac{|b|}{\sqrt{\rho}} \left( {\bf U} \cdot {\bf n}^* + \frac{|{\bf B}^*|^2}{2} \right)  - b ( {\bf v}^* \cdot  {\bf B}^* ).
	\end{equation}
\end{lemma}

\begin{proof}
	The identity \eqref{eq:indentityS} can be directly verified. The inequality \eqref{eq:widelyusedIEQ} 
	is shown as follows.
	\begin{align*}
		{\bf U} \cdot {\bf n}^* + \frac{|{\bf B}^*|^2}{2}
		& = \frac{\rho}{2} | {\bf v} - {\bf v}^*  |^2 + \frac{|{\bf B} - {\bf B}^*|^2}{2}  + {\mathcal E}({\bf U}) 
		\\
		& > \frac{\rho}{2} | {\bf v} - {\bf v}^*  |^2 + \frac{|{\bf B} - {\bf B}^*|^2}{2}
		\\
		&\ge 	
		|\sqrt{\rho}( {\bf v} - {\bf v}^* ) \cdot ( {\bf B} - {\bf B}^* )|. 
	\end{align*}
	Combining \eqref{eq:indentityS} and \eqref{eq:widelyusedIEQ} gives \eqref{eq:widelyusedIEQ2}. 
\end{proof}

We also need the following inequality, which was technically constructed and proved in 
\cite{Wu2017a}, and has played a pivotal role 
in analyzing the PP properties of conservative schemes for the ideal MHD \cite{Wu2017a}.

\begin{lemma}\label{theo:MHD:LLFsplit}
	If ${\bf U}, \tilde{\bf U} \in {\mathcal G}$, then the inequality
	\begin{equation}\label{eq:MHD:LLFsplit}
	\bigg( {\bf U} - \frac{ {\bf F}_i({\bf U})}{\alpha}
	+
	\tilde{\bf U} + \frac{ {\bf F}_i(\tilde{\bf U})}{\alpha}
	\bigg) \cdot {\bf n}^* + |{\bf B}^*|^2
	+  \frac{  B_i - \tilde B_i }{\alpha} ({\bf v}^* \cdot {\bf B}^*)  > 0 ,
	\end{equation}
	holds for any ${\bf v}^*, {\bf B}^* \in {\mathbb{R}}^3$ and any $|\alpha|>\alpha_{i} ({\bf U},\tilde{\bf U}) $, where $i\in\{1,2,3\}$, and 
	\begin{gather}\label{eq:alpha_i}
		\alpha_{i} ({\bf U},\tilde{\bf U}) = 
		\min_{\sigma \in \mathbb{R}} \alpha_{i} ( {\bf U}, \tilde {\bf U};\sigma ),
		\\ \nonumber
		\alpha_{i} ( {\bf U}, \tilde {\bf U};\sigma ) 
		=
		\max\big\{ |v_i|+ {\mathscr{C}}_i,|\tilde v_i| +  \tilde {\mathscr{C}}_i ,| \sigma v_i + (1-\sigma) \tilde v_i | + \max\{ {\mathscr{C}}_i , \tilde {\mathscr{C}}_i \}  \big\}
		+ f ( {\bf U}, \tilde {\bf U};  \sigma ),
	\end{gather}
	with 
	\begin{align*}
		&		f( {\bf U}, \tilde {\bf U}; \sigma) = \frac{ |\tilde{\bf B}-{\bf B}| }{\sqrt{2}} \sqrt{  \frac{\sigma^2}{\rho} + \frac{ (1-\sigma)^2 }{\tilde \rho}  },
		\\	
		&	
		{\mathscr{C}}_i = \frac{1}{\sqrt{2}}  \left[ {\mathscr{C}}_s^2 + \frac{ |{\bf B}|^2}{\rho} + \sqrt{ \left( {\mathscr{C}}_s^2 + \frac{ |{\bf B}|^2}{\rho} \right)^2 - 4 \frac{ {\mathscr{C}}_s^2 B_i^2}{\rho}  } \right]^\frac12,
	\end{align*}
	and ${\mathscr{C}}_s=\frac{p}{\rho \sqrt{2e}}$.
\end{lemma}

	In practice, it is not easy to know the minimum value in 
	\eqref{eq:alpha_i}. Because $\alpha_i({\bf U},\tilde{\bf U})$ only serves as a lower bound, 
	one can certainly replace it with $\alpha_i({\bf U},\tilde{\bf U};\sigma)$ for a special 
	$\sigma$. For example, taking $\sigma=\frac{\rho}{\rho+\tilde \rho}$ 
	minimizes $f({\bf U},\tilde{\bf U};\sigma)$ and implies 
	{\small
		\begin{equation*}%\label{eq:alpha_iAA}	
			\alpha_{i} \bigg({\bf U},\tilde{\bf U};\frac{\rho}{\rho+\tilde \rho} \bigg) =
			\max\bigg\{ |v_i|+ {\mathscr{C}}_i, |\tilde v_i| +  \tilde {\mathscr{C}}_i , \frac{|\rho v_i + \tilde \rho \tilde v_i |}{\rho + \tilde \rho} + \max\{ {\mathscr{C}}_i , \tilde {\mathscr{C}}_i \}  \bigg\}
			+ \frac{ |{\bf B}-\tilde{\bf B}| }{ \sqrt{2 (\rho + \tilde \rho)}  }.
	\end{equation*}} 
	Taking $\sigma=\frac{ \sqrt{\rho}}{\sqrt{\rho}+\sqrt{\tilde \rho}}$ gives
	{\small  
		\begin{equation*} 
			\alpha_{i} \bigg({\bf U},\tilde{\bf U};\frac{\sqrt{\rho}}{\sqrt{\rho}+\sqrt{\tilde \rho}} \bigg) =
			\max\bigg\{ |v_i|+ {\mathscr{C}}_i, |\tilde v_i| +  \tilde {\mathscr{C}}_i , \frac{| \sqrt{\rho} v_i + \sqrt{\tilde \rho} \tilde v_i |} {\sqrt{\rho}+\sqrt{\tilde \rho}} + \max\{ {\mathscr{C}}_i , \tilde {\mathscr{C}}_i \}  \bigg\}
			+ \frac{ |{\bf B}-\tilde{\bf B}| }{ \sqrt{\rho} + \sqrt{\tilde \rho}  }.
		\end{equation*}
	}

Let  ${\mathscr{R}}_i ({\bf U})$ denote the spectral radius of the Jacobian matrix 
of the MHD system \eqref{eq:MHD:GP} in the $x_i$-direction, $i=1,2,3$. 
For the gamma-law EOS \eqref{eq:gEOS}, we have \cite{Powell1994}
$$
{\mathscr{R}}_i ({\bf U}) = |v_i| + \mathcal{C}_i,
$$
with 
$$
{\mathcal C}_i := \frac{1}{\sqrt{2}} \left[ \mathcal{C}_s^2 + \frac{ |{\bf B}|^2}{\rho} + \sqrt{ \left( \mathcal{C}_s^2 + \frac{ |{\bf B}|^2}{\rho} \right)^2 - 4 \frac{ \mathcal{C}_s^2 B_i^2}{\rho}  }  \right]^\frac12,
$$
where $\mathcal{C}_s=\sqrt{\gamma p/\rho}$ denotes the local sound speed. 
Let  $a_i := \max\{ {\mathscr{R}}_i ({\bf U}), {\mathscr{R}}_i(\tilde{\bf U}) \}$. For the gamma-law EOS,  it was shown in \cite{Wu2017a} that  
	\begin{gather} \label{eq:aaaaWKL}
		\alpha_{i} ({\bf U},\tilde{\bf U}) 
		\le 2a_i,
		\\ \label{eq:bbbbWKL}
		\alpha_{i} ({\bf U},\tilde{\bf U}) 
		\le a_i + \min \big\{ \big| | v_i| - |\tilde v_i| \big|, \big|  {\mathscr{C}}_i - \tilde {\mathscr{C}}_i  \big|  \big\} 
		+ \frac{ |{\bf B}-\tilde{\bf B}| }{ \sqrt{2 (\rho + \tilde \rho)}  },
	\end{gather}
where the latter implies that $\alpha_{i} ({\bf U},\tilde{\bf U}) \le a_i + {\mathcal O}( |{\bf U} -\tilde {\bf U}| )$, $i=1,2,3$.

\begin{remark}
	We would like to emphasize the importance of the last term at the left-hand side of \eqref{eq:MHD:LLFsplit}.
	This term is very technical, necessary and crucial in proving the PP property of the schemes proposed in the following. 
	The inclusion of this term is a key point in 
	our present PP analysis; see also \cite{Wu2017a}.   
	This term is not always negative or positive. However, dropping  it, the inequality \eqref{eq:MHD:LLFsplit} will not hold, even if we replace $\alpha_i$ with $\chi\alpha_i$ for any constant $\chi \ge 1$.  
	More interestingly and importantly, this term will help us to skillfully 
	utilize the contribution of the discretized GP source term
	that makes the proposed schemes PP.  
\end{remark}

\section{Provably Positivity-Preserving Methods}\label{sec:2Dpcp}

This section develops provably PP methods for the modified MHD system \eqref{eq:MHD:GP} in two dimension ($d=2$). 
The extension to three-dimensional case ($d=3$) is quite straightforward.

To avoid confusing subscripts, we will use 
the symbols $({\tt x},{\tt y})$ to represent the variables $(x_1,x_2)$ in \eqref{eq:MHD:GP}. 
Assume that the 2D spatial domain is divided into a uniform rectangular mesh with cells $\big\{I_{ij}=({\tt x}_{i-\frac{1}{2}},{\tt x}_{i+\frac{1}{2}})\times
({\tt y}_{j-\frac{1}{2}},{\tt y}_{j+\frac{1}{2}}) \big\}$. 
The spatial step-sizes in ${\tt x}$ and ${\tt y}$ directions are denoted by
$\Delta x$ and $\Delta y$, respectively.
 The time interval is also divided into the mesh $\{t_0=0, t_{n+1}=t_n+\Delta t_{n}, n\geq 0\}$
with the time step-size $\Delta t_{n}$ determined by the CFL condition.  
%We aim at seeking numerical schemes whose solution  
%$\bar {\bf U}_{ij}^n$ always  
%stay at $\mathcal G$. 

\subsection{First-order scheme} \label{sec:2D:FirstOrder} 
We consider the following first-order scheme for \eqref{eq:MHD:GP} 
\begin{equation} \label{eq:2DMHD:LFscheme}
\begin{split}
\bar {\bf U}_{ij}^{n+1} &= \bar {\bf U}_{ij}^n - \frac{\Delta t_n}{\Delta x} \Big( \hat {\bf F}_1 ( \bar {\bf U}_{ij}^n, \bar {\bf U}_{i+1,j}^n  ) - \hat {\bf F}_1 ( \bar {\bf U}_{i-1,j}^n, \bar {\bf U}_{ij}^n  ) \Big) \\
& \quad
- \frac{\Delta t_n}{\Delta y} \Big( \hat {\bf F}_2 ( \bar {\bf U}_{ij}^n ,\bar {\bf U}_{i,j+1}^n) - \hat {\bf F}_2 ( \bar {\bf U}_{i,j-1}^n, \bar {\bf U}_{ij}^n )  \Big)
- \Delta t_n \big( {\rm div}_{ij} \bar{\bf B}^n \big)  
{\bf S} ( \bar{\bf U}_{ij}^n ),
\end{split}
\end{equation}
where $\bar {\bf U}_{ij}^n $ is the numerical approximation to the cell average of the exact solution ${\bf U}({\tt x},{\tt y},t)$ over $I_{ij}$ at time $t_n$, and $\hat {\bf F}_1,\hat {\bf F}_2$ are  the numerical fluxes. We focus on the Lax--Friedrichs (LF) 
flux 
\begin{equation}\label{eq:LFflux}
\hat {\bf F}_\ell ( {\bf U}^- , {\bf U}^+ ) = \frac{1}{2} \Big( {\bf F}_\ell (  {\bf U}^- ) + {\bf F}_\ell ( {\bf U}^+ ) - 
\alpha_{\ell ,n}^{\tt LF} (  {\bf U}^+ -  {\bf U}^- ) \Big),\quad \ell=1,2, 
\end{equation}
where $\alpha_{\ell ,n}^{\tt LF}$ denotes the numerical viscosity parameter. 
The last term at the right-hand side of \eqref{eq:2DMHD:LFscheme} 
is a penalty-type term, in which ${\rm div}_{ij} \bar{\bf B}^n$ is the discrete divergence \cite{Wu2017a} defined by   
\begin{equation}\label{eq:DisDivB}
\mbox{\rm div} _{ij} \bar {\bf B}^n := \frac{ \left( \bar  B_1\right)_{i+1,j}^n - \left( \bar  B_1 \right)_{i-1,j}^n } {2\Delta x} + \frac{ \left( \bar  B_2 \right)_{i,j+1}^n - \left( \bar B_2 \right)_{i,j-1}^n } {2\Delta y}.
\end{equation}
The discrete divergence $\mbox{\rm div} _{ij} \bar {\bf B}^n$ can be considered as a discretization of  $\nabla \cdot {\bf B}$ at the center of $I_{ij}$. Such discretization was also used in \cite{Chandrashekar2016}.

The PP property of \eqref{eq:2DMHD:LFscheme} is rigorously proved as follows.

\begin{theorem} \label{theo:2DMHD:LFscheme}
	Assume that the parameters $\alpha_{1,n}^{\tt LF}$ and $\alpha_{2,n}^{\tt LF}$ in \eqref{eq:LFflux} satisfy
	\begin{equation}\label{eq:Lxa12}
	\alpha_{1,n}^{\tt LF} > \alpha_{1,n}^{\tt PP} := \max_{i,j} \alpha_1 ( \bar {\bf U}_{i+1,j}^n, \bar {\bf U}_{i-1,j}^n ),~
	\alpha_{2,n}^{\tt LF} > \alpha_{2,n}^{\tt PP} := \max_{i,j} \alpha_2 ( \bar {\bf U}_{i,j+1}^n, \bar {\bf U}_{i,j-1}^n ).
	\end{equation}
	If $\bar {\bf U}_{ij}^n \in {\mathcal G}$  for all $i$ and $j$,
	then the solution $ \bar {\bf U}_{ij}^{n+1}$ of \eqref{eq:2DMHD:LFscheme} belongs to ${\mathcal G}$ under the CFL-type condition
	\begin{equation}\label{eq:CFL:LF2D}
	0< \Delta t_n \bigg( \frac{ \alpha_{1,n}^{\tt LF} }{\Delta x}  + \frac{  \alpha_{2,n}^{\tt LF} } {\Delta y} + 
\vartheta_n \bigg)  \le  1,
	\end{equation}
	where 
\begin{equation}\label{eq:beta2DLF}
	\vartheta_n = \max_{i,j} \frac{ | \mbox{\rm div} _{ij} \bar {\bf B}^n | }{ \sqrt{ \bar \rho_{ij}^n }}.
\end{equation}
\end{theorem}

\begin{proof}
	Substituting  \eqref{eq:LFflux} into \eqref{eq:2DMHD:LFscheme} gives
	\begin{equation}\label{eq:proofWKL-EQ1}
	\bar {\bf U}_{ij}^{n+1} = \lambda_1 {\bf \Xi}_1 + \lambda_2 {\bf \Xi}_2 
	+  (1-\lambda) \bar {\bf U}_{ij}^n 
	- \Delta t_n \big( {\rm div}_{ij} \bar{\bf B}_{ij}^n \big)  {\bf S} ( \bar{\bf U}_{ij}^n ),   
	\end{equation}
	where 
	$$\lambda_1 = \frac{ \alpha_{1,n}^{\tt LF} \Delta t_n } { \Delta x}, \quad  \lambda_2 = \frac{ \alpha_{2,n}^{\tt LF} \Delta t_n }{\Delta y}, \quad \lambda = \lambda_1+ \lambda_2,$$
	 and 
	\begin{align*}
	{\bf \Xi}_1 &= \frac12
	\left( \bar {\bf U}_{i+1,j}^n - \frac{ {\bf F}_1( \bar {\bf U}_{i+1,j}^n)}{ \alpha_{1,n}^{\tt LF} } +
	\bar {\bf U}_{i-1,j}^n + \frac{ {\bf F}_1( \bar {\bf U}_{i-1,j}^n) }{ \alpha_{1,n}^{\tt LF} } \right), \\
    	{\bf \Xi}_2	&= \frac12 \left( \bar {\bf U}_{i,j+1}^n - \frac{ {\bf F}_2( \bar {\bf U}_{i,j+1}^n)}{ \alpha_{2,n}^{\tt LF} } +
	\bar {\bf U}_{i,j-1}^n + \frac{ {\bf F}_2( \bar {\bf U}_{i,j-1}^n) }{ \alpha_{2,n}^{\tt LF} } \right).
	\end{align*}

	Under the condition \eqref{eq:Lxa12}, Lemma \ref{lem:rho} implies ${\bf \Xi}_k \in {\mathcal G}_\rho$, 
	i.e., the first component of ${\bf \Xi}_k$ is positive, $k=1,2$. Therefore, we have $\bar {\rho}_{ij}^{n+1} 
	> (1-\lambda) \bar {\rho}_{ij}^{n} \ge 0$, by noting that first component of ${{\bf S} ( \bar{\bf U}_{ij}^n )}$ is zero. 
	
	For any ${\bf v}^*,{\bf B}^* \in {\mathbb R}^3$, by using the identity \eqref{eq:indentityS}, we 
	derive from \eqref{eq:proofWKL-EQ1} that 
	$$
	\bar {\bf U}_{ij}^{n+1} \cdot {\bf n}^* + \frac{ |{\bf B}^*|^2 }{2} = \Pi_1 + \Pi_2, 
	$$
	where 
	\begin{align*}
		\Pi_1 &= \lambda_1 \left( {\bf \Xi}_1 \cdot {\bf n}^* + \frac{|{\bf B}^*|^2}{2} \right) 
	+ \lambda_2 \left( {\bf \Xi}_2 \cdot {\bf n}^* + \frac{|{\bf B}^*|^2}{2} \right)  
		+ \Delta t_n \big( \mbox{\rm div} _{ij} \bar {\bf B}^n \big) \big( {\bf v}^* \cdot {\bf B}^* \big),\\
		\Pi_2 &= 
		(1-\lambda) \left( \bar {\bf U}_{ij}^n \cdot {\bf n}^* + \frac{|{\bf B}^*|^2}{2} \right)
		- \Delta t_n \big( {\rm div}_{ij} \bar{\bf B}^n \big) ( \bar{\bf v}_{ij}^n - {\bf v}^*  )
		\cdot  ( \bar{\bf B}_{ij}^n - {\bf B}^*  ) .
	\end{align*}
	%$\overset{\eqref{eq:MHD:LLFsplit}}{>}$
	The inequality \eqref{eq:MHD:LLFsplit} implies 
\begin{equation*}
		\begin{aligned}
		\Pi_1 &>  \frac12 \left( - \lambda_1 \frac{  ( \bar B_1)_{i+1,j}^n - ( \bar B_1)_{i-1,j}^n }
		{  \alpha_{1,n}^{\tt LF}  } 
		 - \lambda_2 \frac{  ( \bar B_2)_{i,j+1}^n - ( \bar B_2)_{i,j-1}^n }
		{  \alpha_{2,n}^{\tt LF}  } \right) ( {\bf v}^* \cdot {\bf B}^* ) 
		\\
		&\quad + \Delta t_n \big( \mbox{\rm div} _{ij} \bar {\bf B}^n \big) \big( {\bf v}^* \cdot {\bf B}^* \big) = 0.
	\end{aligned}
\end{equation*}
	Using the inequality \eqref{eq:widelyusedIEQ} gives 
	\begin{align*}
		\Pi_2 & \ge (1-\lambda) 
		\left( \bar {\bf U}_{ij}^n \cdot {\bf n}^* + \frac{|{\bf B}^*|^2}{2} \right)
		- \Delta t_n  \vartheta_n \left| \sqrt{\bar \rho_{ij}^n}  ( \bar{\bf v}_{ij}^n - {\bf v}^*  )
		\cdot  ( \bar{\bf B}_{ij}^n - {\bf B}^*  )  \right|
		\\
		&\ge   (1-\lambda- \Delta t_n  \vartheta_n)  \left( \bar {\bf U}_{ij}^n \cdot {\bf n}^* + \frac{|{\bf B}^*|^2}{2} \right) \ge 0.
	\end{align*}
Hence we obtain $\bar {\bf U}_{ij}^{n+1} \cdot {\bf n}^* + \frac{ |{\bf B}^*|^2 }{2} > 0,~\forall {\bf v}^*,{\bf B}^* \in {\mathbb R}^3$. 

According Lemma \ref{theo:eqDefG}, we have $\bar {\bf U}_{ij}^{n+1} \in {\mathcal G}$. 
The proof is completed.
\end{proof}

\begin{remark}
Let $\alpha_{\ell,n}^{\tt std}:=\max_{i,j} {\mathscr{R}}_\ell ( \bar{\bf U}_{ij}^n )$ be the standard parameter in the LF flux. 
It was proved in \cite{Wu2017a} that even the 1D LF scheme 
with this standard parameter is not PP in general, regardless of how small 
the CFL number is. 
While the lower bounds given in \eqref{eq:Lxa12} for the parameters $\alpha_{\ell,n}^{\tt LF},\ell=1,2$, are acceptable, because 
one can derive from \eqref{eq:aaaaWKL} and \eqref{eq:bbbbWKL} that 
$$
\alpha_{\ell,n}^{\tt PP}
\le 2 \alpha_{\ell,n}^{\tt std},\qquad \ell=1,2,
$$
and for smooth problems,
$$
\alpha_{\ell,n}^{\tt PP} \le \alpha_{\ell,n}^{\tt std} 
+ {\mathcal O}( \max\{\Delta x,\Delta y\} ),\qquad \ell=1,2.
$$ 
\end{remark}

\begin{remark}
The scheme
\eqref{eq:2DMHD:LFscheme} without the penalty-type term 
reduces to the 2D LF scheme for 
the conservative MHD system \eqref{eq:MHD}. 
It was shown in \cite{Wu2017a} that 
the 2D LF scheme for \eqref{eq:MHD} is generally not PP,  
unless a discrete divergence-free condition, $\mbox{\rm div} _{ij} \bar {\bf B}^n=0$, is satisfied. 
While, by including the penalty-type term, 
the scheme \eqref{eq:2DMHD:LFscheme} becomes PP even if  
that discrete divergence-free condition is not met.  
\end{remark}

\subsection{High-order schemes}\label{sec:High2D}

We now present the provably PP high-order methods for the 2D MHD equations \eqref{eq:MHD:GP}. We mainly focus on the PP 
high-order discontinuous Galerkin (DG) methods, keeping in mind that 
the same framework also applies to high-order finite volume schemes. 
The PP high-order schemes are built on the locally divergence-free schemes designed in Section \ref{sec:111}.  

For convenience, we first focus on the forward Euler method for time discretization, while high-order time discretization will be discussed later. 

\subsubsection{Locally divergence-free schemes} \label{sec:111}

To achieve high-order spatial accuracy, we approximate the exact solution ${\bf U}({\tt x},{\tt y},t_n)$ 
with a discontinuous piecewise polynomial function  
${\bf U}_h^n ( {\tt x},{\tt y} )$, which is sought in the locally divergence-free space \cite{Li2005} 
$$
{\mathbb W}_h^{\tt K} 
= \left\{ {\bf w}=(w_1,\cdots,w_8)^\top~\Big|~ w_\ell \big|_{I_{ij}} \in {\mathbb P}^{\tt K} (I_{ij}),~  
\bigg( \frac{ \partial w_5}{\partial {\tt x}} + \frac{ \partial w_6}{\partial {\tt y}}  \bigg)\bigg|_{I_{ij}}  = 0,~\forall i,j,\ell   \right\},
$$
where ${\mathbb P}^{\tt K} (I_{ij})$ denotes the space of polynomials in $I_{ij}$ of degree at most $\tt K$.

We consider the ${\mathbb P}^{\tt K}$-based locally divergence-free DG method for the Godunov form \eqref{eq:MHD:GP} of the ideal MHD equations. Specifically, ${\bf U}_h^n \in {\mathbb W}_h^{\tt K} $ is evolved by 
%we look for ${\bf U}_h^n \in {\mathbb W}_h^{\tt K} $, such that for any ${\bf w} \in {\mathbb W}_h^{\tt K} $, 
\begin{equation}\label{eq:2DDGUh}
\begin{split}
& \int_{I_{ij}} {\bf w} \cdot  \frac{{\bf U}_h^{n+1}- {\bf U}_h^{n}}{\Delta t_n}   d{\tt x} d {\tt y}
 =  \int_{I_{ij}}   \partial_{\tt x} {\bf w} \cdot  
 {\bf F}_1  ( {\bf U}_h^{n} )  d{\tt x} d {\tt y}  
  \\ 
& 
 +  \int_{I_{ij}}  \partial_{\tt y} {\bf w}  \cdot  
 {\bf F}_2  ( {\bf U}_h^{n} )  d{\tt x} d {\tt y} 
 - \int_{{\tt y}_{j-\frac12}}^{{\tt y}_{j+\frac12}} {\bf H}_{1,i} ({\tt y}) d {\tt y} 
- \int_{{\tt x}_{i-\frac12}}^{{\tt x}_{i+\frac12}} {\bf H}_{2,j} ({\tt x}) d {\tt x},~\forall {\bf w} \in {\mathbb W}_h^{\tt K},
\end{split} 
\end{equation} 
where 
\begin{align*}
	\begin{split}
& {\bf H}_{1,i} ({\tt y}) 
= {\bf w} ( {\tt x}_{i+\frac12} ^-, {\tt y} ) \cdot  \hat {\bf F}_{1,i+\frac12}({\tt y}) 
- {\bf w} ( {\tt x}_{i-\frac12} ^+, {\tt y} ) \cdot  \hat {\bf F}_{1,i-\frac12}({\tt y})
\\
&   + {\mathscr B}_{1,i+\frac12} ({\tt y}) 
{\bf w} ( {\tt x}_{i+\frac12} ^-, {\tt y} ) \cdot {\bf S} ({\bf U}_h^n( {\tt x}_{i+\frac12} ^-, {\tt y} ) )
 + {\mathscr B}_{1,i-\frac12} ({\tt y}) 
{\bf w} ( {\tt x}_{i-\frac12} ^+, {\tt y} ) \cdot {\bf S} ( {\bf U}_h^n( {\tt x}_{i-\frac12} ^+, {\tt y} )),
\end{split} 
\\
	\begin{split}
&	{\bf H}_{2,j} ({\tt x}) 
	= {\bf w} ( {\tt x}, {\tt y}_{j+\frac12} ^- ) \cdot  \hat {\bf F}_{2,j+\frac12}({\tt x}) 
	- {\bf w} ( {\tt x}, {\tt y}_{j-\frac12} ^+ ) \cdot  \hat {\bf F}_{2,j-\frac12}({\tt x})
	\\
	&   + {\mathscr B}_{2,j+\frac12} ({\tt x}) 
	{\bf w} ( {\tt x}, {\tt y}_{j+\frac12} ^- ) \cdot {\bf S} ( {\bf U}_h^n( {\tt x}, {\tt y}_{j+\frac12} ^- )) 
	+ {\mathscr B}_{2,j-\frac12} ({\tt x}) 
	{\bf w} ( {\tt x}, {\tt y}_{j-\frac12} ^+ ) \cdot {\bf S} ( {\bf U}_h^n( {\tt x}, {\tt y}_{j-\frac12} ^+ )),
\end{split} 
\end{align*}
with the superscripts $-$ and $+$ on ${\tt x}_{i+\frac12} $ indicating that the associated limit is a left- or right-handed limit, and 
\begin{align*}
&
\hat {\bf F}_{1,i+\frac12}({\tt y}) = \hat {\bf F}_1 \left( {\bf U}_h^n ( {\tt x}_{i+\frac12} ^-, {\tt y} ), {\bf U}_h^n ( {\tt x}_{i+\frac12} ^+, {\tt y} ) \right),\\
&
\hat {\bf F}_{2,j+\frac12}({\tt x}) = \hat {\bf F}_2 \left( {\bf U}_h^n ( {\tt x}, {\tt y}_{j+\frac12} ^- ), {\bf U}_h^n ( {\tt x}, {\tt y}_{j+\frac12} ^+ ) \right),\\
&
{\mathscr B}_{1,i+\frac12} ({\tt y}) = \frac12 \left( (B_1)_h^n( {\tt x}_{i+\frac12} ^+,{\tt y} ) - (B_1)_h^n( {\tt x}_{i+\frac12} ^- ,{\tt y} )\right),\\
&
{\mathscr B}_{2,j+\frac12} ({\tt x}) = \frac12 \left( (B_2)_h^n( {\tt x},{\tt y}_{j+\frac12} ^+ ) - (B_2)_h^n( {\tt x} ,{\tt y}_{j+\frac12} ^- )\right),
\end{align*}
with $\hat {\bf F}_1, \hat {\bf F}_2$ taken the LF fluxes in \eqref{eq:LFflux}. 
Similar discretization of the GP source term in \eqref{eq:MHD:GP} was also used in \cite{Chandrashekar2016b,LiuShuZhang2017} recently.

In the computations, the boundary and volume integrals at the right-hand side of \eqref{eq:2DDGUh} 
are discretized by the Gaussian quadratures %of sufficiently high order of accuracy%, for example, 
\begin{align*}
	& \int_{I_{ij}}  \big( \partial_{\tt x} {\bf w} \cdot  
	{\bf F}_1  ( {\bf U}_h^{n} )  \big) d{\tt x} d {\tt y} 
 \approx 
	\Delta x \Delta y 
	\sum_{\mu=1}^{\tt Q} \sum_{\nu=1}^{\tt Q} \omega_{\mu} \omega_{\nu} 
	\big( \partial_{\tt x} {\bf w} \cdot  
	{\bf F}_1  ( {\bf U}_h^{n} )  \big) ( {\tt x}_i^{(\mu)} , {\tt y}_j^{(\nu)}  ),
	\\
		& \int_{I_{ij}}  \big( \partial_{\tt y} {\bf w} \cdot  
	{\bf F}_2  ( {\bf U}_h^{n} )  \big) d{\tt x} d {\tt y} 
	\approx 
	\Delta x \Delta y 
	\sum_{\mu=1}^{\tt Q} \sum_{\nu=1}^{\tt Q} \omega_{\mu} \omega_{\nu} 
	\big( \partial_{\tt y} {\bf w} \cdot  
	{\bf F}_2  ( {\bf U}_h^{n} )  \big) ( {\tt x}_i^{(\mu)} , {\tt y}_j^{(\nu)}  ),
	\\
	& \int_{{\tt y}_{j-\frac12}}^{{\tt y}_{j+\frac12}} {\bf H}_{1,i} ({\tt y}) d {\tt y} 
	\approx \Delta y \sum_{\mu=1}^{\tt Q} \omega_{\mu} {\bf H}_{1,i} ({\tt y}_j^{(\mu)}),
	\quad 
	\int_{{\tt x}_{i-\frac12}}^{{\tt x}_{i+\frac12}} {\bf H}_{2,j} ({\tt x}) d {\tt x}
	 \approx \Delta x \sum_{\mu=1}^{\tt Q} \omega_{\mu} {\bf H}_{2,j} ({\tt x}_i^{(\mu)}),
\end{align*}
where ${\mathbb S}_i^{\tt x}=\{  {\tt x}_i^{(\mu)} \}_{\mu=1}^{\tt Q}$ and ${\mathbb S}_j^{\tt y}=\{  {\tt y}_j^{(\mu)} \}_{\mu=1}^{\tt Q}$ are the $\tt Q$-point Gauss-Legendre quadrature nodes in 
%the intervals 
$[ {\tt x}_{i-\frac12}, {\tt x}_{i+\frac12} ]$ and $[ {\tt y}_{j-\frac12}, {\tt y}_{j+\frac12} ]$, respectively, and $\{\omega_\mu\}_{\mu=1}^{\tt Q}$ are the associated weights satisfying
$\sum_{\mu=1}^{\tt Q} \omega_\mu = 1$, with 
${\tt Q} \ge {\tt K}+1$ for accuracy requirement \cite{Cockburn0}. 

Let denote 
$${\bf U}_h^n\big|_{I_{ij}}=: {\bf U}_{ij}^n ({\tt x},{\tt y}),
$$ 
whose cell average 
 over $I_{ij}$ is denoted by $\bar {\bf U}_{ij}^n$. One can derive from \eqref{eq:2DDGUh} the 
 evolution equations for the cell averages $\{\bar {\bf U}_{ij}^n\}$ as follows 
\begin{equation}\label{eq:2DMHD:cellaverage}
\bar {\bf U}_{ij}^{n+1}  = \bar {\bf U}_{ij}^{n} + \Delta t_n {\bf L}_{ij} ( {\bf U}_h^n  ),
\end{equation}
where 
\begin{equation*}%\label{eq:2DMHD:cellaverage}
	\begin{split}
		{\bf L}_{ij} ( {\bf U}_h^n  ) & := 
		- \frac{1}{\Delta x} \sum\limits_{\mu =1}^{\tt Q}  \omega_\mu \bigg[ \left(
		\hat {\bf F}_{1,i+\frac12}({\tt y}_{j}^{(\mu)}) 
		-   \hat {\bf F}_{1,i-\frac12}({\tt y}_{j}^{(\mu)})
		\right) 
		\\
		& \quad 
		+ \left( {\mathscr B}_{1,i+\frac12} ({\tt y}_{j}^{(\mu)}) 
		{\bf S} ( {\bf U}_h^n( {\tt x}_{i+\frac12} ^-, {\tt y}_{j}^{(\mu)} ) )
		+ {\mathscr B}_{1,i-\frac12} ({\tt y}_{j}^{(\mu)})  {\bf S} ( {\bf U}_h^n( {\tt x}_{i-\frac12} ^+, {\tt y}_{j}^{(\mu)} ) )
		\right)  \bigg]
		\\
		&\quad  -  \frac{ 1 }{\Delta y} \sum\limits_{\mu =1}^{\tt Q} \omega_\mu \bigg[\left(
		\hat {\bf F}_{2,j+\frac12}({\tt x}_{i}^{(\mu)}) 
		-   \hat {\bf F}_{2,j-\frac12}({\tt x}_{i}^{(\mu)}) \right)
		\\
		& \quad 
		+
		\left(
		{\mathscr B}_{2,j+\frac12} ({\tt x}_{i}^{(\mu)}) 
		{\bf S} ( {\bf U}_h^n( {\tt x}_{i}^{(\mu)}, {\tt y}_{j+\frac12} ^- ) )
		+ {\mathscr B}_{2,j-\frac12} ({\tt x}_{i}^{(\mu)})  {\bf S} ( {\bf U}_h^n( {\tt x}_{i}^{(\mu)}, {\tt y}_{j-\frac12} ^+ ))
		\right) \bigg].
	\end{split}
\end{equation*}

The discrete equations \eqref{eq:2DMHD:cellaverage} can also be derived 
from a finite volume method for \eqref{eq:MHD:GP}, if the approximate function 
${\bf U}_h^n  $ in \eqref{eq:2DMHD:cellaverage} 
is reconstructed from   
the cell averages $\{\bar {\bf U}_{ij}^n\}$ by a locally 
divergence-free approach (see e.g., \cite{ZhaoTang2017}) such that ${\bf U}_h^n \in {\mathbb W}_h^{\tt K}  $. 

When ${\tt K}=0$, %i.e., $ {\bf U}_{ij}^n ({\tt x},{\tt y}) \equiv \bar  {\bf U}_{ij}^n $, 
the above DG and finite volume schemes reduce to 
the first-order scheme \eqref{eq:2DMHD:LFscheme}, which has been proved to be PP. 
When ${\tt K}\ge 1$, 
the above high-order DG and finite volume schemes are not PP in general. 
However, we find that these high-order locally divergence-free schemes 
can be modified to provably PP high-order schemes, see the discussions 
in Section \ref{sec:222}.

\subsubsection{Provably PP schemes}\label{sec:222}

Based on the high-order locally divergence-free schemes presented above,  
we construct the provably PP high-order DG and finite volume schemes 
as follows. The rigorous proof of the PP property will be given later.  

\vspace{2mm}
\noindent
{\bf Step 0.} Initialization. Set $t=0$ and $n=0$. 
Using the initial data  
computes $\{\bar {\bf U}_{ij}^0\}$ 
and $\{ {\bf U}_{ij}^0 ({\tt x},{\tt y}) \}$. 
The admissibility of  $\bar {\bf U}_{ij}^0$ can be ensured by the convexity of $\mathcal G$, 
and ${\bf U}_h^0 \in  {\mathbb W}_h^{\tt K} $ is easily guaranteed if  
a local $L^2$-projection of the initial data onto $ {\mathbb W}_h^{\tt K}$ is used.

\vspace{2mm}
\noindent
{\bf Step 1.} Given admissible cell averages $\big\{\bar {\bf U}_{ij}^n\big\}$ 
and ${\bf U}_h^n  \in {\mathbb W}_h^{\tt K} $, perform the PP limiting procedure. 
Use the PP limiter in \cite{cheng} to modify the polynomials $\big\{ {\bf U}_{ij}^n ({\tt x},{\tt y})  \big\}$ as
$\big\{\widetilde {\bf U}_{ij}^n ({\tt x},{\tt y})  \big\}$, such that the modified polynomials satisfy
\begin{equation}\label{eq:FVDGsuff}
\widetilde {\bf U}_{ij}^n ({\tt x},{\tt y}) 
\in {\mathcal G}, \quad 
\forall ( {\tt x},{\tt y} ) \in {\mathbb S}_{ij} := ( \hat{\mathbb S}_i^{\tt x} \otimes {\mathbb S}_j^{\tt y} ) \cup  ( {\mathbb S}_i^{\tt x} \otimes \hat{\mathbb S}_j^{\tt y} ), 
\end{equation}
where $ \hat{\mathbb S}_i^{\tt x} = \{ \hat {\tt x}_i^{(\nu)}  \}_{\nu=1}^{\tt L}$,
$ \hat{\mathbb S}_i^{\tt y} = \{ \hat {\tt y}_j^{(\nu)}  \}_{\nu=1}^{\tt L}$ are
the $\tt L$-point Gauss-Lobatto quadrature nodes in the intervals $[{\tt x}_{i-\frac{1}{2}},{\tt x}_{i+\frac{1}{2}} ]$, $[{\tt y}_{j-\frac{1}{2}},{\tt y}_{j+\frac{1}{2}} ]$, respectively, 
with $2{\tt L}-3\ge {\tt K}$. 
Let $\widetilde{\bf U}_h^n ( {\tt x},{\tt y} )$ denote the discontinuous piecewise polynomial function
defined by $\widetilde{\bf U}_{ij}^n ({\tt x},{\tt y})$. 
Then we have $\widetilde{\bf U}_h^n  \in {\mathbb W}_h^{\tt K} $, 
because the PP limiter \cite{cheng} only involves element and component wise
convex combination of ${\bf U}_{ij}^n ({\tt x},{\tt y})$ and its cell average.

\vspace{2mm}
\noindent
{\bf Step 2.} Update the cell averages by the scheme
\begin{equation}\label{eq:PP2DMHD:cellaverage}
\bar {\bf U}_{ij}^{n+1}  = \bar {\bf U}_{ij}^{n} + \Delta t_n {\bf L}_{ij} ( \widetilde{\bf U}_h^n  ),
\end{equation}
{\em As shown in Theorem \ref{thm:PP:2DMHD} later, the PP limiting procedure in {Step 1} can ensure the computed
	$\bar {\bf U}_{ij}^{n+1} \in {\mathcal G}$, which meets the condition of performing PP limiting procedure
	in the next time-forward step.}

\vspace{2mm}
\noindent
{\bf Step 3.} Build the discontinuous piecewise polynomial function 
${\bf U}_h^{n+1}$. For our $\mathbb{P}^{\tt K}$-based DG method $({\tt K}\ge 1)$, 
evolve the high-order ``moments'' of the polynomials $\{ {\bf U}_{ij}^{n+1}({\tt x},{\tt y}) \}$ by \eqref{eq:2DDGUh}, in which   
${\bf U}_h^n$ is replaced with $\widetilde{\bf U}_h^n$. 
For a high-order finite volume scheme, 
reconstruct the approximate solution polynomials $\{{\bf U}_{ij}^{n+1}({\tt x},{\tt y})\}$ from the
cell averages $\big\{ \bar{\bf U}_{ij}^{n+1} \big\}$ by a locally 
divergence-free approach (see e.g., \cite{ZhaoTang2017}) such that ${\bf U}_h^{n+1} \in {\mathbb W}_h^{\tt K}  $. The details are omitted here, as these does not affect the PP property of the proposed schemes.

\vspace{2mm}
\noindent
{\bf Step 4.} Set $t_{n + 1}  = t_n  + \Delta t_n$. If $t_{n + 1}  < T_{\tt stop}$, assign $n \leftarrow n+1$ and go to {Step 1}, where the admissibility of $\big\{\bar {\bf U}_{ij}^{n+1}\}$
has been ensured in Step 2. Otherwise, output numerical results and stop.

We now prove the PP property, i.e., show that 
the cell averages $\bar {\bf U}_{ij}^{n+1}$ 
computed by \eqref{eq:PP2DMHD:cellaverage} always belong to $\mathcal G$. 
The discovery of the PP property and the proof are very nontrivial and technical, becoming the most highlighted point of this paper. It is worth emphasizing that using 
the locally divergence-free scheme as the base scheme 
is crucial for achieving the provably PP scheme.  
The locally divergence-free 
property also plays an important role in the proof of the PP property. 

Let $ \{\hat \omega_\nu\}_{\nu=1} ^ {\tt L}$ denote the  
$\tt L$-point Gauss-Lobatto quadrature weights satisfying that 
$\sum_{\nu=1}^{\tt L} \hat\omega_\nu = 1,~\omega_1 = \omega_{\tt L} = \frac{1}{{\tt L}({\tt L}-1)}.$

\begin{theorem} \label{thm:PP:2DMHD}
	If the polynomial vectors $\{\widetilde{\bf U}_{ij}^n({\tt x},{\tt y})\}$ satisfy 
	the condition \eqref{eq:FVDGsuff}, 
	then the scheme \eqref{eq:PP2DMHD:cellaverage} preserves 
	$\bar{\bf U}_{ij}^{n+1} \in {\mathcal G}$ under the CFL-type condition
	\begin{equation}\label{eq:CFL:2DMHD}
	0< \Delta t_n \left( \frac{\alpha_{1,n}^{\tt LF} }{\Delta x} + \frac{ \alpha_{2,n}^{\tt LF} }{\Delta y} \right) 
	 \le   \theta \hat \omega_1 ,
	\end{equation}
	where 	
	$$
	\theta = \frac{1}  { 1+  \max \left\{ \frac{\vartheta_1}{\alpha_{1,n}^{\tt LF}}, \frac{\vartheta_2}{\alpha_{2,n}^{\tt LF}}  \right\}  }, 
	$$
	the parameters $\alpha_{1,n}^{\tt LF}$ and $\alpha_{2,n}^{\tt LF}$ are set to satisfy
	$$ \alpha_{1,n}^{\tt LF} >
	\max_{ i,j,\mu}  \alpha_1 \big(  { \bf U }_{i+\frac12,j}^{\pm,\mu} ,  { \bf U }_{i-\frac12,j}^{\pm,\mu}  \big)
	,\quad \alpha_{2,n}^{\tt LF} >  \max_{ i,j,\mu}  \alpha_2 \big(  { \bf U }_{i,j+\frac12}^{\mu,\pm} ,  { \bf U }_{i,j-\frac12}^{\mu,\pm}  \big),
	$$
	and 
	\begin{align*}%\label{eq:2Dlimitvalues}
		&{\bf U}^{\pm,\mu}_{i+\frac{1}{2},j} := \widetilde{\bf U}_h^n ({\tt x}_{i+\frac12}^\pm,{\tt y}_j^{(\mu)}),\quad
	{\bf U}^{\mu,\pm}_{i,j+\frac{1}{2}} := \widetilde{\bf U}_h^n ({\tt x}_i^{(\mu)},{\tt y}_{j+\frac12}^\pm),
	\\
	& \vartheta_1:=\max_{i,j,\mu} \max \left\{ 
	\frac{ | {\mathscr B}_{1,i+\frac12} ({\tt y}_{j}^{(\mu)}) | } { \sqrt{ \rho_{i+\frac12,j}^{-,\mu} } },
	\frac{ | {\mathscr B}_{1,i-\frac12} ({\tt y}_{j}^{(\mu)}) | } { \sqrt{ \rho_{i-\frac12,j}^{+,\mu} } }
	 \right\},
	 \\
	 & 
	 \vartheta_2:=\max_{i,j,\mu} \max \left\{ 
	 \frac{ | {\mathscr B}_{2,j+\frac12} ({\tt x}_{i}^{(\mu)}) | } { \sqrt{ \rho_{i,j+\frac12}^{\mu,-} } },
	 \frac{ | {\mathscr B}_{2,j-\frac12} ({\tt x}_{i}^{(\mu)}) | } { \sqrt{ \rho_{i,j-\frac12}^{\mu,+} } }
	 \right\}.
	\end{align*}
\end{theorem}

\begin{remark}\label{rem:CFLresonable}
Before the proof, it is worth clarifying that 
the condition \eqref{eq:CFL:2DMHD} is close to the standard CFL condition for the PP DG schemes by Zhang and Shu \cite{zhang2010b}. To this end, we justify that the value of $\theta$ is close to one, because 
${\vartheta_\ell}/{\alpha_{\ell,n}^{\tt LF}}$, $\ell=1,2$, 
are small as supported by the following evidences. 
\begin{enumerate}
	\item For the {\em exact} solution of the system \eqref{eq:MHD}, the divergence-free condition \eqref{eq:2D:BxBy0} implies that, across every cell interface, 
	 the normal component of magnetic field is {\em always continuous}, regardless of the regularity of the solution (e.g., near shocks). 
	 This yields that the two limiting values $ (B_1)_h^n( {\tt x}_{i+1/2} ^+,{\tt y} )$ and $(B_1)_h^n( {\tt x}_{i+1/2} ^- ,{\tt y} )$ approximate  
	 the exact normal magnetic field $B_1( {\tt x}_{i+1/2},{\tt y},t_n )$. 
	 Hence the jump in normal magnetic filed, $| {\mathscr B}_{1,i+\frac12} ({\tt y}_{j}^{(\mu)}) |$, is close to the discretization error and would be very small.  
	 Similar arguments for $| {\mathscr B}_{2,j+\frac12} ({\tt x}_{i}^{(\mu)}) | $. 
	 %Similar arguments for the jump $| {\mathscr B}_{2,j+\frac12} ({\tt x}_{i}^{(\mu)}) | $ is the $\tt y$-direction. 
	\item Note that, even in low density and strongly magnetized region, $\vartheta_1$ and $\vartheta_2$ may be large, however, 
	 the ratio ${\vartheta_\ell}/{\alpha_{\ell,n}^{\tt LF}}$, which involved in the definition of $\theta$, is usually small. 
	 In fact, 
	$|{\bf B}|/\sqrt{\rho}$ can be controlled by 
	$\alpha_{\ell,n}^{\tt LF}$ because  
		\begin{align*}
	\frac{|{\bf B}|}{\sqrt{\rho}} &\le \frac{1}{\sqrt{2}} \sqrt{ 2 \max \bigg\{ {\mathscr C}_s^2, \frac{|{\bf B}|^2}{{\rho}} \bigg\} }
	\\		& = \frac{1}{\sqrt{2}} \left[ \mathscr{C}_s^2 + \frac{ |{\bf B}|^2}{\rho} + \sqrt{ \left( \mathscr{C}_s^2 + \frac{ |{\bf B}|^2}{\rho} \right)^2 - 4 \frac{ \mathscr{C}_s^2 |{\bf B}|^2}{\rho}  }  \right]^\frac12 \le {\mathscr C}_\ell \le |v_\ell| +  {\mathscr C}_\ell. 
	\end{align*}
	\item Some numerical evidences given in Section \ref{sec:examples} (see Figs.\ \ref{fig:theta} and \ref{fig:theta2}) show that  ${\vartheta_\ell}/{\alpha_{\ell,n}^{\tt LF}}$, $\ell=1,2,$ are very small, and $\theta$ is very close to one, in the %extreme
	 tested cases involving strong discontinuity, low density and strong magnetic field. 
\end{enumerate}
Note that our CFL condition \eqref{eq:CFL:2DMHD} is sufficient, but may be not necessary especially for those mild problems. It is certainly possible to estimate sharper CFL condition. 
% with other techniques
%, as the numerical tests show our PP schemes still work with larger time step-size. 
\end{remark}

We are now in the position to present the proof of Theorem \ref{thm:PP:2DMHD}.

\begin{proof}
	Using the exactness of the Gauss-Lobatto quadrature rule with $\tt L$ nodes and the Gauss quadrature rule with $\tt Q$ nodes for the polynomials of degree $\tt K$,
	one can derive (cf. \cite{zhang2010b} for more details) that
	\begin{equation} \label{eq:U2Dsplit}
	\begin{split}
	\bar{\bf U}_{ij}^n
	&= \frac{\lambda_1}{\lambda}  \sum \limits_{\nu = 2}^{{\tt L}-1} \sum \limits_{\mu = 1}^{\tt Q}  \hat \omega_\nu \omega_\mu  \widetilde{\bf U}_{ij}^n\big(\hat {\tt x}_i^{(\nu)},{\tt y}_j^{(\mu)}\big) + \frac{\lambda_2}{\lambda} \sum \limits_{\nu = 2}^{{\tt L}-1} \sum \limits_{\mu = 1}^{\tt Q}  \hat \omega_\nu \omega_\mu  \widetilde{\bf U}_{ij}^n\big( {\tt x}_i^{(\mu)},\hat {\tt y}_j^{(\nu)} \big) \\
	&\quad + \frac{\lambda_1 \hat \omega_1}{\lambda} \sum \limits_{\mu = 1}^{\tt Q}  \omega_\mu \left( {\bf U}_{i-\frac{1}{2},j}^{+,\mu} +
	{\bf U}_{i+\frac{1}{2},j}^{-,\mu} \right)
	+ \frac{\lambda_2 \hat \omega_1}{\lambda} \sum \limits_{\mu = 1}^{\tt Q}  \omega_\mu \left(  {\bf U}_{i,j-\frac{1}{2}}^{\mu,+} +
	{\bf U}_{i,j+\frac{1}{2}}^{\mu,-} \right) ,
	\end{split}
	\end{equation}
	where $\hat \omega_1 = \hat \omega_{\tt L}$ is used, and $\lambda_1 = \frac{ \alpha_{1,n}^{\tt LF} \Delta t_n } { \Delta x}, \lambda_2 = \frac{ \alpha_{2,n}^{\tt LF} \Delta t_n }{\Delta y}, \lambda = \lambda_1+ \lambda_2 $. 
	After substituting \eqref{eq:U2Dsplit} into \eqref{eq:PP2DMHD:cellaverage}, we rewrite the scheme  \eqref{eq:PP2DMHD:cellaverage} by technical arrangement into the form
	\begin{align}	\label{eq:2DMHD:split:proof}
		&\bar{\bf U}_{ij}^{n+1}
		= %(1-2\hat \omega_1) {\bf \Xi}_1 
		\sum \limits_{\nu = 2}^{{\tt L}-1} \hat \omega_\nu {\bf \Xi}_\nu 
		+  2	\lambda {\bf \Xi}_1 		
		+ 
		2( \hat \omega_1 - \lambda )  {\bf \Xi}_{\tt L}  -  {\bf S}_1 - {\bf S}_2 ,
	\end{align}
	where ${\bf \Xi}_1 = \frac12 \left( {\bf \Xi}_- +  {\bf \Xi}_+ \right)$, and 
	\begin{align*}
		\begin{split}
		&{\bf \Xi}_\nu
		=  
		\frac{\lambda_1}{\lambda}  \sum \limits_{\mu = 1}^{\tt Q} \omega_\mu  \widetilde{\bf U}_{ij}^n \big(\hat {\tt x}_i^{(\nu)},{\tt y}_j^\mu\big) + \frac{\lambda_2}{\lambda}
		\sum \limits_{\mu = 1}^{\tt Q} \omega_\mu  \widetilde{\bf U}_{ij}^n \big( {\tt x}_i^{(\mu)},\hat {\tt y}_j^{(\nu)} \big),\quad 
		2 \le \nu \le {\tt L}-1,
		\\[0.5mm] 
			&{\bf \Xi}_{\tt L} = %\left(  \frac{\alpha_{1}^{\tt LF} }{\Delta x} + \frac{\alpha_{2}^{\tt LF}}{\Delta y} \right)^{-1}
			\frac{1}{ 2\lambda }
			\sum\limits_{\mu=1}^{\tt Q} { \omega _\mu  }
			\bigg(
			\lambda_1 \left(
			{ \bf U }_{i+\frac12,j}^{-,\mu} 
			+ { \bf U }_{i-\frac12,j}^{+,\mu}  
			\right)
			+ \lambda_2 \left(
			{ \bf U }_{i,j+\frac12}^{\mu,-} 
			+ { \bf U }_{i,j-\frac12}^{\mu,+}  
			\right)
			\bigg),
		\end{split}
%		\\[0.5mm] \nonumber
%		\begin{split}
%			&{\bf \Xi}_- = %\left(  \frac{\alpha_{1}^{\tt LF} }{\Delta x} + \frac{\alpha_{2}^{\tt LF}}{\Delta y} \right)^{-1}
%			\frac{1}{ 2 \lambda }
%			\sum\limits_{\mu=1}^{\tt Q} { \omega _\mu  }
%			\left[
%			\lambda_1 \left(
%			{ \bf U }_{i+\frac12,j}^{-,\mu} -  \frac{ {\bf F}_1 ( { \bf U }_{i+\frac12,j}^{-,\mu}  )  }{ \alpha_{1}^{\tt LF} }
%			+ { \bf U }_{i-\frac12,j}^{-,\mu}  +   \frac{ {\bf F}_1 ( { \bf U }_{i-\frac12,j}^{-,\mu}  )  }{ \alpha_{1}^{\tt LF} }
%			\right)\right.
%			\\%&\quad \quad \quad \quad \quad
%			&\qquad + \left.
%			\lambda_2 \left(
%			{ \bf U }_{i,j+\frac12}^{\mu,-} -  \frac { {\bf F}_2 (  { \bf U }_{i,j+\frac12}^{\mu,-}  )  } { \alpha_{2}^{\tt LF} }
%			+ { \bf U }_{i,j-\frac12}^{\mu,-}  +  \frac { {\bf F}_2 (  { \bf U }_{i,j-\frac12}^{\mu,-} )  } { \alpha_{2}^{\tt LF} }
%			\right)
%			\right],
%		\end{split}
		\\[0.5mm] \nonumber
		\begin{split}
			&{\bf \Xi}_\pm = %\left(  \frac{\alpha_{1}^{\tt LF} }{\Delta x} + \frac{\alpha_{2}^{\tt LF}}{\Delta y} \right)^{-1}
			\frac{1}{ 2 \lambda}
			\sum\limits_{\mu=1}^{\tt Q} { \omega _\mu  }
			\left[
			\lambda_1 \left(
			{ \bf U }_{i+\frac12,j}^{\pm,\mu} -  \frac{ {\bf F}_1 ( { \bf U }_{i+\frac12,j}^{\pm,\mu}  )  }{ \alpha_{1,n}^{\tt LF} }
			+ { \bf U }_{i-\frac12,j}^{\pm,\mu}  +   \frac{ {\bf F}_1 ( { \bf U }_{i-\frac12,j}^{\pm,\mu}  )  }{ \alpha_{1,n}^{\tt LF} }
			\right)\right.
			\\%&\quad \quad \quad \quad \quad
			&\qquad + \left.
			\lambda_2 \left(
			{ \bf U }_{i,j+\frac12}^{\mu,\pm} -  \frac { {\bf F}_2 (  { \bf U }_{i,j+\frac12}^{\mu,\pm}  )  } { \alpha_{2,n}^{\tt LF} }
			+ { \bf U }_{i,j-\frac12}^{\mu,\pm}  +  \frac { {\bf F}_2 (  { \bf U }_{i,j-\frac12}^{\mu,\pm} )  } { \alpha_{2,n}^{\tt LF} }
			\right)
			\right],
		\end{split}
	\\
			 &  {\bf S}_1 = 
	 \frac{\Delta t_n}{\Delta x} \sum\limits_{\mu =1}^{\tt Q}  \omega_\mu  \bigg( {\mathscr B}_{1,i+\frac12} ({\tt y}_{j}^{(\mu)}) 
	{\bf S} \big( {\bf U}_{i+\frac12,j}^{-,\mu} \big) 
	+ {\mathscr B}_{1,i-\frac12} ({\tt y}_{j}^{(\mu)})  {\bf S}  \big( {\bf U}_{i-\frac12,j}^{+,\mu} \big) 
	\bigg),
	\\
	& {\bf S}_2 =  \frac{\Delta t_n}{\Delta y} \sum\limits_{\mu =1}^{\tt Q} \omega_\mu 
	\bigg(
	{\mathscr B}_{2,j+\frac12} ({\tt x}_{i}^{(\mu)}) 
	{\bf S}  \big( {\bf U}_{i,j+\frac12}^{\mu,-} \big) 
	+ {\mathscr B}_{2,j-\frac12} ({\tt x}_{i}^{(\mu)})  {\bf S} \big( {\bf U}_{i,j-\frac12}^{\mu,+} \big) 
	\bigg).
	\end{align*}

	Using Lemma \ref{lem:rho} gives ${\bf \Xi}_k \in {\mathcal G}_\rho$, 
	i.e., the first component of ${\bf \Xi}_k$ is positive, $k=1,2,\cdots,{\tt L}$. 
	Because the first components of ${\bf S}_1$ and ${\bf S}_2$ are both zeros, 
	we know from \eqref{eq:2DMHD:split:proof} that the density $\bar \rho_{ij}^{n+1}$ is a convex combination 
	of the first components of ${\bf \Xi}_k$, $k=1,2,\cdots,{\tt L}$. Therefore, $\bar \rho_{ij}^{n+1} >0$.

	For any $ {\bf v}^*,{\bf B}^* \in {\mathbb R}^3$, we turn to show that $\bar{\bf U}_{ij}^{n+1} \cdot {\bf n}^* + \frac{|{\bf B}^*|^2}{2} >0$. 
	Note that the condition \eqref{eq:FVDGsuff} implies ${\bf \Xi}_\nu \in {\mathcal G},~2 \le \nu \le {\tt L}-1$,  
	by the convexity of $\mathcal G$. According to Lemma \ref{theo:eqDefG}, we have  
	$$
	{\bf \Xi}_\nu \cdot {\bf n}^* + \frac{|{\bf B}^*|^2}{2}  >0,\quad 2 \le \nu \le {\tt L}-1. 
	$$
	It follows from \eqref{eq:2DMHD:split:proof} that 
	\begin{align}\label{eq:22}
		\bar{\bf U}_{ij}^{n+1} \cdot {\bf n}^* + \frac{|{\bf B}^*|^2}{2}  
		& = 
		\sum \limits_{\nu = 2}^{{\tt L}-1} \hat \omega_\nu \left( {\bf \Xi}_\nu \cdot {\bf n}^* + \frac{|{\bf B}^*|^2}{2} \right)
		+ \Pi_1 + \Pi_2 
		\ge  \Pi_1 + \Pi_2,
	\end{align}
	where 
		\begin{align*}
		&\Pi_1 := 2	\lambda \left( {\bf \Xi}_1 \cdot {\bf n}^* + \frac{|{\bf B}^*|^2}{2}  \right),
		\\
		&\Pi_2 := 2( \hat \omega_1 - \lambda )  \left( {\bf \Xi}_{\tt L}  \cdot {\bf n}^* + \frac{|{\bf B}^*|^2}{2} \right) 
		- ( {\bf S}_1  +  {\bf S}_2 ) \cdot {\bf n}^*. 
	\end{align*}
	In the following, we estimate the lower bounds of $\Pi_1$ and $\Pi_2$ respectively. 
	
		Let first consider $\Pi_1$ and split it as 
	$\Pi_1=\Pi_1^- + \Pi_1^+,$ 
	where 
	\begin{equation*}
	\begin{split}
		\Pi_1^\pm  &= \lambda \left( {\bf \Xi}_\pm \cdot {\bf n}^* + \frac{ |{\bf B}^*|^2 }{2} \right)	
		\\ 
		& =			
		\sum\limits_{\mu=1}^{\tt Q} \frac{ \omega _\mu  }{ 2 }
		\left\{ 
		\lambda_1 \left[ \left(
		{ \bf U }_{i+\frac12,j}^{\pm,\mu} -  \frac{ {\bf F}_1 ( { \bf U }_{i+\frac12,j}^{\pm,\mu}  )  }{ \alpha_{1,n}^{\tt LF} }
		+ { \bf U }_{i-\frac12,j}^{\pm,\mu}  +   \frac{ {\bf F}_1 ( { \bf U }_{i-\frac12,j}^{\pm,\mu}  )  }{ \alpha_{1,n}^{\tt LF} }
		\right) \cdot {\bf n}^* +  |{\bf B}^*|^2  \right]
		\right.
		\\ 
		& \quad  + \left. 
		\lambda_2  \left[ \left(
		{ \bf U }_{i,j+\frac12}^{\mu,\pm} -  \frac { {\bf F}_2 (  { \bf U }_{i,j+\frac12}^{\mu,\pm}  )  } { \alpha_{2,n}^{\tt LF} }
		+ { \bf U }_{i,j-\frac12}^{\mu,\pm}  +  \frac { {\bf F}_2 (  { \bf U }_{i,j-\frac12}^{\mu,\pm} )  } { \alpha_{2,n}^{\tt LF} }
		\right) \cdot {\bf n}^* +  |{\bf B}^*|^2 
		\right] \right\}
		\\
		& \overset{\eqref{eq:MHD:LLFsplit}}{>} 
		\sum\limits_{\mu=1}^{\tt Q} \frac{ \omega _\mu  }{ 2 }
		\left( 
		 -\lambda_1 
		  \frac{  (B_1)_{i+\frac12,j}^{\pm,\mu} -   (B_1)_{i-\frac12,j}^{\pm,\mu}  }{ \alpha_{1,n}^{\tt LF} } 
		   -
		\lambda_2  
		  \frac {   (B_2)_{i,j+\frac12}^{\mu,\pm} -  (B_2)_{i,j-\frac12}^{\mu,\pm}   } { \alpha_{2,n}^{\tt LF} }  
		 \right) ({\bf v}^* \cdot {\bf B}^*)
		 \\
			& = - \frac{ \Delta t_n }{2}	\sum\limits_{\mu=1}^{\tt Q}  \omega _\mu  
		\left( 
		\frac{  (B_1)_{i+\frac12,j}^{\pm,\mu} -   (B_1)_{i-\frac12,j}^{\pm,\mu}  }{ \Delta x} 
		+  
		\frac {   (B_2)_{i,j+\frac12}^{\mu,\pm} -  (B_2)_{i,j-\frac12}^{\mu,\pm}   } { \Delta y }  
		\right) ({\bf v}^* \cdot {\bf B}^*).
	\end{split}
		\end{equation*}
	Thus we have 
	\begin{equation}\label{eq:ddffa}
	\Pi_1 > - \Delta t_n \big( {\rm div}_{ij} {\bf B}^n  \big) ({\bf v}^* \cdot {\bf B}^*),
	\end{equation}
	where ${\rm div}_{ij} {\bf B}^n$ is the discrete divergence (\cite{Wu2017a}) of $\widetilde{\bf B}^n_h({\tt x},{\tt y})$ defined by 
	 \begin{align*}%\label{eq:DivB:cst1}
	 	&
	 	\mbox{\rm div} _{ij} {\bf B}^n = \frac1 {\Delta x}  {\sum\limits_{\mu=1}^{\tt Q} \omega_\mu \left(   \overline{(B_1)}_{i+\frac{1}{2},j}^{\mu} 
	 		- \overline{(B_1)}_{i-\frac{1}{2},j}^{\mu}  \right)}  + \frac1 {\Delta y} {\sum \limits_{\mu=1}^{\tt Q} \omega_\mu \left(  \overline{ ( B_2)}_{i,j+\frac{1}{2}}^{\mu} 
	 		-  \overline{ ( B_2)}_{i,j-\frac{1}{2}}^{\mu}    \right)},
	 \end{align*}
 	 with 
 	 	 $
 	 \overline{(B_1)}_{i+\frac{1}{2},j}^{\mu} = \frac12 \big( (B_1)_{i+\frac{1}{2},j}^{-,\mu} + (B_1)_{i+\frac{1}{2},j}^{+,\mu} \big)$ and  
 	 $
 	 \overline{ ( B_2) }_{i,j+\frac{1}{2}}^{\mu} = \frac12 \big( ( B_2)_{i,j+\frac{1}{2}}^{\mu,-} + ( B_2)_{i,j+\frac{1}{2}}^{\mu,+} \big).
 	 $
	
 	Let then consider $\Pi_2$. Using the inequality \eqref{eq:widelyusedIEQ} gives
{\small\begin{align*}%\label{eq:dfaafff}
 {\bf S}_1 \cdot {\bf n}^* 
& \le \frac{\Delta t_n}{\Delta x} \sum\limits_{\mu =1}^{\tt Q}  \omega_\mu  
 \left[ 
\frac{ | {\mathscr B}_{1,i+\frac12} ({\tt y}_{j}^{(\mu)}) | } { \sqrt{ \rho_{i+\frac12,j}^{-,\mu} } }
 \left( {\bf U}_{i+\frac12,j}^{-,\mu} \cdot {\bf n}^* + \frac{ |{\bf B}^*|^2 }{2}  \right) 
 - {\mathscr B}_{1,i+\frac12} ({\tt y}_{j}^{(\mu)}) ( {\bf v}^* \cdot {\bf B}^* ) \right.
\\
& \quad \left. +
\frac{  |{\mathscr B}_{1,i-\frac12} ({\tt y}_{j}^{(\mu)}) | } { \sqrt{ \rho_{i-\frac12,j}^{+,\mu} } }
\left( {\bf U}_{i-\frac12,j}^{+,\mu}  \cdot {\bf n}^* + \frac{ |{\bf B}^*|^2 }{2}  \right) 
- {\mathscr B}_{1,i-\frac12} ({\tt y}_{j}^{(\mu)}) ( {\bf v}^* \cdot {\bf B}^* ) \right]
\\
& \le \frac{\Delta t_n}{\Delta x} \sum\limits_{\mu =1}^{\tt Q}  \omega_\mu  
\left[
\vartheta_1 \Pi_{\mu}^{ x} - 
\big( {\mathscr B}_{1,i+\frac12} ({\tt y}_{j}^{(\mu)})  +  {\mathscr B}_{1,i-\frac12} ({\tt y}_{j}^{(\mu)}) \big)  ( {\bf v}^* \cdot {\bf B}^* ) 
  \right],
\end{align*}}
where  
$$
\Pi_{\mu}^{ x} :=( {\bf U}_{i+\frac12,j}^{-,\mu} + {\bf U}_{i-\frac12,j}^{+,\mu}    ) \cdot {\bf n}^*
+ |{\bf B}^*|^2>0.
$$
Similarly, we have 
 \begin{align*}%\label{eq:dfaafff}
	{\bf S}_2 \cdot {\bf n}^* 
 \le \frac{\Delta t_n}{\Delta y} \sum\limits_{\mu =1}^{\tt Q}  \omega_\mu  
	\left[
	\vartheta_2 
	\Pi_{\mu}^{ y}  -  \big(  {\mathscr B}_{2,j+\frac12} ({\tt x}_{i}^{(\mu)})  +  {\mathscr B}_{2,j-\frac12} ({\tt x}_{i}^{(\mu)}) \big)  ( {\bf v}^* \cdot {\bf B}^* ) 
	\right],
\end{align*} 	
 where  
 $$
 \Pi_{\mu}^{ y} := ( {\bf U}_{i,j+\frac12}^{\mu,-} + {\bf U}_{i,j-\frac12}^{\mu,+}    ) \cdot {\bf n}^*
 + |{\bf B}^*|^2 > 0.
 $$
 Therefore, 
   \begin{align}\label{eq:dfaasdds}
 	({\bf S}_1 + {\bf S}_2) \cdot {\bf n}^* \le \Pi_{s1} + \Pi_{s2} - \Delta t_n {\Pi}_{s3}  ( {\bf v}^* \cdot {\bf B}^* ),  
\end{align} 
where 
$$
\Pi_{s1} = \vartheta_1 \frac{\Delta t_n}{\Delta x} \sum\limits_{\mu =1}^{\tt Q}  \omega_\mu  
 \Pi_{\mu}^{ x}, \qquad \Pi_{s2} = \vartheta_2 \frac{\Delta t_n}{\Delta y} \sum\limits_{\mu =1}^{\tt Q}  \omega_\mu  
 \Pi_{\mu}^{ y},
$$ 
and
\begin{align*}
	{\Pi}_{s3} &=  \sum\limits_{\mu =1}^{\tt Q}  \omega_\mu  
	\Bigg(  \frac{ {\mathscr B}_{1,i+\frac12} ({\tt y}_{j}^{(\mu)})    }{\Delta x} +  
	  \frac{{\mathscr B}_{1,i-\frac12} ({\tt y}_{j}^{(\mu)})    }{\Delta x} + 
	\frac{  {\mathscr B}_{2,j+\frac12} ({\tt x}_{i}^{(\mu)})    }{\Delta y} 
	+ 
	\frac{    {\mathscr B}_{2,j-\frac12} ({\tt x}_{i}^{(\mu)})  }{\Delta y} \Bigg).
\end{align*}
Note that \eqref{eq:CFL:2DMHD} implies $\lambda \le \theta \hat \omega_1$, thus  
$$
\vartheta_\ell \le  \alpha_{\ell ,n}^{\tt LF} \max \left\{ {\vartheta_1}/{\alpha_{1,n}^{\tt LF}}, {\vartheta_2}/{\alpha_{2,n}^{\tt LF}}  \right\} 
= \alpha_{\ell,n}^{\tt LF} (\theta^{-1}-1) \le \alpha_{\ell,n}^{\tt LF} ( \hat \omega_1 \lambda^{-1} - 1 ), \quad \ell=1,2. 
$$
It follows that 
\begin{equation}\label{eq:dfaasdds3}
\Pi_{s1} \le ( \hat \omega_1 \lambda^{-1} - 1 ) \lambda_1 \sum\limits_{\mu =1}^{\tt Q}  \omega_\mu  
\Pi_{\mu}^{ x},\quad \Pi_{s2} \le ( \hat \omega_1 \lambda^{-1} - 1 ) \lambda_2 \sum\limits_{\mu =1}^{\tt Q}  \omega_\mu  
\Pi_{\mu}^{ y}.
\end{equation}
Combining \eqref{eq:dfaasdds} and \eqref{eq:dfaasdds3}, we obtain the estimate for $\Pi_2$
\begin{align*}
	\Pi_2  & \overset{ \eqref{eq:dfaasdds} }{\ge}  
	2( \hat \omega_1 - \lambda )  \left( {\bf \Xi}_{\tt L}  \cdot {\bf n}^* + \frac{|{\bf B}^*|^2}{2} \right) - 
	(\Pi_{s1}+\Pi_{s2}) +  \Delta t_n {\Pi}_{s3}  ( {\bf v}^* \cdot {\bf B}^* )
	\\
	\begin{split}
	& \overset{ \eqref{eq:dfaasdds3} }{\ge} 
	2( \hat \omega_1 - \lambda )  \left( {\bf \Xi}_{\tt L}  \cdot {\bf n}^* + \frac{|{\bf B}^*|^2}{2} \right) - 
	( \hat \omega_1 \lambda^{-1} - 1 ) 
	\\
	& \qquad \times 
	 \sum\limits_{\mu =1}^{\tt Q} \big( \lambda_1 \omega_\mu  
	\Pi_{\mu}^{ x} + \lambda_2 \omega_\mu  
	\Pi_{\mu}^{ y} \big) +  \Delta t_n {\Pi}_{s3}  ( {\bf v}^* \cdot {\bf B}^* ) 
	\end{split}
	\\
	& =\Delta t_n {\Pi}_{s3}  ( {\bf v}^* \cdot {\bf B}^* ).
\end{align*}
Note that $\Pi_{s3}$ can be rewritten as 
\begin{align*}
	{\Pi}_{s3} 
	& 
	= \sum\limits_{\mu =1}^{\tt Q}  \omega_\mu  
	\Bigg( 
	\frac{ \overline{(B_1)}_{i+\frac12,j}^\mu - (B_1)_{i+\frac12,j}^{-,\mu}   }{  \Delta x }
	+ 
	\frac{   (B_1)_{i-\frac12,j}^{+,\mu} - \overline{(B_1)}_{i-\frac12,j}^\mu   }{  \Delta x }
	+ \\
	& \quad  
	\frac{ \overline{(B_2)}_{i,j+\frac12}^\mu - (B_2)_{i,j+\frac12}^{\mu,-}   }{  \Delta y }
	+ 
	\frac{   (B_2)_{i,j-\frac12}^{\mu,+} - \overline{(B_2)}_{i,j-\frac12}^\mu   }{  \Delta y }
	\Bigg)
	%\\
	 = {\rm div}_{ij} {\bf B}^n  - {\rm div}_{ij}^{\tt in} {\bf B}^n,
\end{align*}
where  
\begin{equation*}{\small
\begin{aligned}	
 {\rm div}_{ij}^{\tt in} {\bf B}^n & =  
\frac1 {\Delta x}  {\sum\limits_{\mu=1}^{\tt Q} \omega_\mu \left(   {(B_1)}_{i+\frac{1}{2},j}^{-,\mu} 
	- {(B_1)}_{i-\frac{1}{2},j}^{+,\mu}  \right)}  + \frac1 {\Delta y} {\sum \limits_{\mu=1}^{\tt Q} \omega_\mu \left(  { ( B_2)}_{i,j+\frac{1}{2}}^{\mu,-} 
	-  { ( B_2)}_{i,j-\frac{1}{2}}^{\mu,+}    \right)}	
\\
& = \frac{1}{\Delta x} \int_{ {\tt y}_{j-\frac12} }^{ {\tt y}_{j+\frac12} }  
\Big( (\widetilde{B_1})_{ij}^n ({\tt x},{\tt y})\Big|_{ {\tt x} = {\tt x}_{i-\frac12} }^{ 
 {\tt x}= {\tt x}_{i+\frac12}  } \Big)
 d {\tt y} 
 + 
 \frac{1}{\Delta y} \int_{ {\tt x}_{i-\frac12} }^{ {\tt x}_{i+\frac12} } 
 \Big( (\widetilde{B_2})_{ij}^n ({\tt x},{\tt y})\Big|_{ {\tt y} = {\tt y}_{j-\frac12} }^{ 
 	{\tt y}= {\tt y}_{j+\frac12}  } \Big)
 d {\tt x} 
 \\
 & = \frac{1}{\Delta x \Delta y} \int_{I_{ij}}
 \nabla \cdot \widetilde{\bf B}_{ij}^n ( {\tt x},{\tt y} ) d {\tt x} d {\tt y} = 0.
 \end{aligned}}
\end{equation*}
In the above identity we have used 
the exactness of $\tt Q$-point Gauss quadrature rule for the polynomials of degree $\tt K$, 
the divergence theorem and the locally divergence-free property of the polynomial vector $\widetilde{\bf B}_{ij}^n( {\tt x},{\tt y} )$. 
Therefore, we obtain 
\begin{equation}\label{key111}
\Pi_2 \ge \Delta t_n  ({\rm div}_{ij} {\bf B}^n)  ( {\bf v}^* \cdot {\bf B}^* ).
\end{equation} 

Substituting the estimates \eqref{eq:ddffa} and \eqref{key111} into the inequality \eqref{eq:22} gives 
$$
\bar{\bf U}_{ij}^{n+1} \cdot {\bf n}^* + \frac{|{\bf B}^*|^2}{2}  > 0,\quad \forall {\bf v}^*,{\bf B}^*\in {\mathbb R}^3,
$$
which along with $\bar\rho_{ij}^{n+1}>0$ imply $\bar{\bf U}_{ij}^{n+1} \in {\mathcal G}$ by Lemma \ref{theo:eqDefG}. The proof is completed. 
\end{proof}

\begin{remark}
There are two features 
in our PP schemes: the locally divergence-free spatial discretization and the penalty-type terms discretized from the GP source term. 
The former leads to zero divergence within each cell, while the latter 
controls the divergence-error across the cell interfaces.  
The proof of Theorem \ref{thm:PP:2DMHD} shows 
 that, thanks to these two features, the PP property is obtained without requiring the discrete divergence-free condition in \cite{Wu2017a}, 
which was proposed for the conservative schemes without penalty-type terms. 	
%The divergence-error within each cell is zero due to the locally divergence-free spatial discretization, 
%and the divergence-error across the cell interface is 
%well controlled by several 
%penalty-type terms discretized from the Godunov source term. 
%As one can see from the proof of Theorem \ref{thm:PP:2DMHD}, by taking this two features into account, the PP property is obtained without requiring the discrete divergence-free condition in \cite{Wu2017a}, 
%which was proposed for the conservative schemes with no penalty-type terms.    
\end{remark}

\begin{remark}
Theorem \ref{thm:PP:2DMHD} still holds if the condition \eqref{eq:FVDGsuff} is replaced with
\begin{equation}\label{eq:2Dcondition2}
\begin{cases}
{\bf U}_{i+\frac{1}{2},j}^{\pm,\mu},~{\bf U}_{i,j+\frac{1}{2}}^{\mu,\pm} \in {\mathcal G},\quad \forall~i,j,\mu,
\\ 
\frac{ \bar{\bf U}_{ij}^n - \hat \omega_1\sum \limits_{\mu = 1}^{\tt Q}  \frac{ \omega_\mu }{ {\lambda}}  \left(   \lambda_1 \big( {\bf U}_{i-\frac{1}{2},j}^{+,\mu} +
	{\bf U}_{i+\frac{1}{2},j}^{-,\mu}  \big) + \lambda_2 \big( {\bf U}_{i,j-\frac{1}{2}}^{\mu,+} +
	{\bf U}_{i,j+\frac{1}{2}}^{\mu,-} \big)   \right) }{1-2\hat \omega_1} \in {\mathcal G},\quad \forall~i,j .
\end{cases}
\end{equation}
In other words, \eqref{eq:2Dcondition2} gives 
a sufficient condition for preserving positivity in those high-order finite volume methods \eqref{eq:PP2DMHD:cellaverage} that only reconstruct 
 limiting values ${\bf U}_{i+\frac{1}{2},j}^{\pm,\mu}, {\bf U}_{i,j+\frac{1}{2}}^{\mu,\pm}$ instead of polynomials
$\{{\bf U}_{ij}^n({\tt x},{\tt y})\}$. 
The PP limiter in \cite{cheng} can also be revised to meet the 
condition \eqref{eq:2Dcondition2}, see e.g., \cite{zhang2011b}. 
\end{remark}

\begin{remark}
	All the above analyses are focused on the first-order Euler forward time discretization. One can also use strong stability
	preserving (SSP) high-order time discretizations (see e.g., \cite{Gottlieb2009})
	to solve the ODE system $\frac{d}{dt}{\bf U}_h = {\bf L} ({\bf U}_h)$. 
	For example, the third order SSP Runge-Kutta method reads 
	\begin{equation}\label{eq:SSP3}
	\begin{split}
	& {\bf U}_h^{*} = \widetilde {\bf U}_h^{n}+\Delta t_n {\bf L} ( \widetilde {\bf U}_h^{n} ),
	\\
	& {\bf U}_h^{**} = \frac34 \widetilde {\bf U}_h^{n}
	+ \frac14 \Big(  \widetilde {\bf U}_h^{*}+\Delta t_n {\bf L} ( \widetilde {\bf U}_h^{*} ) \Big),
	\\
	&{\bf U}_h^{n+1} = \frac13 \widetilde {\bf U}_h^{n}
	+ \frac23 \Big(  \widetilde {\bf U}_h^{**}+\Delta t_n {\bf L} ( \widetilde {\bf U}_h^{**} ) \Big),
	\end{split}
	\end{equation}
	where the numerical solutions with ``$\sim$'' at above denote the PP limited 
	solutions.  
	Since a SSP method is a convex combination of the Euler forward method, the PP property of the full scheme also holds 
	according to the convexity of $\mathcal G$.
\end{remark}

\section{Numerical Experiments}\label{sec:examples}

This section conducts numerical experiments on several 2D challenging 
 MHD problems  with either strong discontinuities, low plasma-beta $\beta=2p/|{\bf B}|^2$, or low density or pressure, to further demonstrate our theoretical analysis, as well as  
the accuracy, high-resolution and robustness of the proposed PP DG methods.  
Without loss of generality, we focus on the proposed PP third-order (${\mathbb P}^2$-based) DG methods with the third-order SSP Runge-Kutta time discretization \eqref{eq:SSP3}. Unless otherwise stated, all 
the computations are restricted to the EOS \eqref{eq:EOS} with the adiabatic index $\gamma=\frac53$, 
and the CFL number is set as 0.15.

\begin{expl}[Smooth problems]\label{example:accuracy} We first test two smooth problems to check the accuracy of our
scheme. The first problem is similar to the one in \cite{zhang2010b}. 
The exact solution of this problem is 
$$
(\rho,{\bf v},p,{\bf B})({\tt x},{\tt y},t) 
= \big( 1+0.99\sin({\tt x}+{\tt y}-2t),~1,~1,~0,~1,~0.1,~0.1,~0\big), 
$$
which describes a MHD sine wave propagating with $\gamma=1.4$ and low density. 
Table \ref{tab:acc1} lists the numerical errors at $t=0.1$ 
in the numerical density and the corresponding convergence rates 
for the PP third-order DG method at different grid resolutions. 
The results show that the expected convergence order is achieved.

\begin{table}[htbp]\label{tab:acc1}
	\centering
	\caption{\small First problem of Example \ref{example:accuracy}:
		Numerical errors at $t=0.1$ in the density and corresponding convergence rates for
		the PP third-order DG method at
		different grid resolutions.
	}
	\label{tab:order}
	\begin{tabular}{c|c|c|c|c|c|c}
		\hline
		~Mesh~ 
		&~$l^1$-error~&~order~&~$l^2$-error~&~order~&$l^\infty$-error&~order~ \\
		\hline
		$15 \times 15$ &  3.45e-2 & --          & 7.05e-3 & -- & 6.08e-3 & -- \\
		$30 \times 30$ &  4.79e-3 & 2.85   & 1.01e-3&  2.81 & 9.15e-4 & 2.73 \\
		$60 \times 60$ &  6.80e-4 & 2.82    & 1.50e-4& 2.75 & 1.38e-4 & 2.73\\
		$120 \times 120$ & 9.19e-5 & 2.89    & 2.07e-5& 2.85 & 1.91e-5 & 2.86 \\
		$240 \times 240$ & 1.16e-5 & 2.99 &  2.64e-6& 2.97 & 2.44e-6 & 2.97 \\
		$480 \times 480$ &  1.45e-6 & 3.00  &3.31e-7& 3.00 & 3.06e-7 & 3.00\\
		\hline
	\end{tabular}
\end{table}

The second problem is the smooth vortex problem \cite{Christlieb} 
with nonzero magnetic field and extremely low pressure. The initial condition is a mean flow 
$$(\rho,{\bf v},p,{\bf B})({\tt x},{\tt y},0) 
=(1, 1, 1, 0, 1, 0, 0, 0),
$$
with vortex perturbations on $v_1, v_2, B_1, B_2,$ and $p$: 
\begin{gather*}
(\delta v_1, \delta v_2)=\frac{\mu}{ \sqrt{2} \pi} {\rm e}^{0.5(1-r^2)} (-{\tt y},{\tt x}),
\quad 
(\delta B_1,\delta B_2) = \frac{\mu}{2\pi} {\rm e}^{0.5(1-r^2)} (-{\tt y},{\tt x}),
\\
\delta p = -\frac{ \mu^2 (1+r^2)  }{8 \pi^2} {\rm e}^{1-r^2},
%\frac{ \mu^2 (1-r^2) - \kappa^2 }{8 \pi^2} {\rm e}^{1-r^2},
\end{gather*}
where $r=\sqrt{{\tt x}^2+{\tt y}^2}$. 
The computational domain is taken as $[-10,10]^2$, and periodic boundary conditions are used. We set %the adiabatic index $\gamma=\frac53$, and 
the vortex strength 
$\mu=5.389489439$ %and $\kappa=\sqrt{2} \mu$ 
such that the lowest pressure in the center of the vortex is about $5.3 \times 10^{-12}$. 
As a result,  
 our DG method does not work without performing the PP limiting procedure, 
 as the condition \eqref{eq:FVDGsuff} is not met automatically. 
The $l^1$-errors of the magnetic field and 
 the velocity at $t = 0.05$ are
shown in Table \ref{tab:acc2}, 
where one can observe that the proposed PP DG scheme
can maintain third-order accuracy as expected.

\begin{table}[htbp]\label{tab:acc2}
	\centering
	\caption{\small Second problem of Example \ref{example:accuracy}: 
		$l^1$-errors at $t=0.05$ in $B_1,B_2,v_1,v_2$, and corresponding convergence rates for
		the PP third-order DG method at
		different grid resolutions.
	}
	\label{tab:order2}
	\begin{tabular}{c|c|c|c|c|c|c|c|c}
		\hline
		\multirow{2}{18pt}{Mesh}
		&\multicolumn{2}{c|}{$B_1$}&\multicolumn{2}{c|}{$B_2$}
		&\multicolumn{2}{c|}{$v_1$} &\multicolumn{2}{c}{$v_2$} 
		\\
		\cline{2-9}
		& error& order & error & order & error& order & error & order \\
		\hline
		$10\times 10$& 1.29e0& --          &1.29e0& --   &  1.85e0 & -- & 1.84e0  &-- \\
		$20\times 20$& 2.89e-1& 2.16   & 2.84e-1&    2.19 & 4.09e-1 & 2.18 & 4.06e-1 & 2.18\\
		$40\times 40$& 4.15e-2& 2.80   & 4.08e-2&  2.80 & 5.86e-2 & 2.80  & 5.87e-2 & 2.79 \\
		$80\times 80$& 4.36e-3 &  3.25    & 4.23e-3&  3.27 & 6.26e-3 & 3.23 & 6.20e-3& 3.24\\
		$160\times 160$& 6.19e-4 &  2.82 & 6.21e-4&    2.77 & 9.06e-4 & 2.79 & 9.09e-4 & 2.77\\
		%$320\times 320$&  9.04e-5 &  2.78  &9.29e-5&  2.74 & 1.35e-4 & 2.75 & 1.32e-4 & 2.78 \\
		\hline
	\end{tabular}
\end{table}

\end{expl}

To verify the capability of the proposed PP DG methods in resolving complex wave configurations, we will simulate  
a shock cloud interaction problem, a rotated shock tube problem, two blast problems  
and several astrophysical jets. For these problems, before the PP limiting procedure, the WENO limiter \cite{Qiu2005} with locally divergence-free reconstruction (\cite{ZhaoTang2017}) is implemented with the aid of the local characteristic decomposition 
to enhance the numerical stability of high-oder DG methods in resolving the strong discontinuities and their interactions. The WENO limiter is only used 
in the ``trouble'' cells adaptively detected by the indicator in \cite{Krivodonova}.

\begin{expl}[Shock cloud interaction]\label{example:sc}This problem, introduced in \cite{Dai1998}, describes the disruption of a high density cloud by a strong shock wave. 
	It is widely simulated in the literature, e.g., \cite{Toth2000,Balbas2006}. 
	Our setup is the same as that in \cite{Toth2000,Balbas2006}.  
The computational domain is $[0,1]^2$ with the right boundary specified as 
supersonic inflow condition and the others as outflow conditions. 
Initially, 
there is a 
discontinuity parallel to the $\tt y$-axis 
at 
${\tt x}=0.6$ 
with the left and right states 
$$(\rho,{\bf v},p,{\bf B})=
\begin{cases}
(3.86859,0,0,0,167.345,0,2.1826182,-2.1826182),\quad &{\tt x}<0.6,
\\
(1,-11.2536,0,0,1,0,0.56418958,0.56418958),\quad &{\tt x}>0.6.
\end{cases}$$ 
%and $\gamma=\frac53$. 
The discontinuity is a combination of a fast shock wave and a rotational discontinuity in $B_3$. 
There is a circular cloud centered at $(0.8,0.5)$ with radius 0.15. 
The cloud has the same states as the surrounding plasma except for a higher density 10. 

Fig.\ \ref{fig:sc} displays the schlieren image of the density 
as well as the magnetic field lines obtained by using the PP third-order DG 
method on the uniform mesh of $400\times 400$ cells. 
One can see that the discontinuities and complex flow structures 
are captured with high resolution, and the results agree well with 
those in \cite{Toth2000,Balbas2006}. We also observe that, 
the condition \eqref{eq:FVDGsuff} should be enforced by the 
the PP limiter, otherwise the high-order DG code breaks down 
at time $t \approx 0.03674$.  
	
	\begin{figure}[htbp]
	  \centering
	  {\includegraphics[width=0.98\textwidth]{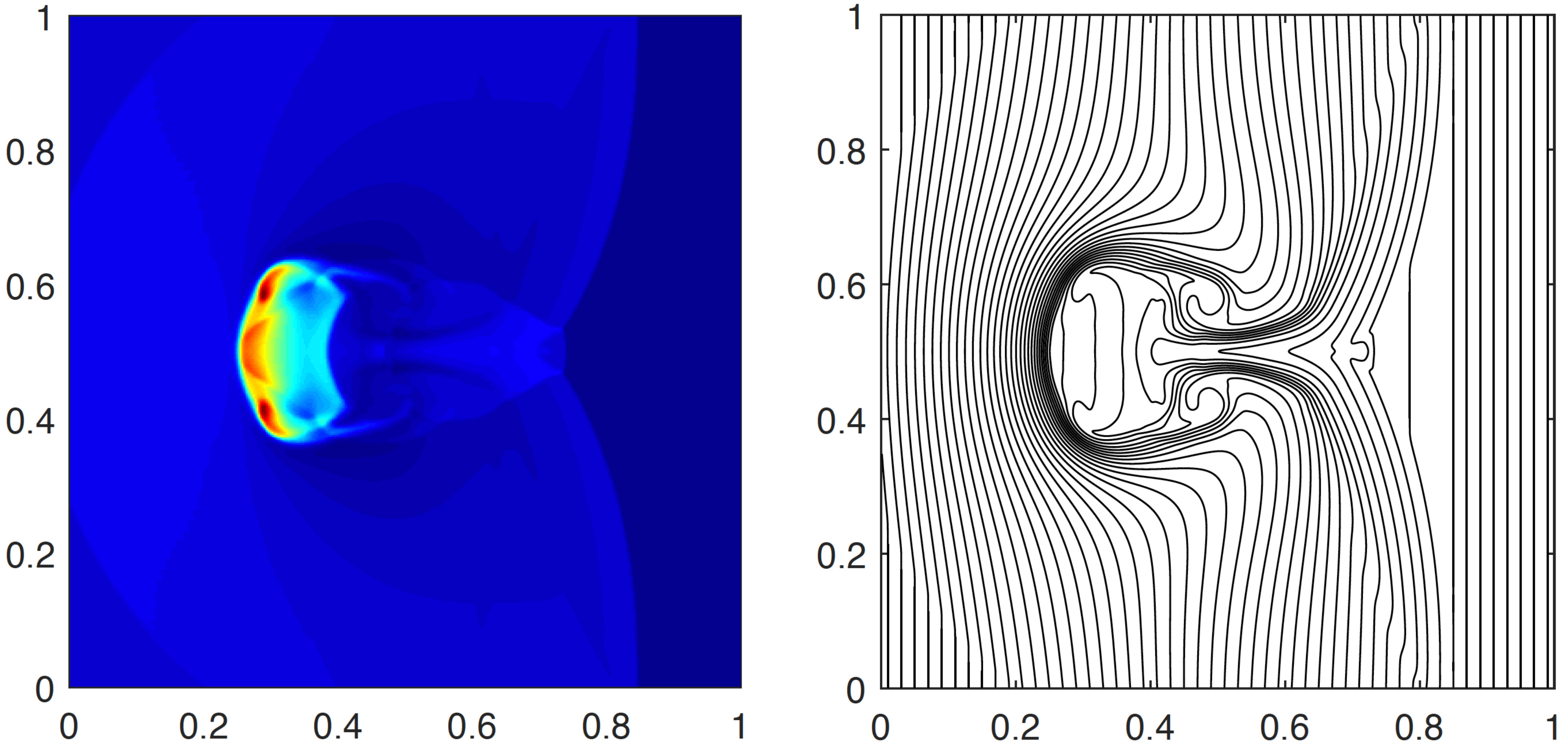}}
	  \caption{\small Example \ref{example:sc}: 
	  	the schlieren image of density (left), and the magnetic field lines (right) at $t=0.06$.} 
	  \label{fig:sc}
	\end{figure}

%	This problem, introduced in \cite{Dai1998}, describes the disruption of a high density cloud by a strong shock wave. It is widely simulated in the literature, e.g., 
%	\cite{Toth2000,Balbas2006,Mishra2012}.  
%	The computational domain is $[0,1]^2$ with the left boundary specified as 
%	supersonic inflow condition and the others as outflow conditions. 
%	Initially, a shock wave moves to the right from 
%	%discontinuity parallel to the $\tt y$-axis 
%	%at 
%	${\tt x}=0.05$ 
%	with the left and right states 
%	$$(\rho,{\bf v},p,{\bf B})=
%	\begin{cases}
%	(3.86859,11.2536,0,0,167.345,0,2.1826182,-2.1826182),\quad &{\tt x}<0.05,
%	\\
%	(1,0,0,0,1,0,0.56418958,0.56418958),\quad &{\tt x}>0.05.
%	\end{cases}$$ 
%	%and $\gamma=\frac53$. 
%	%The discontinuity is a combination of a fast shock wave and a rotational discontinuity in $B_3$. 
%	There exists a rest circular cloud centered at $(0.8,0.5)$ with radius 0.15. 
%	The cloud has the same states as the surrounding plasma except for a higher density 10. 
\end{expl}

\begin{expl}[Rotated shock tube problem]\label{ex:rst}This is a rotated Riemann problem \cite{Ryu1998} with the left and right
	states respectively given by $(\rho,v_{\parallel},v_{\perp},u_3,p,B_{\parallel},B_{\perp},B_3)$ $=(1,10,0,0,20,5/\sqrt{4\pi},5/\sqrt{4\pi},0)$ and $(1,-10,0,0,1,5/\sqrt{4\pi},5/\sqrt{4\pi},0)$. 
	The initial discontinuity is oblique to the Cartesian mesh and at an angle of $
	\varphi = 45^\circ $ to the $\tt x$-axis. Similar to 
	\cite{Toth2000,Chandrashekar2016,LiuShuZhang2017}, 
	the computational domain is taken as $[0,1]\times [0,2/N]$, 
	and divided into a Cartesian mesh with $N\times 2$ square cells. 
	The left and right boundaries are fixed according to the initial condition, 
	and we stop the computation at $t=0.08\cos(\alpha)$ before the fast shocks reach 
	the left and right boundaries. The shifted periodic type boundary conditions are used on the top and bottom of the domain as explained in \cite{Toth2000}. 
	We set $N=512$ and plot the numerical solution at the first row ($j=1$) of the physical mesh in Fig.\ \ref{fig:rst}. 
	For comparisons, the non-rotated 1D solution 
	on a fine mesh of 10000 cells is also displayed.  
	%This may be compared with the solution from \cite{Ryu1998} except for that 
	%they plot the slice along the line ${\tt x}={\tt y}$. 
	Similar to the nonconservative eight-wave type schemes in \cite{Toth2000,Chandrashekar2016,LiuShuZhang2017}, 
	the proposed DG method also has the problem that the parallel component of the magnetic field, 
	$B_{\parallel}$, which should be constant, shows a large error due to the nonconservative formulation. 
	In our result, the $l^\infty$-norm of this error is about 0.0176, 
	which is much less than that (about 0.2) in \cite{Chandrashekar2016} 
	obtained by a second-order finite volume scheme, and that (about 0.13) in \cite{LiuShuZhang2017} by a third-order DG scheme. The other quantities have good behavior in comparison with 
	the reference solution and the results in \cite{Ryu1998}.

	\begin{figure}[htbp]
	\centering
	{\includegraphics[width=0.49\textwidth]{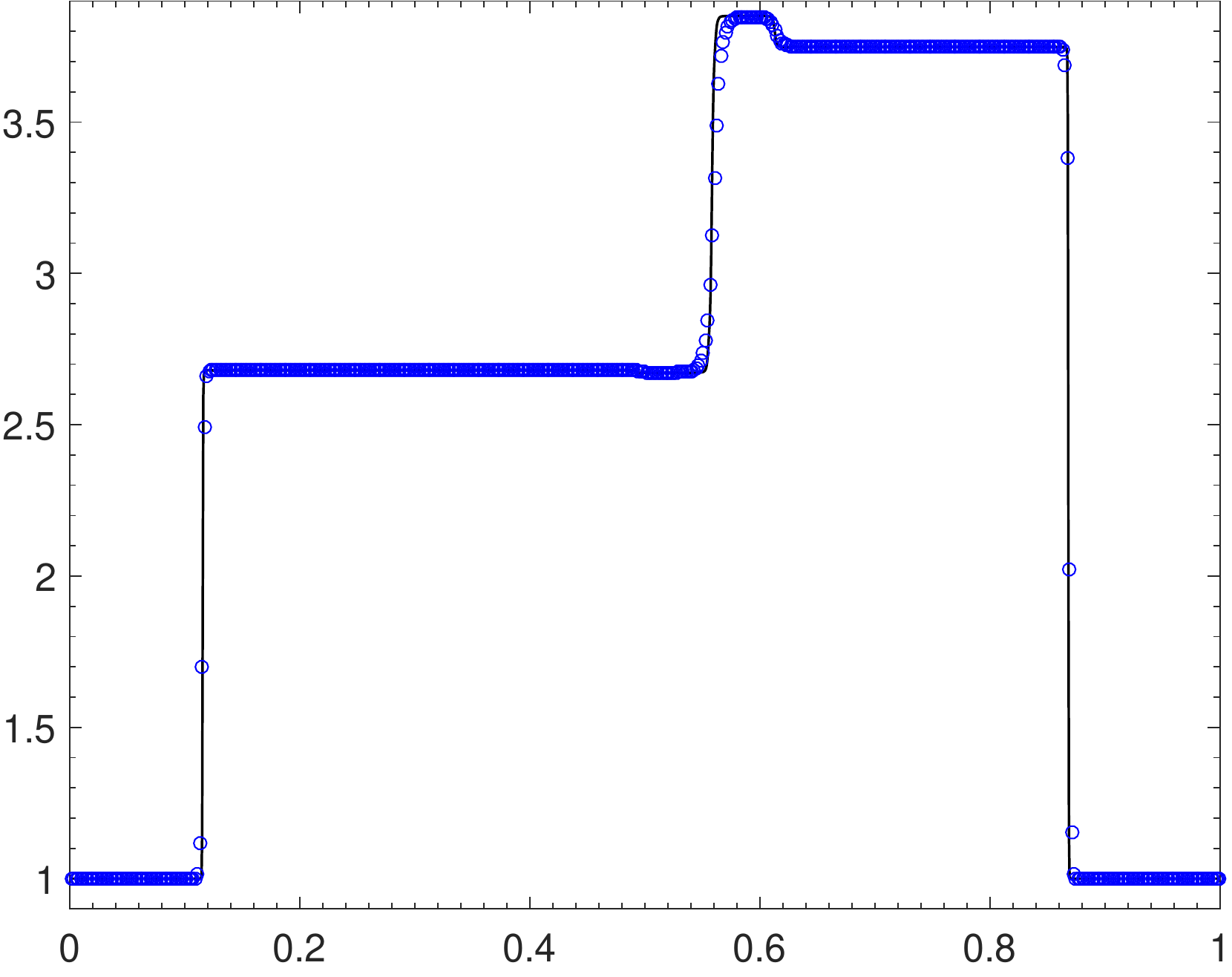}}
	{\includegraphics[width=0.49\textwidth]{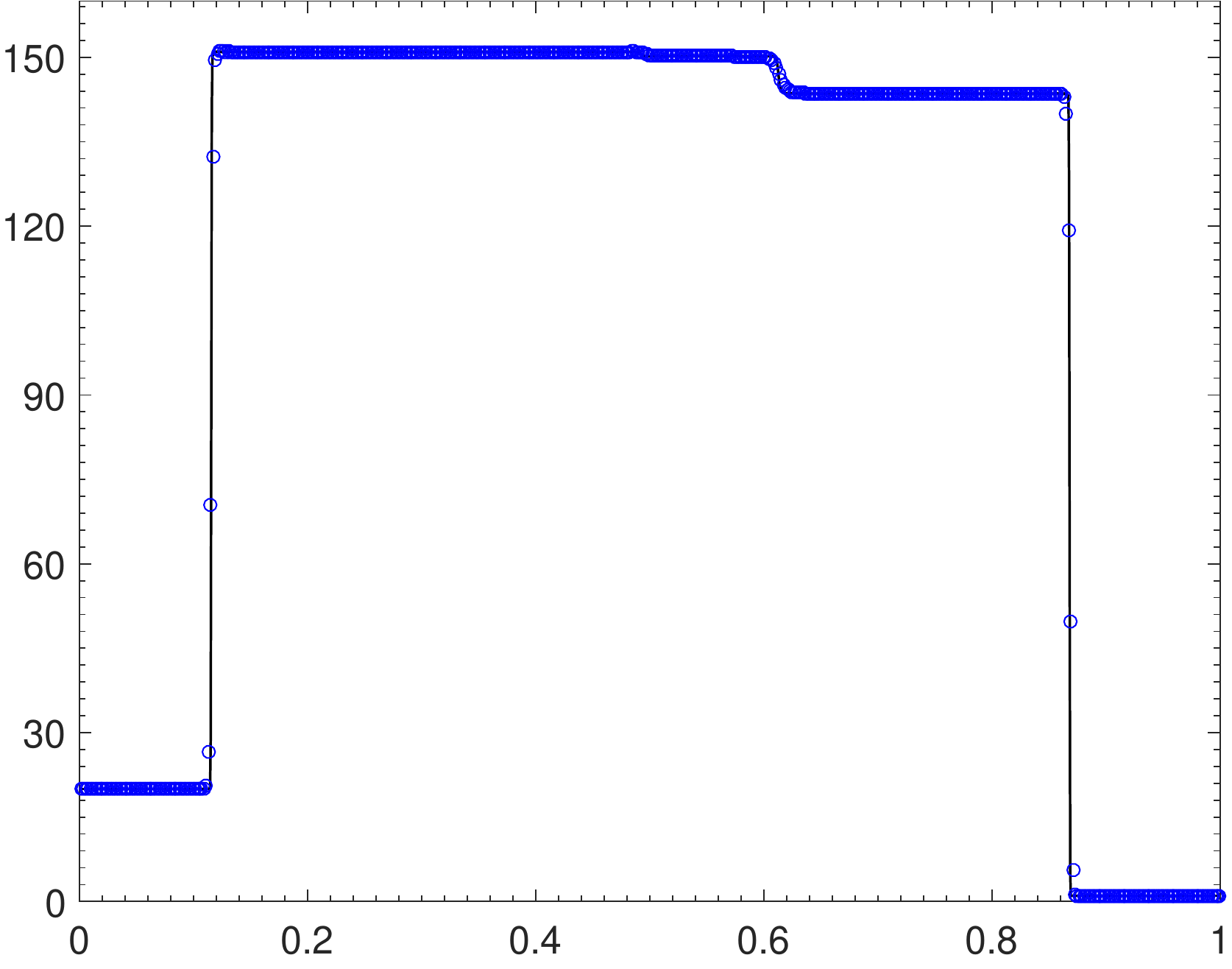}}
	{\includegraphics[width=0.49\textwidth]{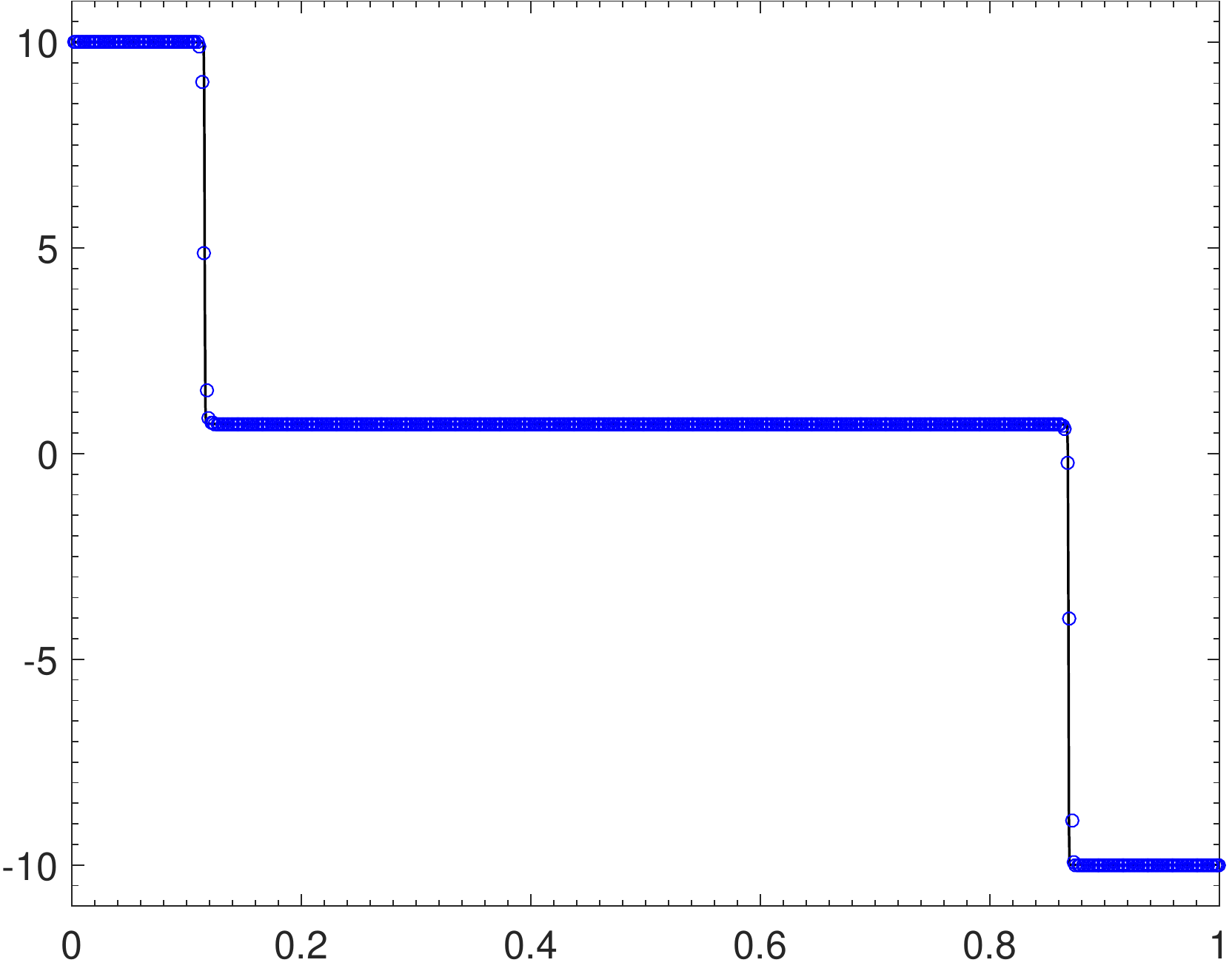}}
	{\includegraphics[width=0.49\textwidth]{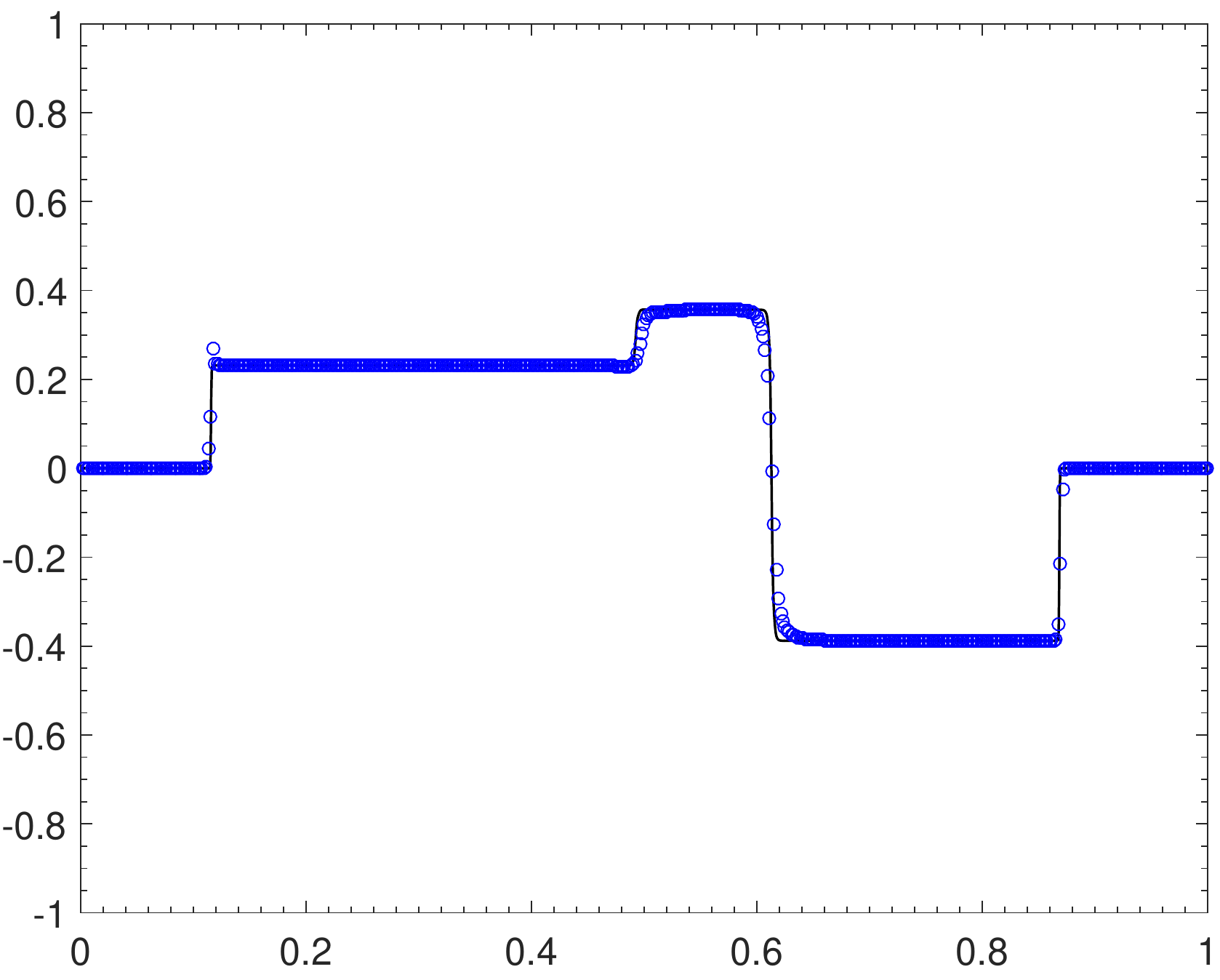}}
	{\includegraphics[width=0.49\textwidth]{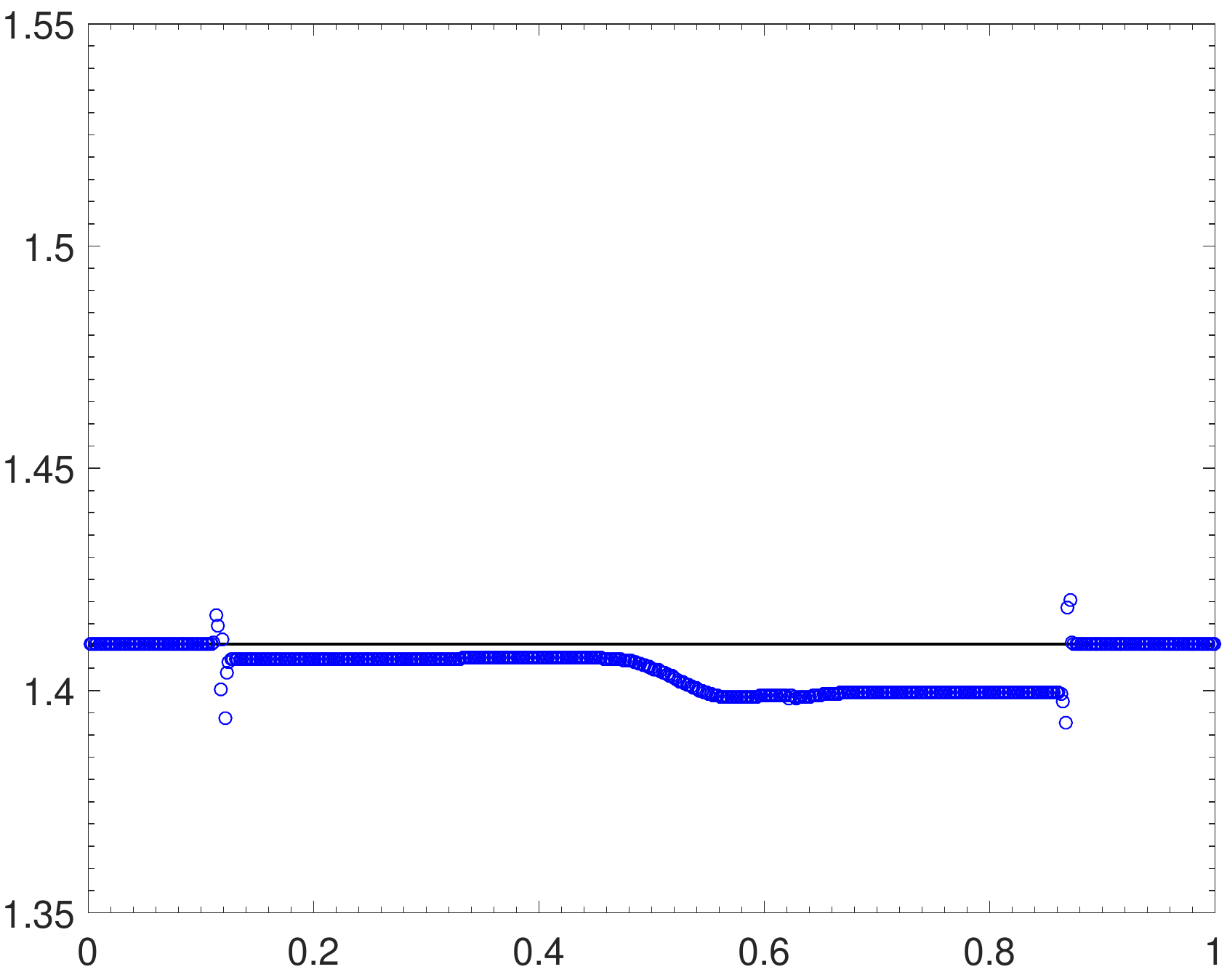}}
	{\includegraphics[width=0.49\textwidth]{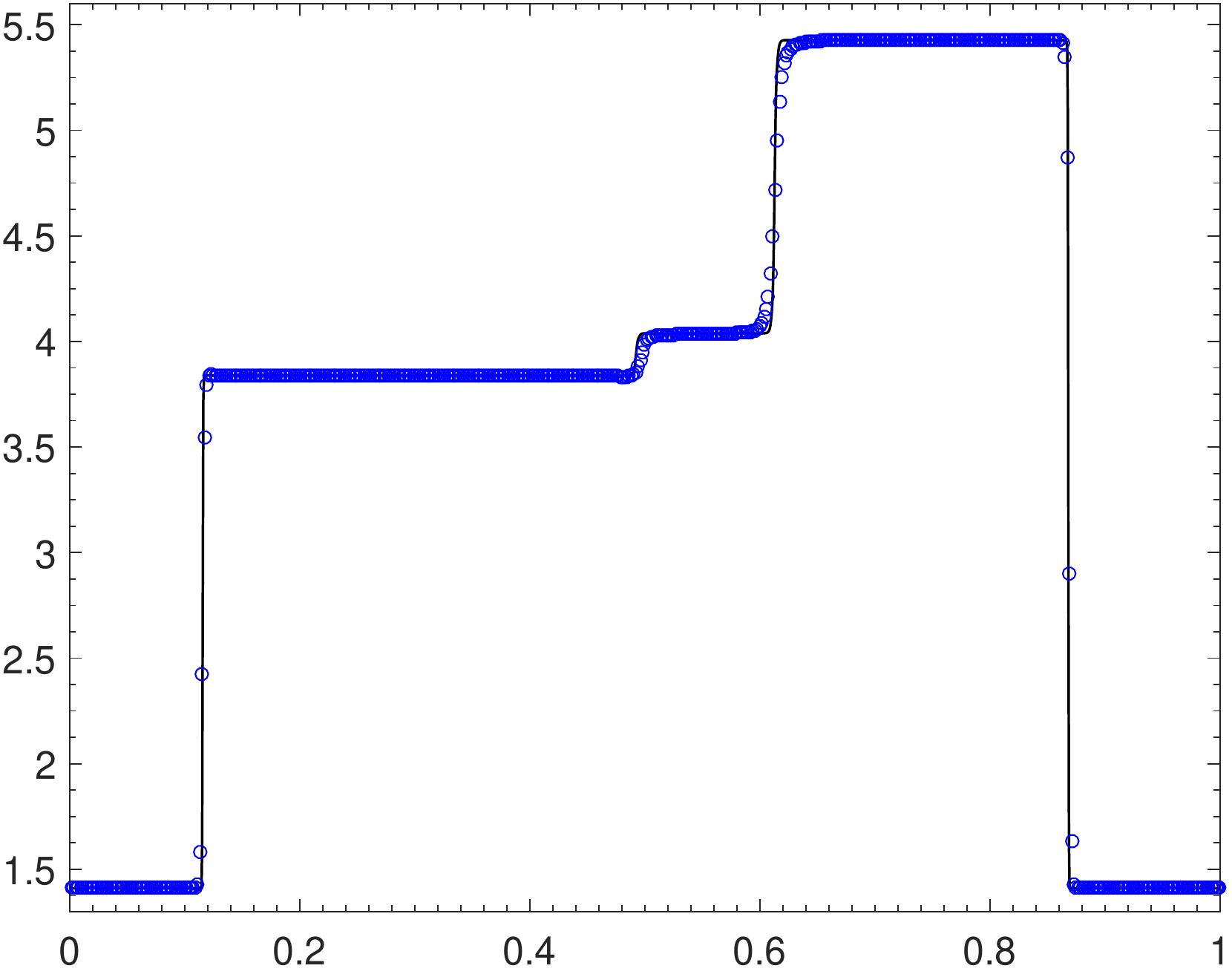}}
	\caption{\small Example \ref{ex:rst}: 
		Numerical solution of the 2D rotated shock tube problem 
		obtained by the PP third-order DG method {\rm(``$\circ$'')}. 
		For reference, the non-rotated 1D solution with 10000 cells is also plotted in solid lines. Top left: $\rho$; 
	top right: $p$; middle left: $v_{\parallel}$; middle right: $v_{\perp}$; bottom left: $B_{\parallel}$; bottom right: $B_{\perp}$.}   		
	\label{fig:rst}
	\end{figure}

\end{expl}

\begin{expl}[Blast problems]\label{ex:BL}MHD blast problem was first introduced by Balsara and Spicer in \cite{BalsaraSpicer1999}, 
	and has become a standard test for 2D MHD codes. 
	It describes the propagation of a
	circular strong fast magneto-sonic shock formulates and propagates into 
	the  ambient
	plasma with low plasma-beta ($\beta=2p/|{\bf B}|^2$). 
	As $\beta$ is set lower, negative pressure is more likely to be produced in the 
	numerical simulation and this problem becomes more challenging. 
	Therefore, it is often used to 
	check the robustness and PP property of MHD schemes, see e.g., \cite{cheng,Christlieb}. 
	Initially, the computational domain $[-0.5,0.5]^2$ 
	is filled with plasma at rest with the unit density and adiabatic index $\gamma=1.4$. The explosion zone 
	$(r<0.1)$ is with a pressure of $p_e$, and the ambient medium $(r>0.1)$ has a lower pressure of $p_a$, where $r=\sqrt{{\tt x}^2+{\tt y}^2}$. We initialize the magnetic field in the $\tt x$-direction as $B_a$.  
	
	We first consider the same setup as in \cite{BalsaraSpicer1999,cheng}, and take  
	$p_e=10^3$, $p_a=0.1$ and $B_a=100/\sqrt{4\pi}$. 
	The corresponding plasma-beta is very small and about $2.51\times 10^{-4}$.  
	Fig.\ \ref{fig:BL1} shows the contour plots of density, pressure, velocity and 
	magnetic pressure at $t=0.01$ computed by the PP third-order DG method with $320\times320$ 
	uniform cells. 
	We can see that 
	the outermost discontinuity in this expanding shell is a fast-shock which is only weakly compressive and
	energetically is dominated by the magnetic field. 
	The density
	image clearly shows two dense shells which propagate parallel to the magnetic field. The outer wave of
	these shells is a slow-shock, and the inner is a contact discontinuity 
	evolved from the initial interface which separates the initially hot, interior gas from the surrounding cool ambient medium \cite{Gardiner2005}. 
	Our results are highly in agreement with those displayed  
	in \cite{BalsaraSpicer1999,Li2011,Christlieb}, 
	and the density profile is well captured with much less oscillations than those shown in \cite{BalsaraSpicer1999,Christlieb}. 
	It is worth mentioning that 
	the third-order DG method fails at $t \approx 2.85\times 10^{-4}$ 
	if the PP limiting procedure is not employed to enforce the condition \eqref{eq:FVDGsuff}. 
	
	\begin{figure}[htbp]
	\centering
	{\includegraphics[width=0.49\textwidth]{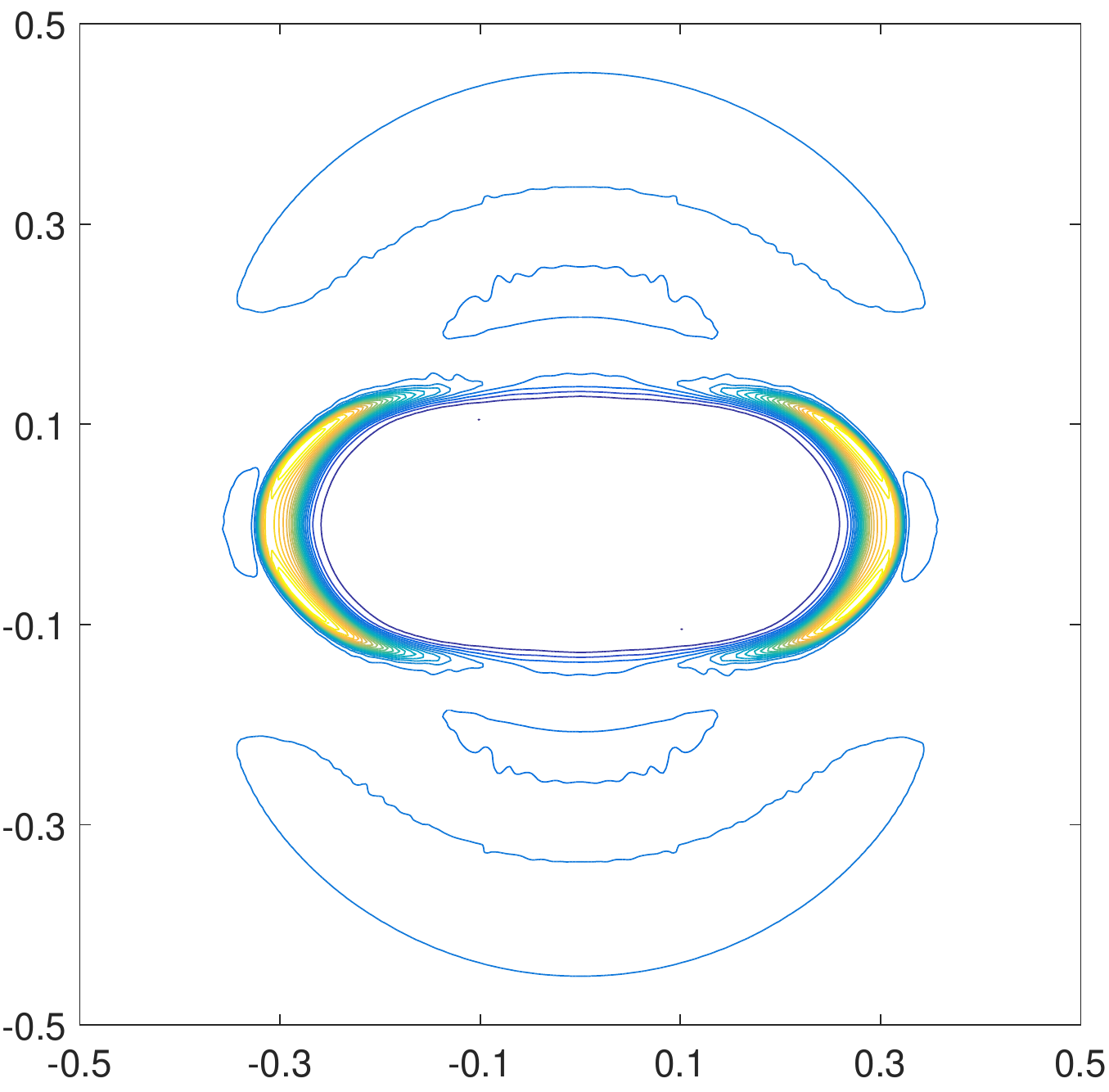}}
	{\includegraphics[width=0.49\textwidth]{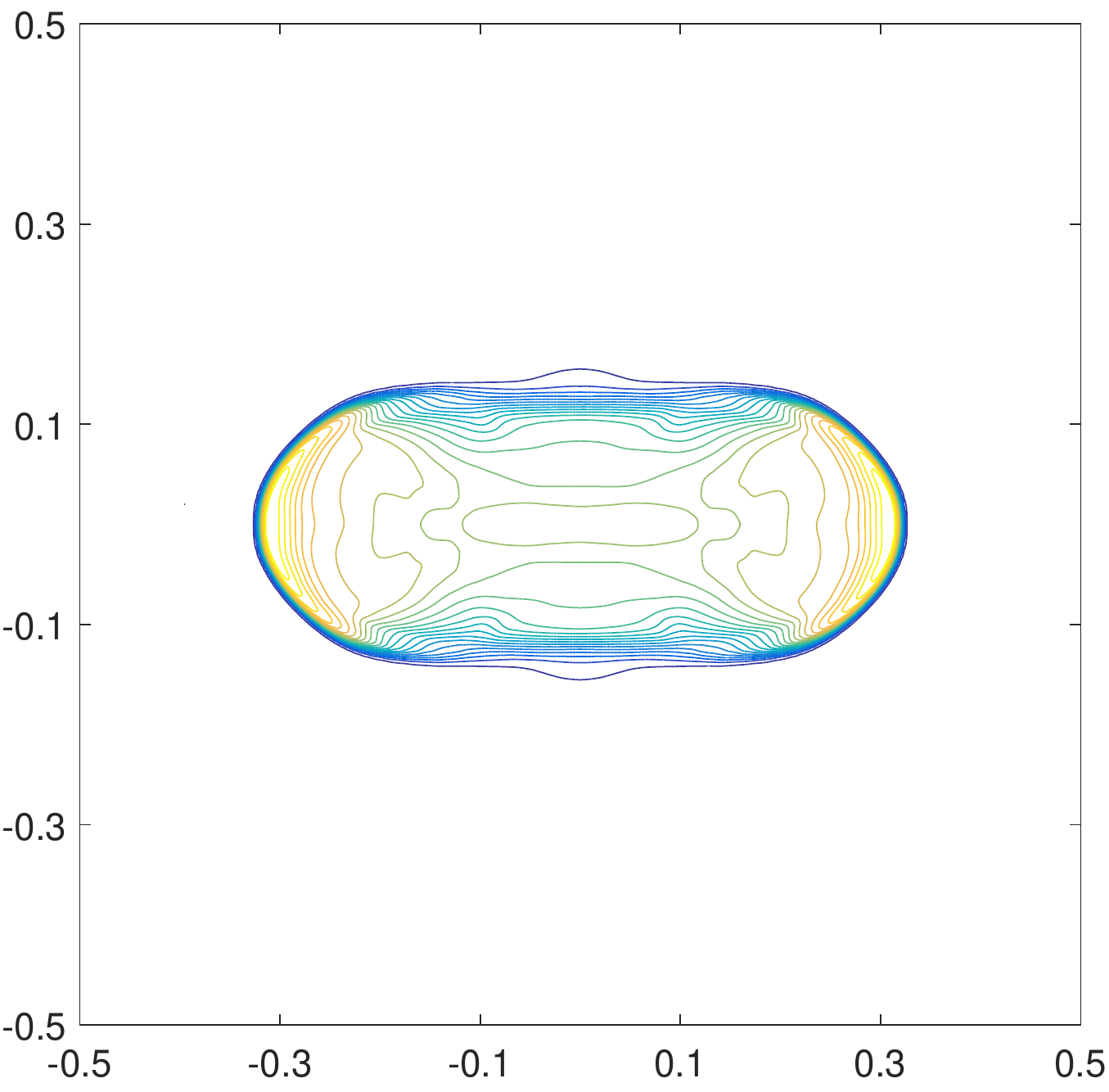}}
	{\includegraphics[width=0.49\textwidth]{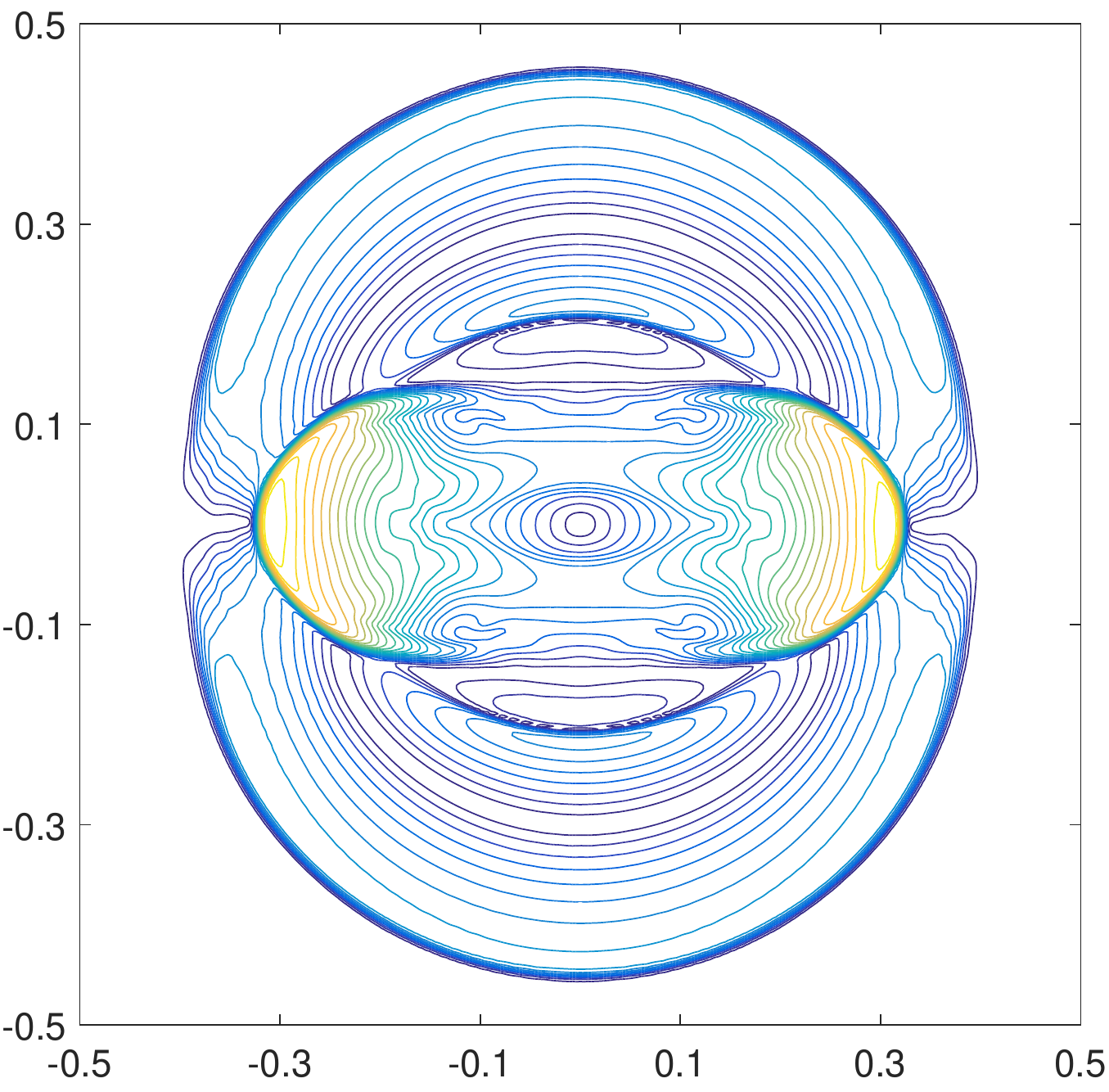}}
	{\includegraphics[width=0.49\textwidth]{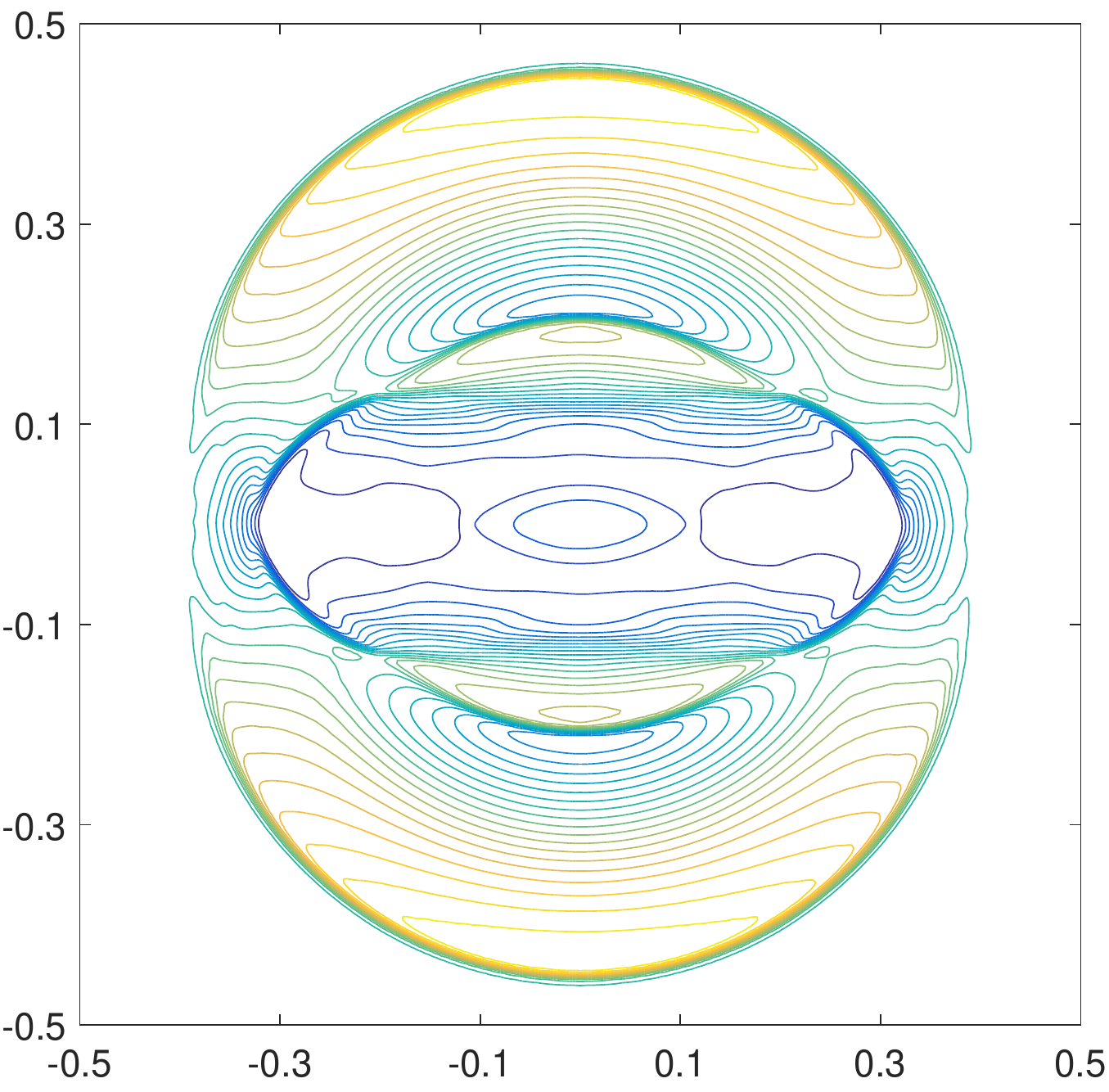}}
	\caption{\small The first blast problem in Example \ref{ex:BL}: 
		contour plots of density (top left), 
		pressure (top right), velocity $|{\bf v}|$ (bottom left) 
		and magnetic pressure (bottom right) at $t=0.01$.} 
	\label{fig:BL1}
\end{figure}

	To further demonstrate the robustness of the proposed PP DG method, we then test a more challenging case with larger initial jump in the pressure and   
	much stronger magnetic field. More specifically, we set $p_e=10^4$, $p_a=0.1$ and $B_a=1000/\sqrt{4\pi}$. 
	The corresponding plasma-beta is extremely small and about $2.51\times 10^{-6}$, 
	which is $1\%$ of that in the above standard setup.    
	  To our best knowledge, such extreme blast problem is rarely considered 
	  in the literature. Fig.\ \ref{fig:BL2} displays the 
	  numerical results at $t=0.001$ obtained by the PP third-order DG method on 
	 	the uniform mesh of $320\times320$ cells. 
	 	One can see that, as the magnetization is increased,  
	 	the external fast shock becomes much weaker and is not visible in the counter plot of density. In this extreme test, it is also necessary to use the PP limiter 
	 	to meet the condition \eqref{eq:FVDGsuff}, 
	 	otherwise the DG method will fail at 
	 	$t\approx1.2\times 10^{-5}$ due to negative numerical pressure.

	\begin{figure}[htbp]
	\centering
	{\includegraphics[width=0.49\textwidth]{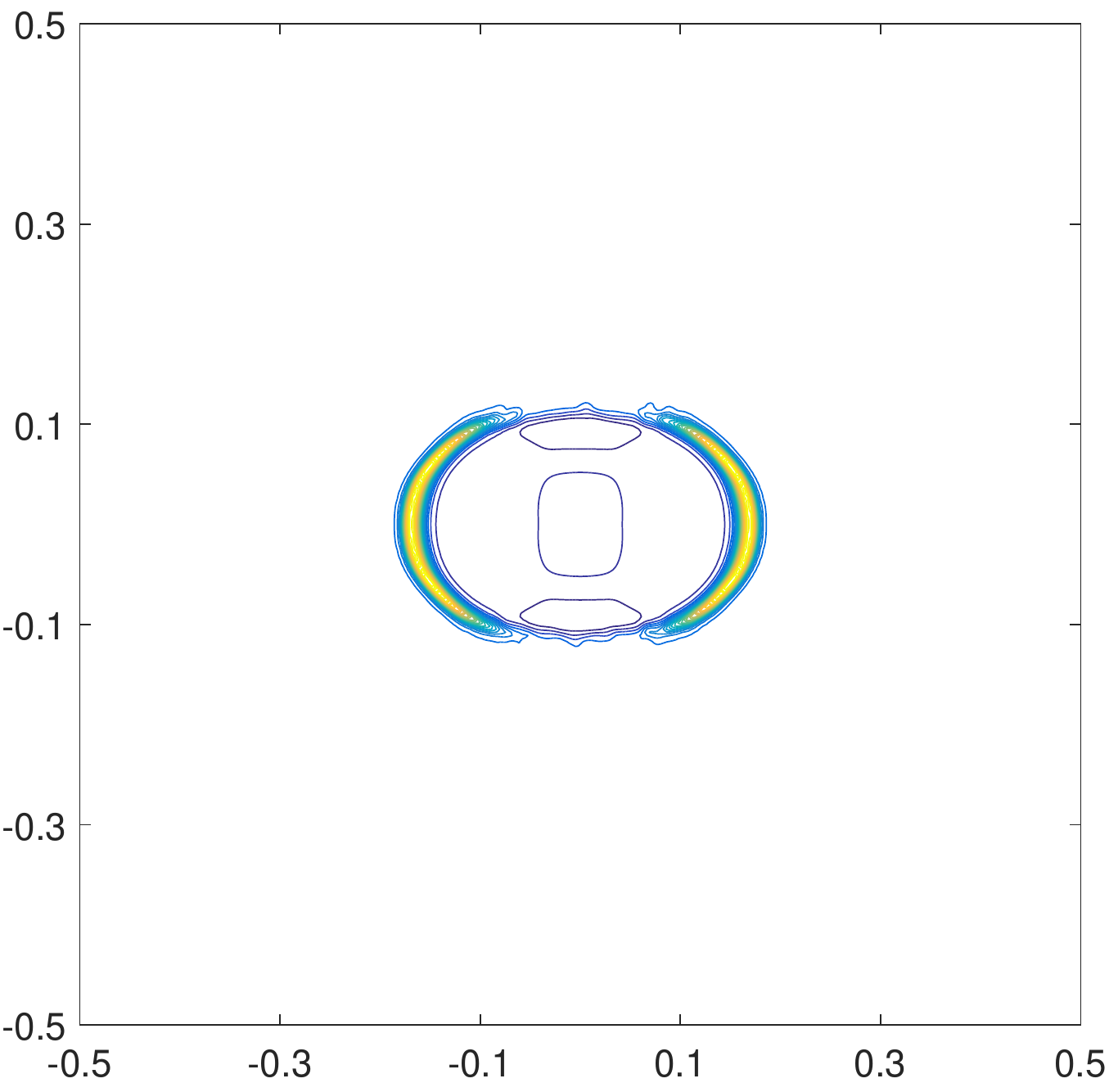}}
	{\includegraphics[width=0.49\textwidth]{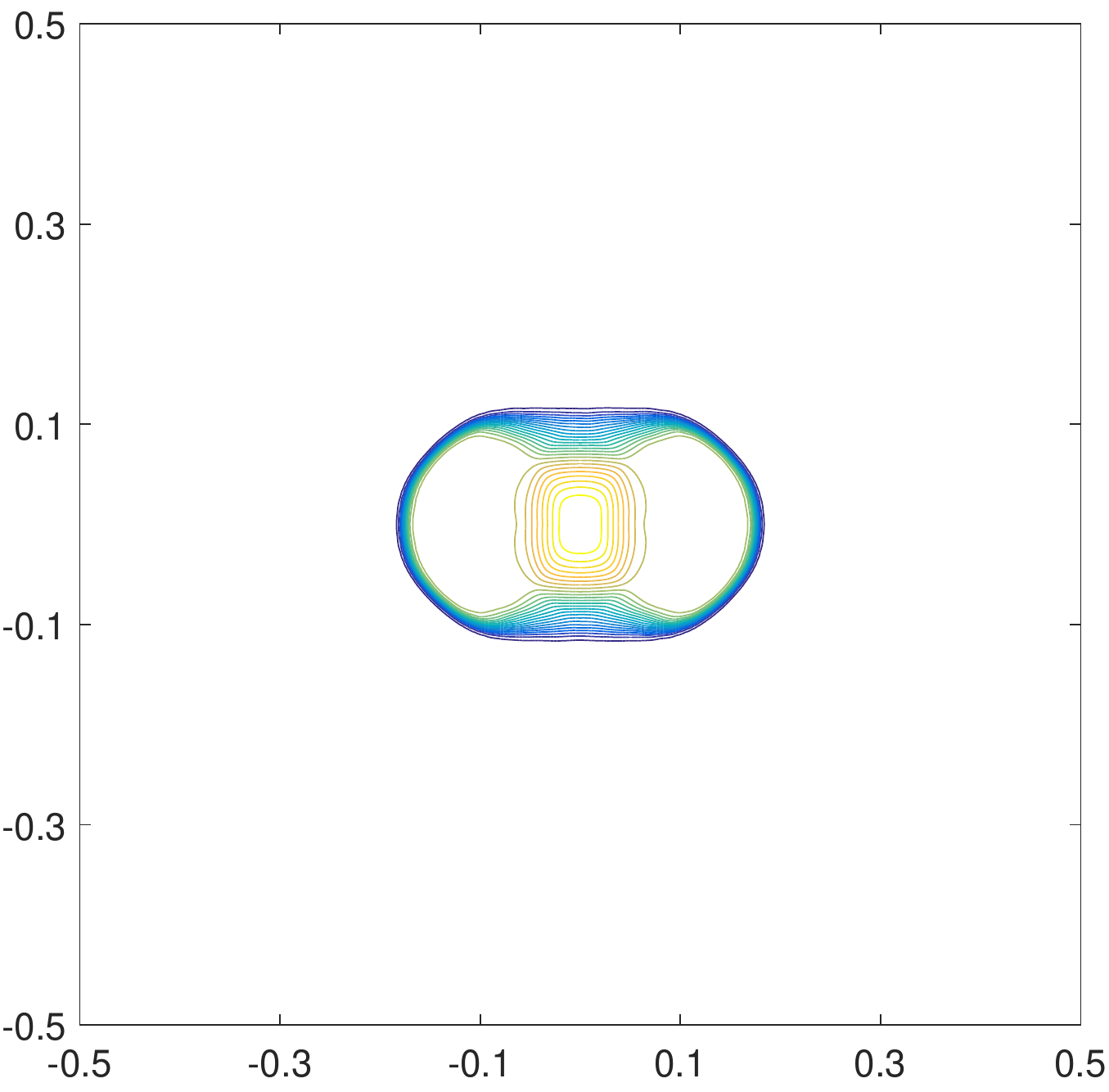}}
	{\includegraphics[width=0.49\textwidth]{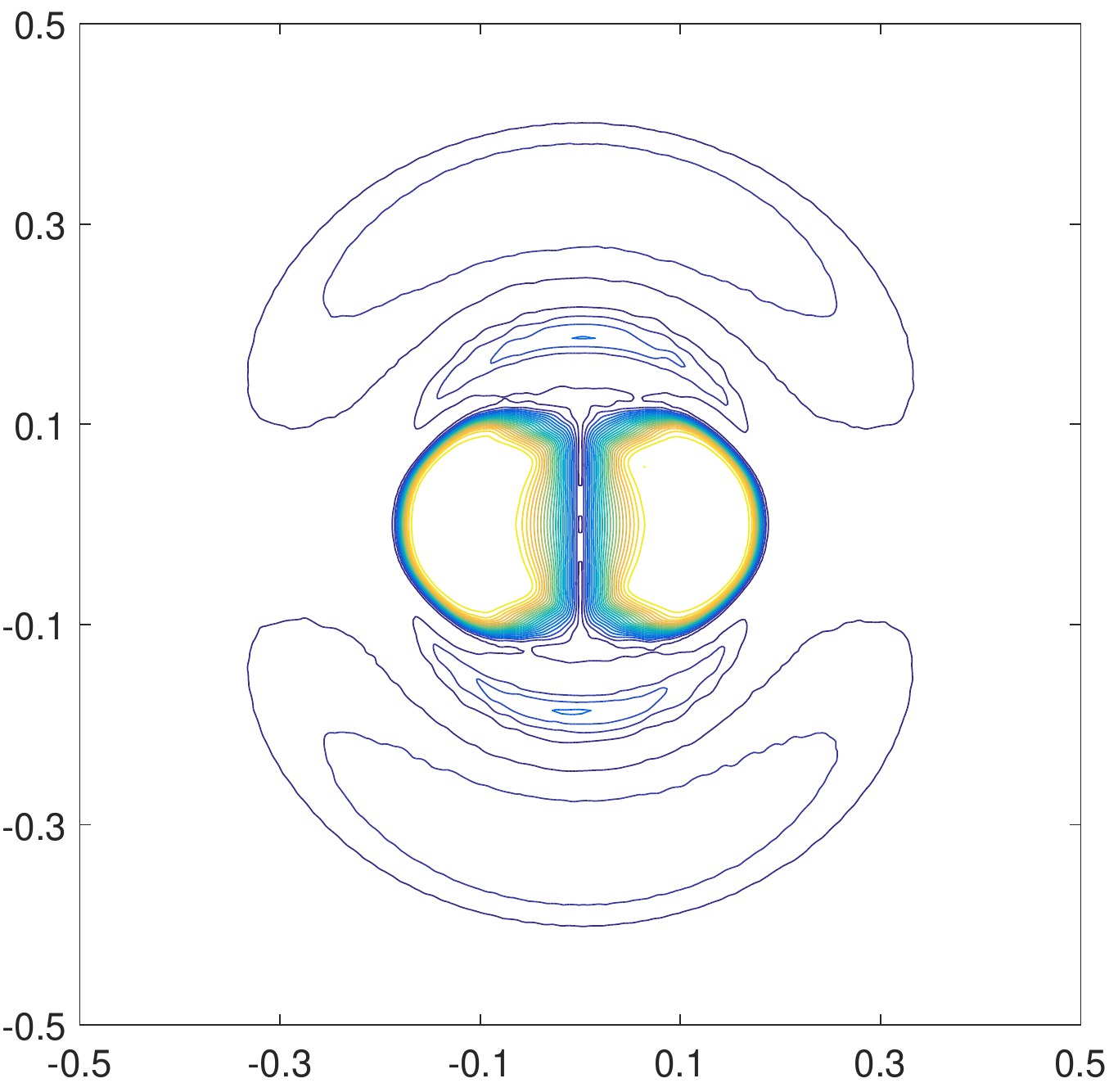}}
	{\includegraphics[width=0.49\textwidth]{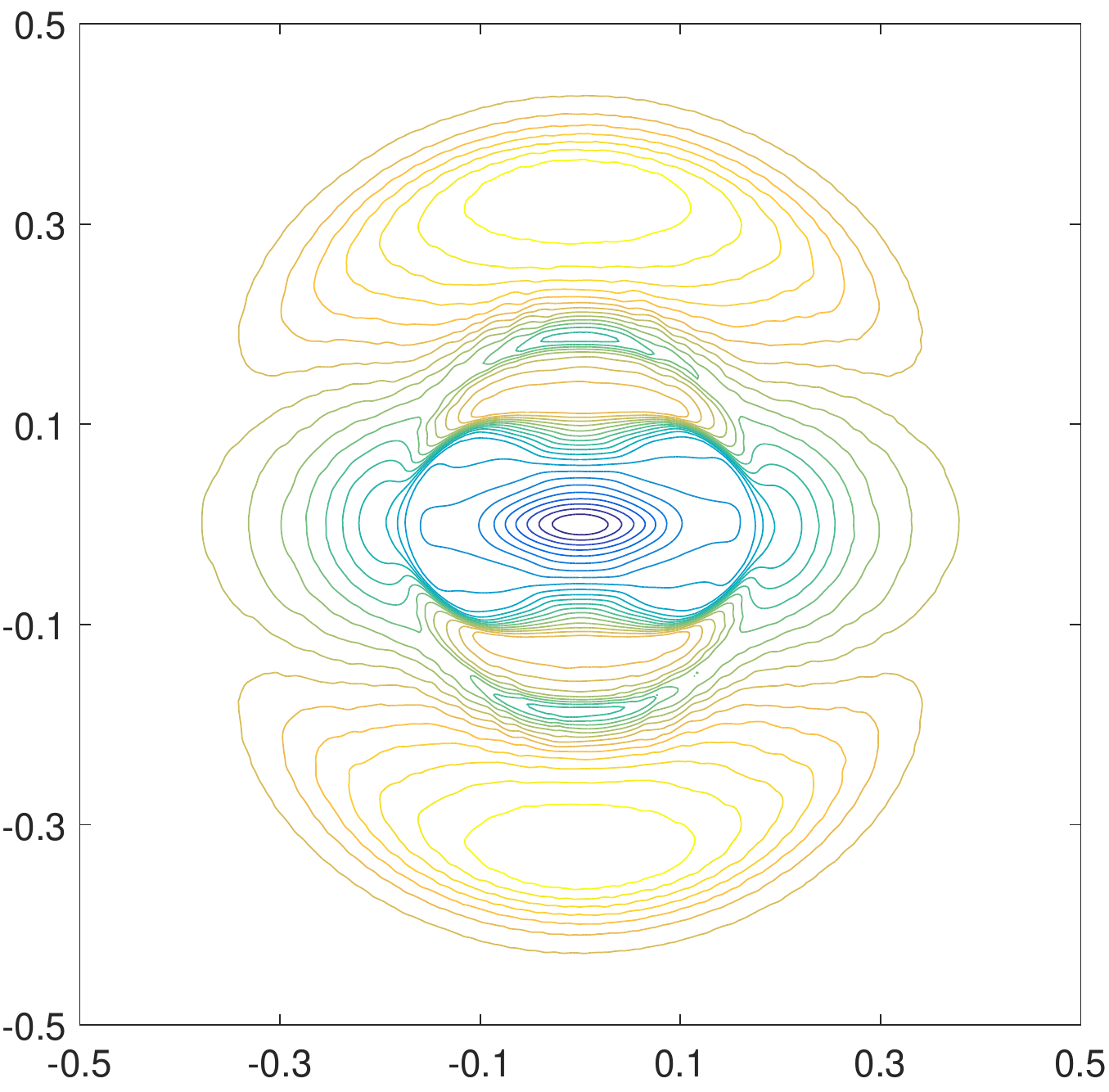}}
	\caption{\small Same as Fig.\ \ref{fig:BL1} except for 
		the second blast problem and $t=0.001$.} 
	\label{fig:BL2}
\end{figure}

To justify that the CFL condition \eqref{eq:CFL:2DMHD} 
	is acceptable, we show 
	the values of 
	${\vartheta_\ell}/{\alpha_{\ell,n}^{\tt LF}}$, $\ell=1,2,$ and 
	$\theta$ in Fig.\ \ref{fig:theta} for above two blast problems. We observe that, during the whole simulations,  
	the ratios ${\vartheta_\ell}/{\alpha_{\ell,n}^{\tt LF}}$,  $\ell=1,2,$ are very small, 
	and $\theta$ is always larger than 0.98 and very close to 1. 
	This is consistent with our analysis in Remark \ref{rem:CFLresonable}, and further confirms that $\theta$ in \eqref{eq:CFL:2DMHD}  does not cause strict restriction on the time step-sizes.

	\begin{figure}[htbp]
	\centering
	{\includegraphics[width=0.49\textwidth]{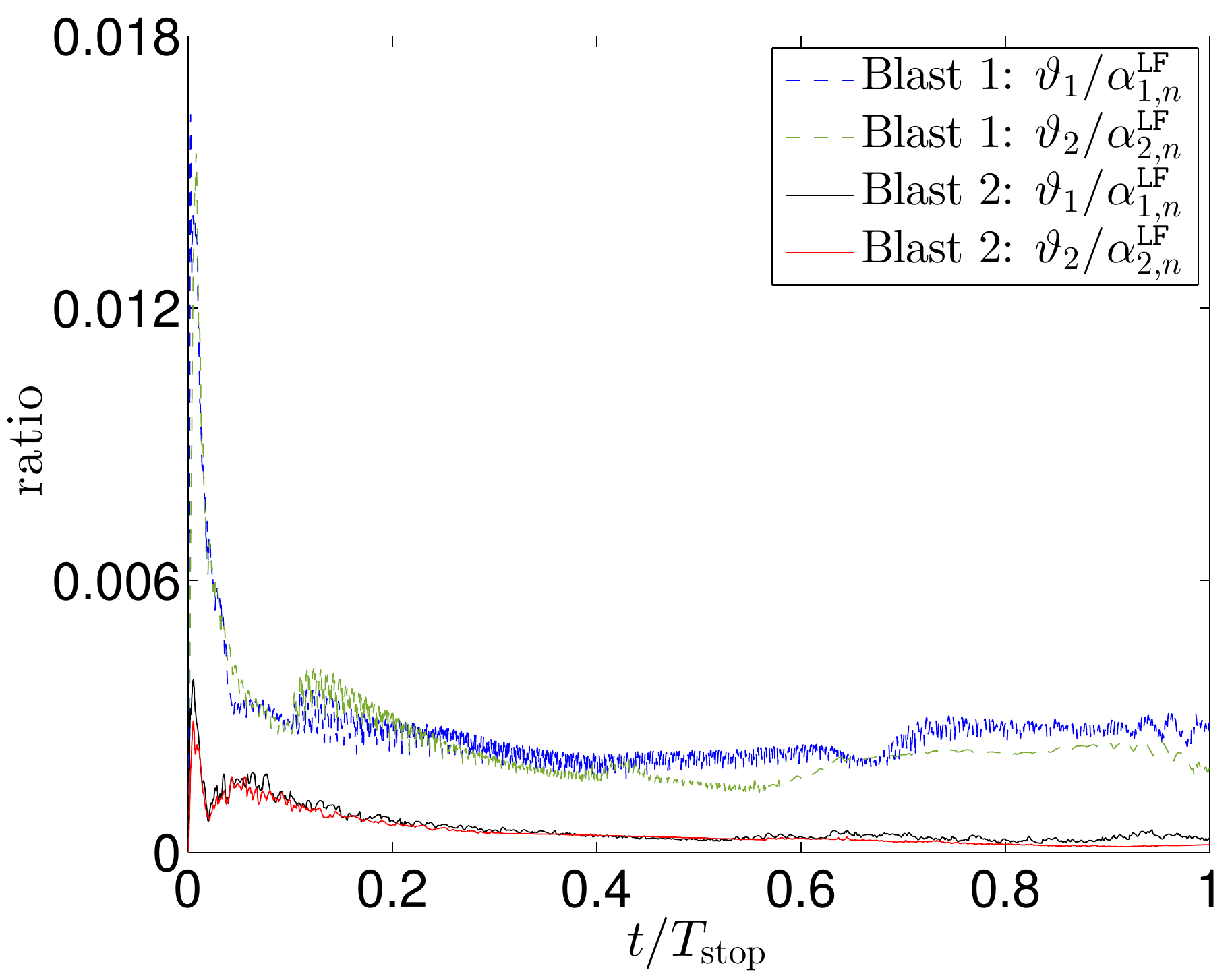}}
	{\includegraphics[width=0.49\textwidth]{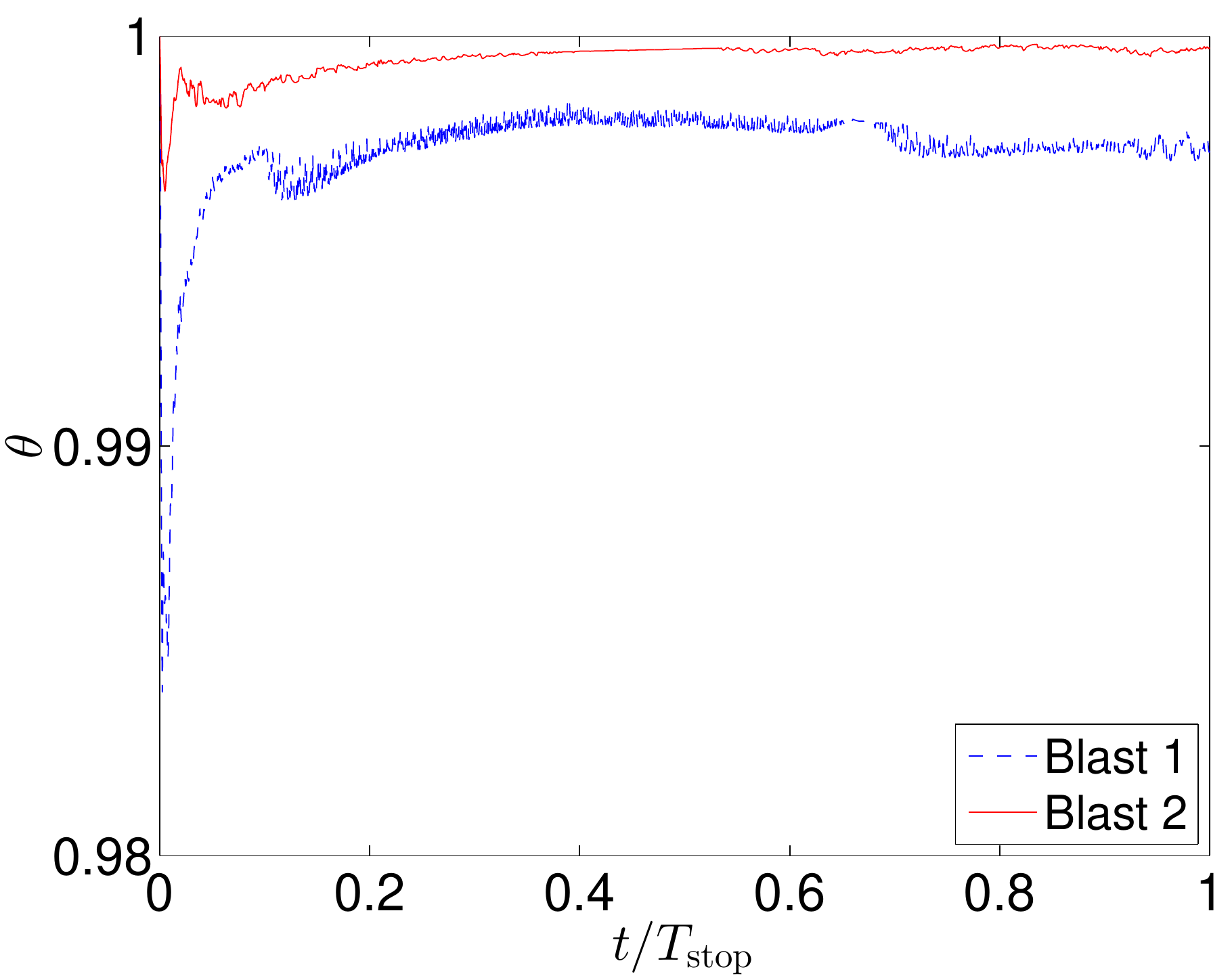}}
	\caption{\small Example \ref{ex:BL}: 
		${\vartheta_\ell}/{\alpha_{\ell,n}^{\tt LF}}$ (left)
		and $\theta$ (right) for the two blast problems.} 
	\label{fig:theta}
\end{figure}

\end{expl}

\begin{expl}[Astrophysical jets]\label{ex:jet}The last example is to simulate several astrophysical jets. 
If the jet speed is extremely high, the Mach number is very large 
and/or the magnetic field is exceedingly strong, 
then it is very challenging to successfully simulate such jet flows, see e.g., \cite{zhang2010b,Balsara2012,WuTang2015,WuTang2017ApJS}. 
Since the internal energy is very
small compared to the huge kinetic energy and/or magnetic energy, 
negative pressure could easily appear in the numerical simulation. 
Moreover, there exist strong shock wave, shear flow and interface
	instabilities etc., in high-speed jet flows.
Therefore, we have a strong motivation to use
the PP high-order DG methods for this kind of problems.

Consider the Mach 800 dense jet in \cite{Balsara2012}, and add a magnetic field to simulate the MHD jet flows. 
Initially, the physical domain $[-0.5,0.5]\times[0,1.5]$ is 
filled with a uniform static medium with density $0.1\gamma$ and unit pressure, 
and the adiabatic index $\gamma$ is set as $1.4$. 
Through the inlet part ($\left|{\tt x}\right|<0.05$) 
on the bottom boundary (${\tt y}=0$), a dense jet with speed $800$ is injected in the $\tt y$-direction with a density of $\gamma$ and a pressure 
equal to the ambient pressure.
The fixed inflow beam condition
is specified on the nozzle $\{{\tt y}=0,\left|{\tt x}\right|<0.05\}$, and the
others are outflow boundary conditions. 
We initialize the magnetic field with magnitude $B_a$ along the $\tt y$-direction.  
With the magnetic field, this test becomes more extreme. 
As $B_a$ is set larger, the initial ambient magnetization becomes higher (plasma-beta becomes lower), and this problem becomes more challenging. 
Numerical experiments in \cite{Wu2017a} indicated that the locally divergence-free, conservative, third-order DG method with the PP limiter is not able to run this test with $B_a\ge \sqrt{200}$ due to the negative numerical pressure. 
In this test, we take the computational domain as  
$[0,0.5]\times[0,1.5]$ with the reflecting boundary
condition specified at ${\tt x}=0$, and divide it into $200 \times 600$ cells. Three configurations are considered:
\\
(i) Moderately magnetized case: $B_a= \sqrt{200}$, corresponding plasma-beta $\beta_a=10^{-2}$.
\\
(ii) Strongly magnetized case: $B_a= \sqrt{2000}$, corresponding plasma-beta $\beta_a=10^{-3}$.
\\
(iii) Extremely strongly magnetized case: $B_a= \sqrt{20000}$, plasma-beta $\beta_a=10^{-4}$.
Figs.\ \ref{fig:jetd} and \ref{fig:jetp} %and \ref{fig:jetpm} 
display, respectively, the schlieren
images of density logarithm and pressure
logarithm %and magnetic pressure 
within the domain $[-0.5,0.5]\times[0,1.5]$. 
The ``colormap'' for plots of pressure
logarithm is carefully chosen close to that in \cite{Balsara2012} for a sake of comparison, while for density logarithm we simply use the ``jet colormap'' predefined in MATLAB. 
As one can see, the flow structures in different magnetized cases  
are very different. 
The Mach shock wave at the jet head and
the beam/cocoon interface are well captured, 
and the proposed PP DG method exhibits good
performance and robustness in such extreme tests. 
And if the PP limiter is turned off, the simulation will break down after 
several time steps due to nonphysical numerical solutions.

	\begin{figure}[htbp]
	\centering
	{\includegraphics[width=0.98\textwidth]{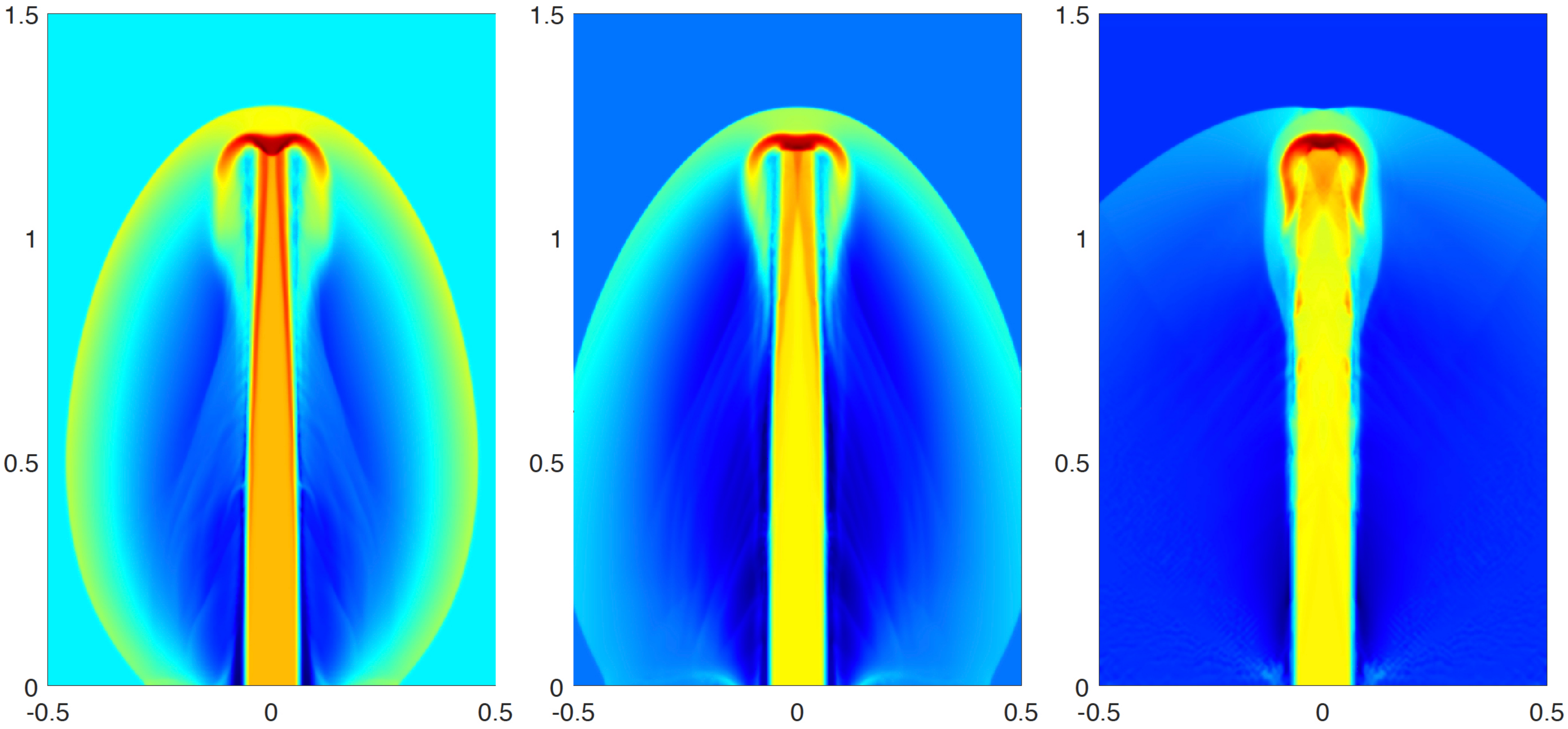}}
	\caption{\small Example \ref{ex:jet}: Schlieren images of the density logarithm at $t=0.002$. From left to right: configurations (i) to (iii).} 
	\label{fig:jetd}
\end{figure}

	\begin{figure}[htbp]
	\centering
	{\includegraphics[width=0.98\textwidth]{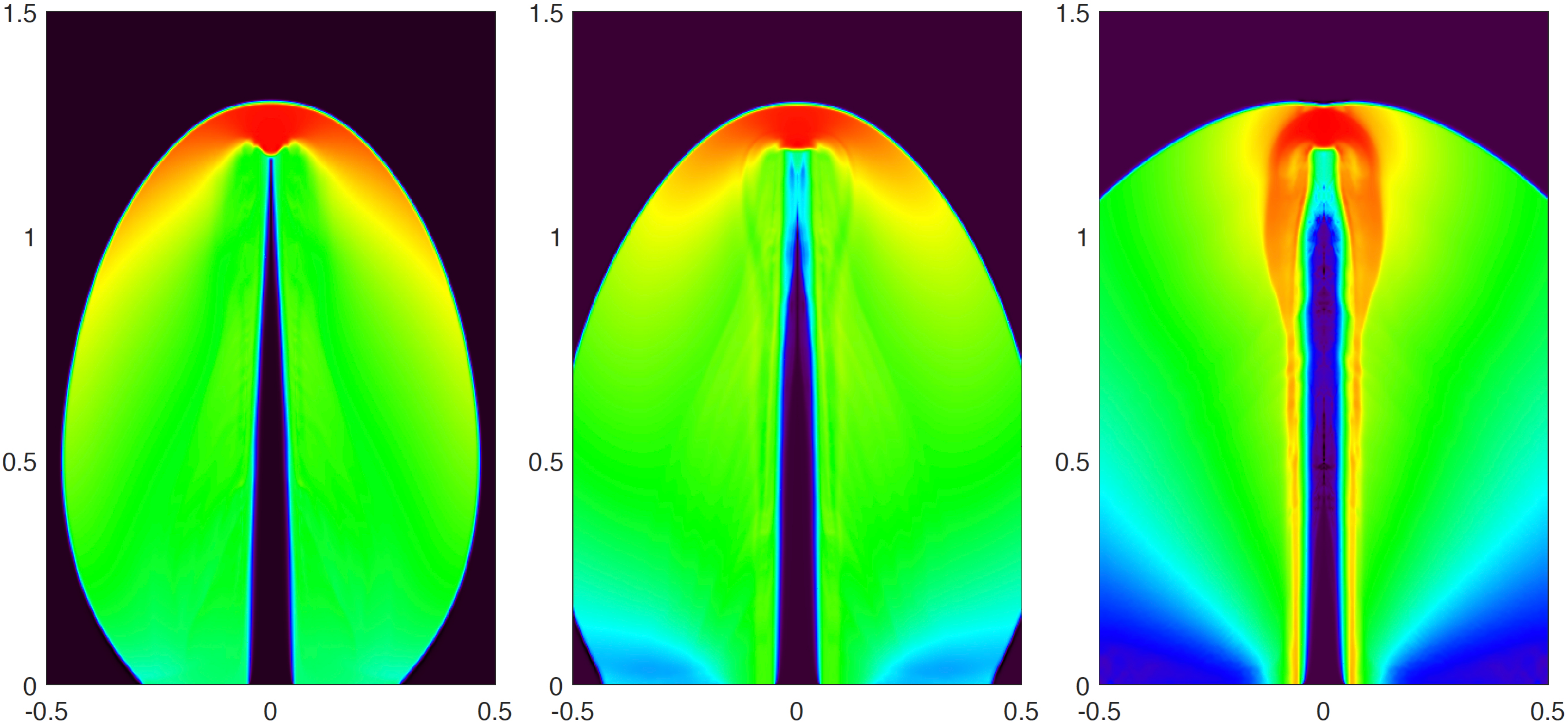}}
	\caption{\small Same as Fig.\ \ref{fig:jetd} except for the schlieren images of pressure logarithm.} 
	\label{fig:jetp}
\end{figure}

We now give more numerical evidences to support our analysis in Remark \ref{rem:CFLresonable} about the CFL condition \eqref{eq:CFL:2DMHD}. The values of 
${\vartheta_\ell}/{\alpha_{\ell,n}^{\tt LF}}$ and 
$\theta$ are shown in Fig.\ \ref{fig:theta2}  for  
the challenging configuration (iii), while the results of configurations (i)--(ii) are similar and omitted. 
We see that, during the whole simulation,  
${\vartheta_\ell}$ is much smaller than 
${\alpha_{\ell,n}^{\tt LF}}$, $\ell=1,2,$ 
and $\theta$ is always very close to 1. This demonstrates, again, that $\theta$ in \eqref{eq:CFL:2DMHD} does not cause strict restriction on the time step-sizes.

\begin{figure}[htbp]
	\centering
	{\includegraphics[width=0.49\textwidth]{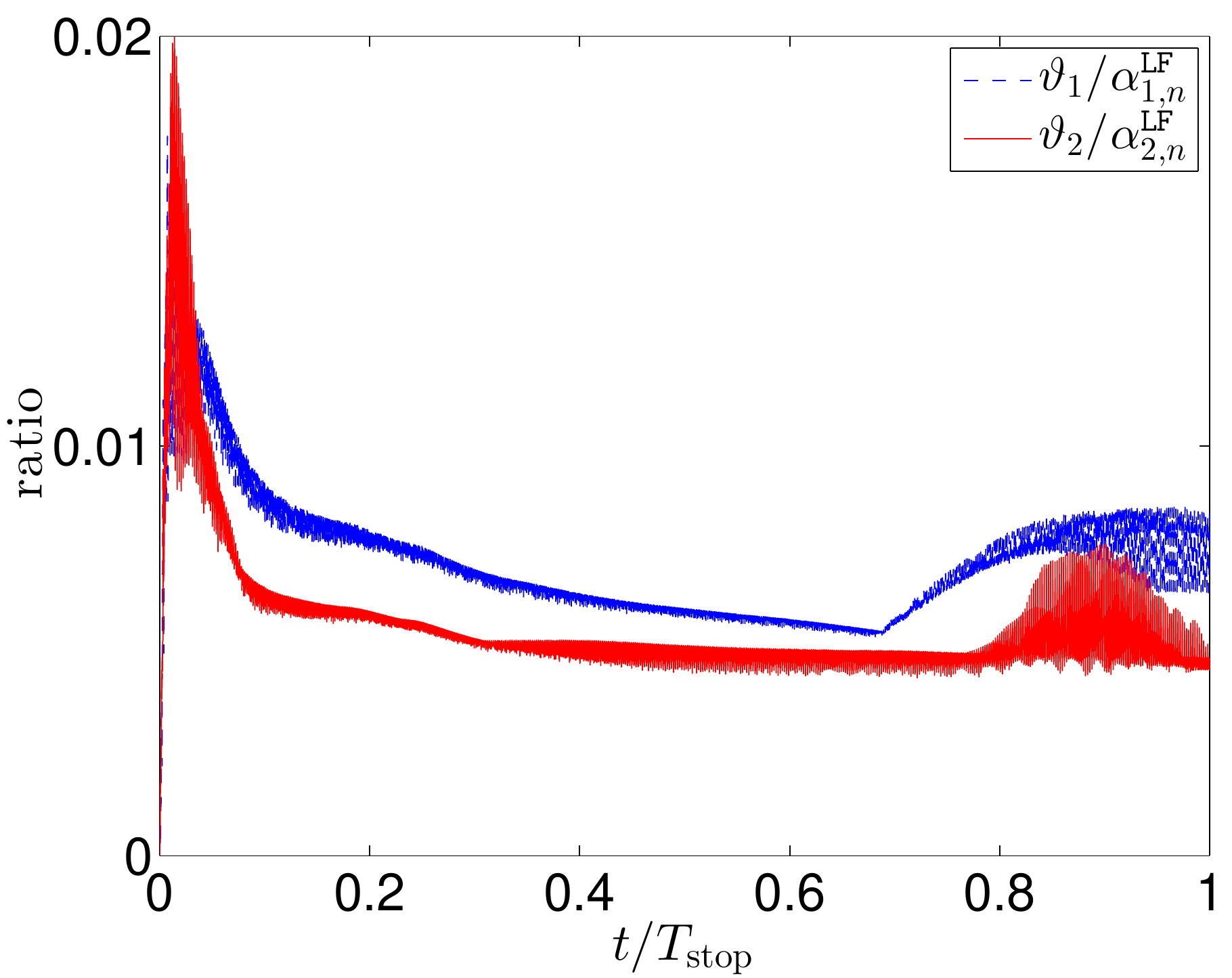}}
	{\includegraphics[width=0.49\textwidth]{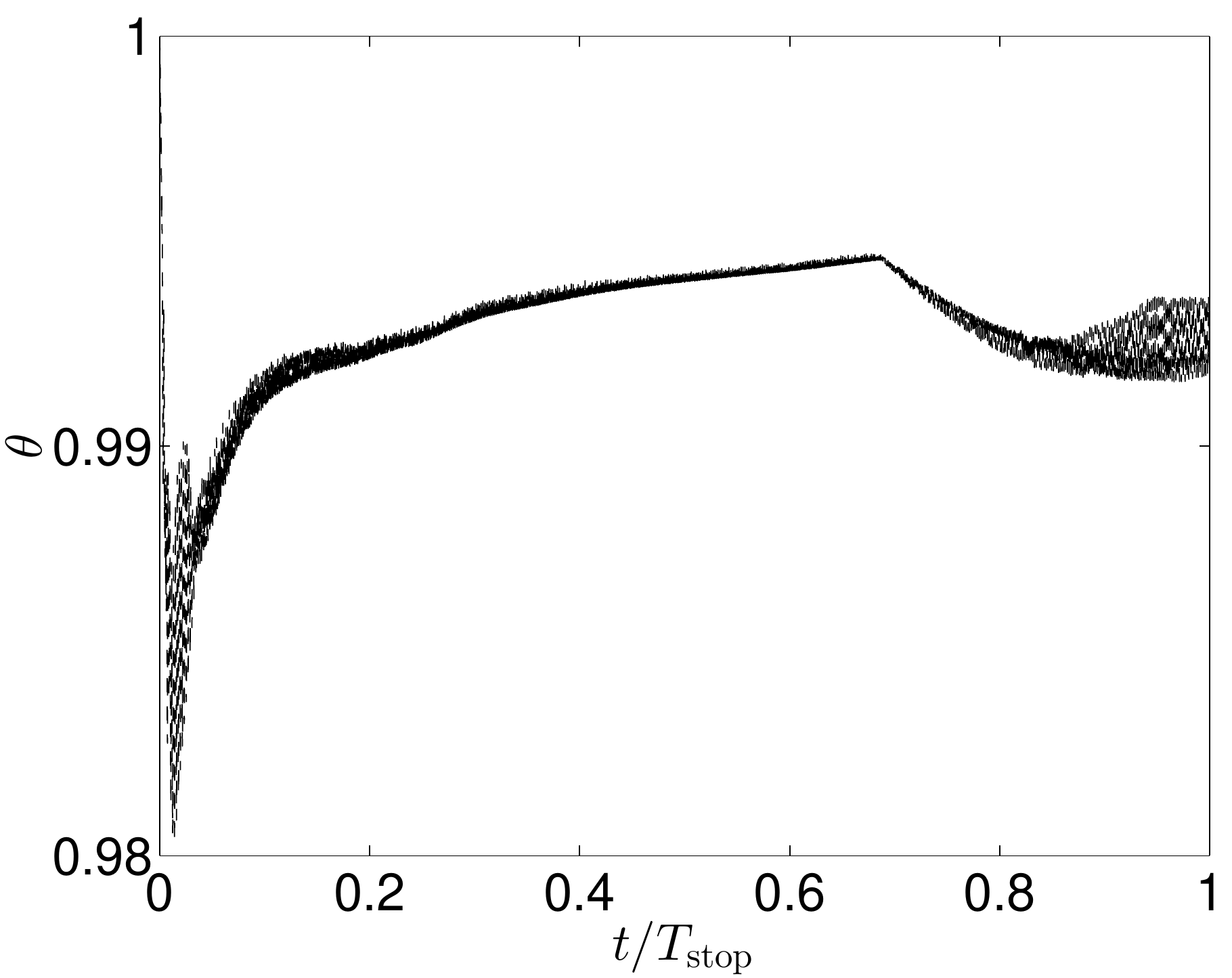}}
	\caption{\small  Example \ref{ex:BL}: 
		${\vartheta_\ell}/{\alpha_{\ell,n}^{\tt LF}}$ (left)
		and $\theta$ (right) for the jet problem in extremely strongly magnetized case.} 
	\label{fig:theta2}
\end{figure}

\end{expl}

\section{Conclusions}\label{sec:con}
We have constructed arbitrarily high-order
accurate positivity-preserving (PP) discontinuous Galerkin (DG) schemes for multidimensional ideal compressible magnetohydrodynamics (MHD). 
It is based on the proposed locally divergence-free high-order DG schemes 
for the symmetrizable ideal MHD equations as the base schemes, 
the PP limiting procedure \cite{cheng} to enforce the positivity of the
 DG solutions, and 
  strong stability preserving methods \cite{Gottlieb2009} for 
 time discretization. 
The significant innovation is that we discover and rigorously prove the 
PP property of the proposed DG schemes by using a novel equivalent form of the admissible state set and some very technical estimates. 
There are two features 
in our PP schemes: the locally divergence-free spatial discretization and the penalty-type terms discretized from the GP source term. 
The former leads to zero divergence within each cell, while the latter 
controls the divergence-error across the cell interfaces.  
Our PP analysis have showed that, thanks to these two features, the PP schemes are obtained without requiring the discrete divergence-free condition in \cite{Wu2017a}, 
which was proposed for the conservative schemes without penalty-type terms. 
Several two-dimensional numerical experiments have confirmed 
the theoretical analysis, and demonstrated the accuracy, effectiveness and robustness 
of the proposed PP DG method.

The motivation of designing PP schemes based on the 
symmetrizable ideal MHD equations comes from an important observation: 
in the multidimensional cases, 
the exact solution of the conservative MHD equations \eqref{eq:MHD} 
may fail to preserve the non-negativity of pressure if $\nabla \cdot {\bf B} \neq 0$, 
while it seems that the symmetrizable form \eqref{eq:MHD:GP} 
with an additional non-conservative source term does not suffer from this issue.  
There is still a conflict between
the requirement of the non-conservative source term, and 
the conservation property of numerical schemes which is lost due to
the source term. 
The extension of our PP methods to unstructured meshes 
is straightforward, but the proof of the PP property is 
much more technical and will be studied separately.

%\section*{Acknowledgements}
%This work was partially supported by

\appendix

\section{Negative pressure may appear in the exact solution of conservative MHD system \eqref{eq:MHD} if $\nabla \cdot {\bf B}\neq 0$}\label{app:evidences} This appendix provides the evidence (not rigorous proof) on the claim that negative pressure may appear in the exact smooth solution of the conservative MHD system \eqref{eq:MHD} if $\nabla \cdot {\bf B}\neq 0$. 

Let focus on the ideal EOS \eqref{eq:EOS} and $d=3$. We consider the following initial condition with nonzero divergence
\begin{equation}\label{eq:np:data}
\begin{aligned}
&	\rho({\bf x},0)= 1, \qquad p({\bf x},0)=1-\exp(-|{\bf x}|^2),
\\
& {\bf v} ({\bf x},0)= (1,~1,~1), \qquad {\bf B}({\bf x},0) 
= (1+\delta B_1,~1+\delta B_2,~1+\delta B_3), 
\end{aligned}
\end{equation}
where ${\bf x}=(x_1,x_2,x_3)$ and $\delta B_i=\frac{\epsilon}{3} \arctan x_i$, $i=1,2,3$, are small perturbations 
with $0<\epsilon \ll 1$. 
Since the initial data \eqref{eq:np:data} is bounded and infinitely differentiable, it is reasonable to expect that,    
there exists a small time interval $[0,t_*)$ such that 
the exact solution of the system \eqref{eq:MHD} with the initial condition \eqref{eq:np:data} is smooth for $t \in [0,t_*)$. 
Under this assumption, one can study the initial time derivative of $p$ at ${\bf x}={\bf 0}$, although the analytical expression of the exact solution for $t>0$ is not available. For smooth solutions, it follows from \eqref{eq:MHD} that 
$$
p_t + {\bf v} \cdot \nabla p + \gamma p \nabla \cdot {\bf v} 
+ (\gamma-1) ( {\bf v} \cdot {\bf B} ) \nabla \cdot {\bf B} =0.
$$
At $t=0$ and ${\bf x}={\bf 0}$, one has $\nabla p={\bf 0}$, $\nabla \cdot {\bf v}=0$ and $\nabla \cdot {\bf B}= \epsilon > 0$, which yield  
$$
p_t({\bf 0},0) = - 3(\gamma-1)  \epsilon < 0.
$$ 
Note that $p({\bf 0},0)=0$. %, by the definition of derivative, 
Thus there exists $t_0 \in [0,t_*)$ such that 
$$
p({\bf 0},t) < 0, \quad \forall t \in (0,t_0).
$$ 
We therefore have the reason to think that the non-negativity of pressure is not positively invariant 
for the conservative MHD system \eqref{eq:MHD} if $\nabla \cdot {\bf B}\neq 0$. While the modified MHD system \eqref{eq:MHD:GP} may not 
suffer from this issue, because \eqref{eq:MHD:GP} implies   
$$
p_t + {\bf v} \cdot \nabla p + \gamma p \nabla \cdot {\bf v} =0.
$$
%{\color{blue}Moreover, using the theorem in (\cite{serre2000systems}, p. 12),  
%one can verify that the closure of ${ \mathcal G}$ is a positively invariant domain 
%of the system 
%\begin{equation*}%\label{eq:MHD:GP:perturbed}
%\frac{{\partial {\bf U}}}{{\partial t}} + \sum\limits_{i = 1}^d {\frac{{\partial { {\bf F}_i}(
%			{\bf U})}}{{\partial x_i}}} 
%= -\big(\nabla \cdot {\bf B} \big)~{\bf S} ( {\bf U} )
%+ %\sum_{i=1}^d
%\sum_{1\le i,j\le d} \frac{\partial }{\partial x_i} 
%\bigg( b_{ij} ({\bf U}) \frac{\partial {\bf U}}{\partial x_j} \bigg),
%\end{equation*}
%which is perturbed from \eqref{eq:MHD:GP} with 
%parabolic terms, and $(b_{ij}({\bf U}))$ is any symmetric positive definite matrix of size $d$.}  

\section{Two additional benchmark  tests}\label{app:test}The Orszag-Tang  problem (see e.g., \cite{Li2005}) and rotor problem \cite{BalsaraSpicer1999} are two benchmark tests widely simulated in the literature. Although not extreme, they are also tested by using  
the proposed PP third-order DG method to verify the effectiveness and high-resolution. Fig.\ \ref{fig:standardtests} gives the contour plots of the computed density, which agree well with those in \cite{BalsaraSpicer1999,Li2005}.

	\begin{figure}[htbp]
	\centering
	{\includegraphics[width=0.98\textwidth]{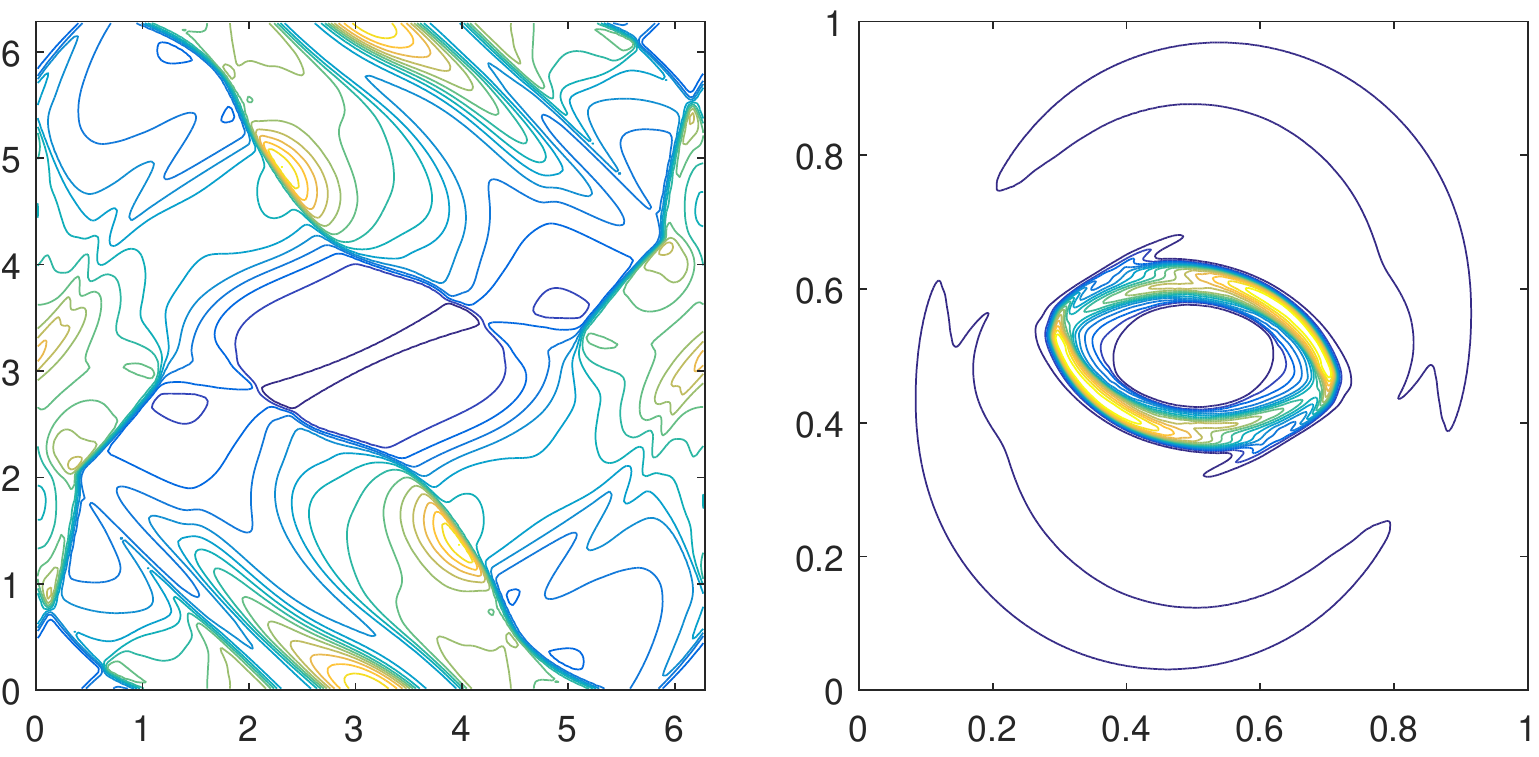}}
	\caption{\small Numerical solutions of the
		Orszag-Tang problem with $200\times 200$ cells at $t=2$ (left) and the rotor problem with $400\times 400$ cells at $t=0.295$ (right).} 
	\label{fig:standardtests}
\end{figure}

\bibliographystyle{siamplain}
\bibliography{references}

\end{document}